\documentclass{amsart}

\usepackage{amsmath,amssymb,amsthm,amsfonts,mathrsfs}
\usepackage{amscd,stmaryrd}
\usepackage{mathtools}
\usepackage[usenames,dvipsnames]{color}
\usepackage{hyperref}
\usepackage{graphicx,subfigure}
\definecolor{nicegreen}{RGB}{0,180,0}
\hypersetup{colorlinks=true,citecolor=nicegreen,linkcolor=Red,urlcolor=blue,pdfstartview=FitH}
\usepackage[divide={2.45cm,*,2.45cm}]{geometry}
\usepackage{enumerate}
\usepackage{color}
\usepackage[all]{xy}
\usepackage{tikz-cd}
\usepackage{wasysym}

\allowdisplaybreaks

\newtheorem{thm}{Theorem}[section]
\newtheorem*{thm*}{Theorem}
\newtheorem{cor}[thm]{Corollary}
\newtheorem{corpf}[thm]{Corollary (of proof)}

\newtheorem{lemma}[thm]{Lemma}

\newtheorem{propn}[thm]{Proposition}
\newtheorem{conj}[thm]{Conjecture}
\newtheorem*{propn*}{Proposition}
\newtheorem*{claim}{Claim}
\newtheorem{step}{Step}

\theoremstyle{definition}
\newtheorem{defn}[thm]{Definition}
\newtheorem{rmk}[thm]{Remark}

\newcommand{\T}{\textnormal{T}}

\newcommand{\triv}{\textnormal{triv}}
\newcommand{\sign}{\textnormal{sign}}

\newcommand{\soc}{\textnormal{soc}}
\newcommand{\Hom}{\textnormal{Hom}}

\newcommand{\Ind}{\textnormal{Ind}}
\newcommand{\cind}{\textnormal{c-ind}}
\newcommand{\Ext}{\textnormal{Ext}}

\newcommand{\aff}{\textnormal{aff}}

\newcommand{\inj}{\textnormal{inj}}
\newcommand{\cores}{\textnormal{cor}}
\newcommand{\res}{\textnormal{res}}
\newcommand{\ord}{\textnormal{Ord}}

\newcommand*{\longhookrightarrow}{\ensuremath{\lhook\joinrel\relbar\joinrel\rightarrow}}
\newcommand*{\longtwoheadrightarrow}{\ensuremath{\relbar\joinrel\twoheadrightarrow}}

\newcommand{\sm}[4]{\left(\begin{smallmatrix} #1 & #2 \\ #3 & #4 \end{smallmatrix}\right)}

\newcommand{\cE}{{\mathcal{E}}}

\newcommand{\cH}{{\mathcal{H}}}

\newcommand{\cJ}{{\mathcal{J}}}
\newcommand{\cK}{{\mathcal{K}}}

\newcommand{\cO}{{\mathcal{O}}}

\newcommand{\cR}{{\mathcal{R}}}

\newcommand{\cZ}{{\mathcal{Z}}}

\newcommand{\bB}{{\mathbf{B}}}

\newcommand{\bG}{{\mathbf{G}}}
\newcommand{\bH}{{\mathbf{H}}}

\newcommand{\bL}{{\mathbf{L}}}
\newcommand{\bM}{{\mathbf{M}}}
\newcommand{\bN}{{\mathbf{N}}}

\newcommand{\bP}{{\mathbf{P}}}

\newcommand{\bS}{{\mathbf{S}}}
\newcommand{\bT}{{\mathbf{T}}}
\newcommand{\bU}{{\mathbf{U}}}

\newcommand{\bZ}{{\mathbf{Z}}}

\newcommand{\bbC}{{\mathbb{C}}}

\newcommand{\bbF}{{\mathbb{F}}}

\newcommand{\bbQ}{{\mathbb{Q}}}

\newcommand{\bbZ}{{\mathbb{Z}}}

\newcommand{\sA}{{\mathscr{A}}}
\newcommand{\sB}{{\mathscr{B}}}

\newcommand{\sI}{{\mathscr{I}}}

\newcommand{\fA}{{\mathfrak{A}}}
\newcommand{\fB}{{\mathfrak{B}}}
\newcommand{\fC}{{\mathfrak{C}}}
\newcommand{\fD}{{\mathfrak{D}}}
\newcommand{\fE}{{\mathfrak{E}}}
\newcommand{\fF}{{\mathfrak{F}}}

\newcommand{\fI}{{\mathfrak{I}}}

\newcommand{\fe}{{\mathfrak{e}}}

\newcommand{\ffi}{{\mathfrak{i}}}

\newcommand{\fm}{{\mathfrak{m}}}
\newcommand{\fn}{{\mathfrak{n}}}

\newcommand{\fp}{{\mathfrak{p}}}

\newcommand{\fv}{{\mathfrak{v}}}

\begin{document}
\nocite{}

\title{Functorial properties of pro-$p$-Iwahori cohomology}
\date{}
\author{Karol Kozio\l}
\address{Department of Mathematics, University of Michigan, 2074 East Hall, 530 Church Street, Ann Arbor, MI 48109-1043} \email{kkoziol@umich.edu}

\subjclass[2010]{20C08 (primary), 20J06, 22E50, 22D35 (secondary)}

\begin{abstract}
Suppose $F$ is a finite extension of $\mathbb{Q}_p$, $G$ is the group of $F$-points of a connected reductive $F$-group, and $I_1$ is a pro-$p$-Iwahori subgroup of $G$.  We construct two spectral sequences relating derived functors on mod-$p$ representations of $G$ to the analogous functors on Hecke modules coming from pro-$p$-Iwahori cohomology.  More specifically: (1) using results of Ollivier--Vign\'eras, we provide a link between the right adjoint of parabolic induction on pro-$p$-Iwahori cohomology and Emerton's functors of derived ordinary parts; and (2) we establish a ``Poincar\'e duality spectral sequence'' relating duality on pro-$p$-Iwahori cohomology to Kohlhaase's functors of higher smooth duals.  As applications, we calculate various examples of the Hecke modules $\textnormal{H}^i(I_1,\pi)$.  
\end{abstract}

\maketitle

\section{Introduction}

Functorial constructions abound in the representation theory of $p$-adic reductive groups.  For example, when working with smooth representations over $\bbC$, supercuspidal representations may be characterized as those annihilated by all Jacquet functors, while duality functors arise naturally in functional equations of automorphic $L$-functions (in the form of contragredient representations).  Such functors give the category of smooth representations a rich and intricate structure.

Motivated by recent advances in the mod-$p$ and $p$-adic Langlands programs (cf. \cite{ahhv}, \cite{breuilpaskunas}, \cite{ceggps}, \cite{colmez:gl2}, \cite{paskunas:montrealfunctor}, \cite{scholze:LT}, \cite{Ast319}, \cite{Ast330}, \cite{Ast331}), we would like to understand similar constructions when dealing with representations over coefficient fields of characteristic $p$.  Unfortunately, most of the analogously defined functors fail to be exact.  It is therefore natural (and indeed necessary) to consider their derived versions.

In order to state our setup more precisely, we introduce some notation.  Let $\bG$ denote a connected reductive group defined over a finite extension $F$ of $\bbQ_p$, and denote by $G$ its group of $F$-points.  We also let $I_1$ denote a choice of pro-$p$-Iwahori subgroup of $G$, which we assume for this introduction to be torsion-free.  All representations and modules appearing in this article will have coefficients in a finite field $C$ of characteristic $p$.  Our starting point is the following theorem of Schneider \cite{schneider:dga}: there exists an equivalence of triangulated categories
$$D\big(\mathfrak{Rep}^\infty(G)\big) \stackrel{\sim}{\longrightarrow} D\big(\mathfrak{dgMod-}\cH^\bullet\big),$$
where $\mathfrak{Rep}^\infty(G)$ denotes the category of smooth (mod-$p$) representations of $G$, $\cH^\bullet$ denotes the differential graded Hecke algebra of $G$ with respect to $I_1$, and $\mathfrak{dgMod-}\cH^\bullet$ denotes the category of differential graded modules over $\cH^\bullet$.  The equivalence is given by sending a complex $\pi^\bullet$ of smooth representations to the complex $\textnormal{R}\textnormal{H}^0(I_1,\pi^\bullet)$ of derived $I_1$-invariants, which naturally comes equipped with an action of the Hecke DGA $\cH^\bullet$.

Using the derived equivalence above, one would hope that various functors on the category $D(\mathfrak{Rep}^\infty(G))$ could be transferred to functors on the category $D(\mathfrak{dgMod-}\cH^\bullet)$.  The purpose of this article is to address several versions of this expectation.  Namely, we simplify the setup somewhat, and consider the cohomology of the above derived equivalence, along with derived functors on both sides.  Thus, given a smooth $G$-representation $\pi$, we consider its pro-$p$-Iwahori cohomology spaces $\textnormal{H}^i(I_1,\pi)$.  These spaces come equipped with a right action of $\cH$, the pro-$p$-Iwahori--Hecke algebra of $I_1$-bi-invariant, compactly supported, $C$-valued functions on $G$.  (The algebra $\cH$ is related to the DGA $\cH^\bullet$ via the isomorphism $h^0(\cH^\bullet) \cong \cH$.)  Our goal is to investigate how various derived functors on $\pi$ are related to derived functors on $\textnormal{H}^i(I_1,\pi)$.  

\vspace{10pt}

We now describe the contents of this article.  The first situation we consider in Section \ref{sec:ordparts} is that of passing from $G$ to a Levi subgroup.  Suppose $P = M \ltimes N$ is a rational parabolic subgroup of $G$, with rational Levi subgroup $M$ (chosen to be ``compatible'' with $I_1$), and recall that we have the functor of parabolic induction
$$\Ind_P^G: \mathfrak{Rep}^{\infty}(M) \longrightarrow \mathfrak{Rep}^{\infty}(G).$$
This functor is exact, and by \cite{vigneras:rightadj}, it admits a right adjoint
$$\cR_P^G: \mathfrak{Rep}^{\infty}(G) \longrightarrow \mathfrak{Rep}^{\infty}(M).$$
On the side of $\cH$-modules, we also have a parabolic induction functor
$$\Ind_{\cH_M}^{\cH}: \mathfrak{Mod-}\cH_M \longrightarrow \mathfrak{Mod-}\cH,$$
where $\cH_M$ denotes the pro-$p$-Iwahori--Hecke algebra of $M$ with respect to $I_{M,1} := I_1 \cap M$.  As above, this parabolic induction functor admits a right adjoint
$$\cR_{\cH_M}^{\cH}: \mathfrak{Mod-}\cH \longrightarrow \mathfrak{Mod-}\cH_M.$$
Work of Ollivier--Vign\'eras in \cite{olliviervigneras} shows that the following diagram commutes up to natural equivalence:
\begin{center}
\begin{tikzcd}[column sep = large, row sep = large]
\mathfrak{Rep}^{\infty}(G) \ar[r, "\cR^G_{P}"] \ar[d, "\textnormal{H}^0(I_1{,}-)"'] & \mathfrak{Rep}^{\infty}(M) \ar[d, "\textnormal{H}^0(I_{M,1}{,}-)"] \\
\mathfrak{Mod-}\cH \ar[r, "\cR_{\cH_M}^{\cH}"] & \mathfrak{Mod-}\cH_M
\end{tikzcd}
\end{center}
Our first result computes the derived version of this diagram.

\begin{thm}
Suppose $\pi \in \mathfrak{Rep}^\infty(G)$ is an admissible representation.  Then we have an $E_2$ spectral sequence of $\cH_M$-modules
$$\textnormal{H}^i\big(I_{M,1},\textnormal{R}^j\cR_P^G(\pi)\big) \Longrightarrow \cR^{\cH}_{\cH_M}\big(\textnormal{H}^{i + j}(I_1, \pi)\big).$$
\end{thm}

The proof of the above result follows from some standard manipulations with derived functors, along with the explicit description of the functor $\cR_{\cH_M}^\cH$.

While the above theorem has the advantage of being valid for arbitrary reductive groups, very little is known about the functors $\textnormal{R}^j\cR_P^G$ (to the author, at least).  Indeed, even the construction of the underived functor $\cR^G_P$ is non-explicit, and follows from abstract category theoretic existence arguments.  In order to actually compute the above spectral sequence, we  employ the following result, due to Abe--Henniart--Vign\'eras \cite{AHV}:  if $\pi$ is an admissible $G$-representation, then we have
$$\cR_P^G(\pi) \cong \ord_{P^-}^G(\pi),$$
where $P^-$ denotes the parabolic subgroup opposite to $P$, and $\ord_{P^-}^G$ denotes Emerton's functor of ordinary parts, constructed in \cite{emerton:ordI}.  The latter functor has the advantage of having an explicit description (see Subsubsection \ref{smooth reps} below).  Moreover, if $\pi$ is an admissible $G$-representation, then $\ord_{P^-}^G(\pi)$ is an admissible $M$-representation.  Thus, the above commutative diagram becomes
\begin{center}
\begin{tikzcd}[column sep = large, row sep = large]
\mathfrak{Rep}^{\textnormal{adm}}(G) \ar[r, "\ord^G_{P^-}"] \ar[d, "\textnormal{H}^0(I_1{,}-)"'] & \mathfrak{Rep}^{\textnormal{adm}}(M) \ar[d, "\textnormal{H}^0(I_{M,1}{,}-)"] \\
\mathfrak{Mod-}\cH \ar[r, "\cR_{\cH_M}^{\cH}"] & \mathfrak{Mod-}\cH_M
\end{tikzcd}
\end{center}

We would like to compute some sort of derived version of the above diagram.  There are several difficulties that arise, however.  For one, the category $\mathfrak{Rep}^{\textnormal{adm}}(G)$ of admissible $G$-representations does not have enough injectives.  We must therefore pass to the category $\mathfrak{Rep}^{\textnormal{ladm}}(G)$ of \emph{locally} admissible $G$-representations, which does have enough injectives.  This must be done very carefully, as it is not known whether $\cR_P^G$ agrees with $\ord_{P^-}^G$ on $\mathfrak{Rep}^{\textnormal{ladm}}(G)$.  Nevertheless, imposing some restrictions, we prove the following theorem.

\begin{thm}
Suppose the simply connected cover of the derived subgroup of $\bG$ is $\bS\bL_{2/F}$, and let $\pi \in \mathfrak{Rep}^{\textnormal{adm}}(G)$ denote an admissible representation which has a central character.  Then we have an $E_2$ spectral sequence of $\cH_M$-modules
$$\textnormal{H}^i\big(I_{M,1},\textnormal{R}^j\ord_{P^-}^G(\pi)\big) \Longrightarrow \cR^{\cH}_{\cH_M}\big(\textnormal{H}^{i + j}(I_1,\pi)\big),$$
where $\textnormal{R}^j\ord_{P^-}^G$ is computed in the category $\mathfrak{Rep}^{\textnormal{ladm}}(G)$.  
\end{thm}

With the hypotheses of the above theorem, the derived functor $\textnormal{R}^j\ord_{P^-}^G$ agrees with the $\delta$-functor $\textnormal{H}^j\ord_{P^-}^G$ defined by Emerton in \cite{emerton:ordII}.  In particular, they may be explicitly computed using results of Emerton and Hauseux (\emph{op. cit.} and \cite{hauseux:parabind}, respectively).

The proof of the above theorem is given for a general reductive group, granting the validity of three interconnected conjectures.  These conjectures in particular guarantee that $\textnormal{R}^j\ord_{P^-}^G \simeq \textnormal{H}^j\ord_{P^-}^G$, and that we may compute $\textnormal{R}^j\ord_{P^-}^G(\pi)$ using certain resolutions in the category $\mathfrak{Rep}^{\textnormal{adm}}(G)$.  In Appendix \ref{app}, we prove the three conjectures for several classes of groups (and in particular for those mentioned in the theorem) using results of Emerton and Pa\v{s}k\={u}nas.  Verifying these conjectures for higher rank groups will likely require new ideas.

The second functor we consider is that of passing to a dual representation.  It is well known that for $G$-representations over $\bbC$, the process of taking smooth vectors in the $\bbC$-linear dual is an exact functor, which preserves admissibility and irreducibility.  On the contrary, when $\pi$ is a $G$-representation over a field of characteristic $p$, the functor
$$\pi \longmapsto S^0(\pi) := (\pi^\vee)^\infty$$
is quite poorly behaved; in particular, it is no longer exact, and annihilates irreducible, admissible, infinite-dimensional representations.  To remedy this, Kohlhaase in \cite{kohlhaase:duality} introduced a contravariant $\delta$-functor of ``higher smooth duals''
$$S^i: \mathfrak{Rep}^{\infty}(G) \longrightarrow \mathfrak{Rep}^{\infty}(G)$$
for $0 \leq i \leq d := \dim_{\bbQ_p}(G)$, extending the functor $S^0$.  Taken together, these functors are much better behaved, and give a satisfactory duality theory on $\mathfrak{Rep}^{\textnormal{adm}}(G)$ (see \emph{op. cit.}, Corollary 3.15).

In Section \ref{sec:duality}, we compare Kohlhaase's $\delta$-functor $(S^i)_{i \geq 0}$ with duality for the functors $\textnormal{H}^i(I_1,-)$ of pro-$p$-Iwahori cohomology.  This is achieved by constructing a certain double complex which incorporates both functors, and examining the two associated spectral sequences using results of Symonds--Weigel \cite{symondsweigel} on the cohomology of $p$-adic analytic groups.  As a result, we obtain the following ``Poincar\'e duality spectral sequence.''

\begin{thm}
Suppose $I_1$ is torsion-free, and $\pi \in \mathfrak{Rep}^\infty(G)$ is a smooth representation.  We then have an $E_2$ spectral sequence of $\cH$-modules
$$\textnormal{H}^i\big( I_1,S^j(\pi)\big) \Longrightarrow \textnormal{H}^{d - i - j}(I_1,\pi)^\vee(\xi),$$
where $\textnormal{H}^{d - i - j}(I_1,\pi)^\vee$ denotes the $\cH$-module whose underlying vector space is the linear dual of $\textnormal{H}^{d - i - j}(I_1,\pi)$, and $(\xi)$ denotes the twist by a certain orientation character.  
\end{thm}

This spectral sequence may be thought of as an $\cH$-equivariant version of the Tate spectral sequence (cf. \cite[Ch. 1, App. 1, Cor.]{serre:galoiscoh}).  Further, the theorem above partially answers a question posed by Harris in \cite[Ques. 4.5]{harris:specs}.  We also note that there is related work of Sorensen relating Kohlhaase's smooth duals to a duality operation on the category $D(\mathfrak{dgMod-}\cH^\bullet)$, assuming the group $G$ is compact.

To conclude, we give in Section \ref{sec:examples} several example calculations using the above spectral sequences.  In particular, using results of Emerton, Kohlhaase, Pa\v{s}k\={u}nas and others, we obtain the following:
\begin{itemize}
\item When $G = \textnormal{GL}_2(\bbQ_p)$ or $G = \textnormal{SL}_2(\bbQ_p)$ with $p \geq 5$, we are able to compute (essentially) all of the cohomology spaces $\textnormal{H}^i(I_1,\pi)$, where $\pi$ is an absolutely irreducible $G$-representation.  
\item When $G = \textnormal{GL}_3(\bbQ_p)$ and $\textnormal{St}_{\textnormal{GL}_3}$ denotes the Steinberg representation of $G$, we compute some of the spaces $\textnormal{H}^i(I_1,\textnormal{St}_{\textnormal{GL}_3})$.  We also are able to deduce a structural result about the representations $S^j(\textnormal{St}_{\textnormal{GL}_3})$.
\item When $G$ is a general reductive group and $\mathbf{1}_G$ denotes the trivial $G$-representation, we obtain a Poincar\'e duality isomorphism of $\cH$-modules
$$\textnormal{H}^i(I_1, \mathbf{1}_G) \cong \textnormal{H}^{d - i}(I_1, \mathbf{1}_G)^\vee(\xi).$$
When $\pi = \Ind_P^G(\chi)$ is the parabolic induction of a character, we have an analogous isomorphism
$$\textnormal{H}^i\big(I_1, \Ind_P^G(\chi)\big) \cong \textnormal{H}^{\dim_{\bbQ_p}(P) - i}\big(I_1, \Ind_P^G(\chi') \big)^\vee(\xi),$$
where $\chi'$ denotes the dual character of $\chi$ twisted by a certain modulus character.
\end{itemize}

\vspace{10pt}

We have decided to work with spectral sequences and derived functors in this article for the sake of explicitly computing the cohomology spaces $\textnormal{H}^i(I_1,\pi)$.  In the spirit of Schneider's derived equivalence, it should be possible to upgrade these results to the level of total derived functors compatible with DGA structures.  We hope to return to this in future work.

\vspace{10pt}

\noindent \textbf{Acknowledgements.}  I would like to thank Noriyuki Abe for his help with an argument below, and Claus Sorensen for several useful conversations.  I would especially like to thank Peter Schneider, for identifying a consequential error in an earlier draft and suggesting a correction.  Additionally, I would like to thank the referee for helpful comments.  During the preparation of parts of this article, funding was provided by Manish Patnaik's Subbarao Professorship and NSERC grant RGPIN 4622.

\vspace{10pt}

\section{Notation}

\subsection{Basic notation}

Let $p$ denote a prime number, and suppose $F$ is a finite extension of $\bbQ_p$.  For an algebraic $F$-group $\bH$, we let $H := \bH(F)$ denote its group of $F$-points.  Throughout the article we will consider a connected reductive $F$-group $\bG$ and its group of $F$-points $G$.  Let $\boldsymbol{\cZ}$ denote the connected center of $\bG$, and let $\widetilde{\cZ}_0$ denote the maximal compact subgroup of $\cZ$.  The group $\cZ/\widetilde{\cZ}_0$ is free of finite rank, say $s$, and choosing a splitting of the surjection $\cZ \longtwoheadrightarrow \bbZ^{\oplus s}$ gives a set $\{z_1, \ldots, z_s\}$ of central elements in $G$.

Let $\bS$ denote a fixed maximal $F$-split torus of $\bG$, and $\bZ$ its centralizer.  We fix a minimal $F$-parabolic subgroup $\bB$ containing $\bZ$, and denote by $\bU$ its unipotent radical, so that $\bB = \bZ \ltimes \bU$.  More generally, by a standard parabolic subgroup $\bP$ we will mean any $F$-parabolic subgroup containing $\bB$.  We write $\bP = \bM \ltimes \bN$, where $\bN$ is the unipotent radical of $\bP$, and $\bM$ is its Levi component.  It will be assumed that all Levi components $\bM$ contain $\bZ$.  We write $\bP^- = \bM \ltimes \bN^-$ for the opposite parabolic subgroup.

\vspace{5pt}

\subsection{Representations}

We let $C$ denote a finite field of characteristic $p$, which will serve as the field of coefficients for all representations and modules appearing.  We let $\mathfrak{Rep}^\infty(G)$ denote the category of smooth $G$-representations.  Further, we denote by $\mathfrak{Rep}^{\textnormal{adm}}(G)$ (resp. $\mathfrak{Rep}^{\textnormal{ladm}}(G)$) the full subcategory of $\mathfrak{Rep}^\infty(G)$ consisting of (locally) admissible representations.  All three categories are abelian, and the categories $\mathfrak{Rep}^\infty(G)$ and $\mathfrak{Rep}^{\textnormal{ladm}}(G)$ have enough injectives (see \cite[Lem. 2.2.6, Props. 2.2.13, 2.2.18]{emerton:ordI} and \cite[Prop. 2.1.1]{emerton:ordII}).  Finally, for $c_1, \ldots, c_s \in C^\times$, we let $\mathfrak{Rep}^{\textnormal{adm}}_{z_i = c_i}(G)$ denote the abelian subcategory of $\mathfrak{Rep}^{\textnormal{adm}}(G)$ consisting of representations on which the $z_i$ act by the scalars $c_i$.

\vspace{5pt}

\subsection{Hecke algebras}
\label{subsec:heckealgs}

We fix a chamber of the semisimple Bruhat--Tits building contained in the apartment corresponding to $\bS$, and let $I$ denote the corresponding Iwahori subgroup and $I_1$ its pro-$p$-radical.  (We make no assumptions about torsion-freeness of $I_1$, until indicated otherwise.)  We then define the pro-$p$-Iwahori--Hecke algebra $\cH$ to be the convolution algebra of $C$-valued, compactly supported, $I_1$-bi-invariant functions on $G$.  For $g\in G$, we let $\T_g$ denote the characteristic function of $I_1gI_1$.  The algebra $\cH$ has a distinguished set of generators (given by those $\T_g$ where $g$ is a lift of a generator of the extended pro-$p$ affine Weyl group of $G$) which satisfy braid relations and quadratic relations, but we shall not need this description.  (See \cite[\S 4]{vigneras:hecke1} for details.)

Let $\mathfrak{Mod-}\cH$ denote the category of right $\cH$-modules.  We will consider objects of $\mathfrak{Mod-}\cH$ coming from the following construction.  Given a smooth $G$-representation $\pi$, the space of $I_1$-invariants $\pi^{I_1}$ has a right action of $\cH$, recalled in \cite[pf. of Lem. 4.5]{olliviervigneras}.  We therefore obtain a functor 
\begin{eqnarray*}
\mathfrak{Rep}^{\infty}(G) & \longrightarrow & \mathfrak{Mod-}\cH \\
\pi & \longmapsto & \pi^{I_1}
\end{eqnarray*}  
Since $I_1$ is open in $G$, the functor $\cind_{I_1}^G$ of compact induction is exact, and therefore its right adjoint $\res^G_{I_1}: \mathfrak{Rep}^{\infty}(G) \longrightarrow \mathfrak{Rep}^{\infty}(I_1)$ preserves injectives.  Thus, we see that the derived functors of the above may be identified with the cohomology spaces $\textnormal{H}^i(I_1,\pi)$, equipped with a right action of the algebra $\cH$.  (We will make this abuse of notation going forward.)  Explicitly, the operator $\T_g$ acts on $v\in \textnormal{H}^i(I_1,\pi)$ by
$$v \cdot \T_g = \left( \cores^{I_1}_{I_1 \cap g^{-1}I_1g} \circ g^{-1}_* \circ \res^{I_1}_{I_1 \cap gI_1g^{-1}} \right)(v).$$

We shall also need a duality operation on Hecke modules.  Given a right $\cH$-module $\fm$, we set $\fm^\vee := \Hom_C(\fm, C)$, and define a right action of $\cH$ on $\fm^\vee$ by
$$(f\cdot \T_g)(m) = f(m\cdot \T_{g^{-1}}),$$
where $f \in \fm^\vee$ and $m \in \fm$.  The fact that this gives a well defined right action follows from \cite[Lemma 2.11]{vigneras:hecke5}.

\vspace{10pt}

\section{Ordinary parts}
\label{sec:ordparts}

\subsection{Preliminaries}

\subsubsection{Smooth representations}
\label{smooth reps}

Let $\bP = \bM \ltimes \bN$ denote a standard parabolic subgroup of $\bG$.  Recall that we have the smooth parabolic induction functor 
$$\Ind_P^G: \quad \mathfrak{Rep}^\infty(M) \longrightarrow \mathfrak{Rep}^\infty(G),$$
which is exact and fully faithful (\cite[Prop. 4.2, Thm. 5.3]{vigneras:rightadj}).  Using \cite[Prop. 4.1.7]{emerton:ordI}, we see that this functor restricts to exact functors
\begin{alignat*}{3}
\Ind_P^G: \quad & \mathfrak{Rep}^{\textnormal{ladm}}(M) ~ & \longrightarrow ~ & \mathfrak{Rep}^{\textnormal{ladm}}(G), \\
\Ind_P^G: \quad & \mathfrak{Rep}^{\textnormal{adm}}(M) ~ & \longrightarrow ~ & \mathfrak{Rep}^{\textnormal{adm}}(G).
\end{alignat*}
By \cite[Prop. 4.2]{vigneras:rightadj}, the first of the above functors admits a right adjoint
$$\cR^G_P: \quad \mathfrak{Rep}^\infty(G) \longrightarrow \mathfrak{Rep}^\infty(M).$$

On the other hand, Emerton in \cite{emerton:ordI} has defined the functors of ordinary parts 
\begin{alignat*}{3}
\ord_{P^-}^G: \quad & \mathfrak{Rep}^{\textnormal{ladm}}(G) ~ & \longrightarrow ~ & \mathfrak{Rep}^{\textnormal{ladm}}(M), \\
\ord_{P^-}^G: \quad & \mathfrak{Rep}^{\textnormal{adm}}(G) ~ & \longrightarrow ~ & \mathfrak{Rep}^{\textnormal{adm}}(M),
\end{alignat*}
which are right adjoint to $\Ind_P^G$.  We recall the construction of $\ord_{P^-}^G(\pi)$ for $\pi \in \mathfrak{Rep}^{\textnormal{ladm}}(G)$.  Let $M^- \subset M$ denote the submonoid of elements $m$ such that $mN_0^- m^{-1} \subset N_0^-$, where $N_0^- := I \cap N^-$.  We let $Z_M$ denote the center of $M$, set $Z_M^- := M^- \cap Z_M$, and let $z \in Z_M^-$ denote an element such that $Z_M$ is generated (as a monoid) by $Z_M^-$ and $z^{-1}$.  As vector spaces, we then have
$$\ord_{P^-}^G(\pi) = C[z^{\pm 1}] \otimes_{C[z]} \pi^{N_0^-},$$
where $z$ acts on $\pi^{N^-_0}$ via the Hecke action 
$$\pi^{N_0^-}  \stackrel{z\cdot}{\longrightarrow} \pi^{zN_0^-z^{-1}} \xrightarrow{\cores_{zN_0^-z^{-1}}^{N_0^-}} \pi^{N_0^-}$$
(cf. \cite[Def. 3.1.3]{emerton:ordI}).  Given $m \in M$, we may write $m = z^{i_m}m^-$ for some $i_m \in \bbZ$ and $m^- \in M^-$.  The action of $m$ on $z^i \otimes v \in \ord_{P^-}^G(\pi)$ is then given by
$$m.(z^i \otimes v) = z^{i + i_m} \otimes m^-\cdot v,$$
where we use the Hecke action in the second tensor factor.  Up to isomorphism, the $M$-representation $\ord_{P^-}^G(\pi)$ is independent of the choice of compact open subgroup $N_0^-$ of $N^-$ (\cite[Prop. 3.1.12]{emerton:ordI}).

The functor $\ord_{P^-}^G: \mathfrak{Rep}^{\textnormal{adm}}(G)  \longrightarrow  \mathfrak{Rep}^{\textnormal{adm}}(M)$ is right adjoint to the admissible parabolic induction functor $\Ind_P^G: \mathfrak{Rep}^{\textnormal{adm}}(M)  \longrightarrow  \mathfrak{Rep}^{\textnormal{adm}}(G)$.  Theorem 4.11 of \cite{AHV} shows that $\cR^G_P$ preserves admissibility, and is therefore also right adjoint to $\Ind_P^G: \mathfrak{Rep}^{\textnormal{adm}}(M)  \longrightarrow  \mathfrak{Rep}^{\textnormal{adm}}(G)$.  Consequently, we have $\ord_{P^-}^G(\pi) \cong \cR_P^G(\pi)$ for every $\pi \in \mathfrak{Rep}^{\textnormal{adm}}(G)$.

Replacing $\pi^{N_0^-}$ above with $\textnormal{H}^i(N_0^-,\pi)$ gives the definition of the $\delta$-functor $\textnormal{H}^i\ord_{P^-}^G:  \mathfrak{Rep}^{\textnormal{adm}}(G)  \longrightarrow  \mathfrak{Rep}^{\textnormal{adm}}(M)$ (and similarly with ``adm'' replaced by ``ladm''; see \cite[Def. 3.3.1, Lem. 3.2.1, Thm. 3.4.7]{emerton:ordII}).

\vspace{5pt}

\subsubsection{Hecke modules}
\label{hecke mods}

We have analogous functors on the side of Hecke modules.  By \cite[\S 2.4.1, 2.4.2]{olliviervigneras}, the group $I_{M,1} := I_1 \cap M$ is a pro-$p$-Iwahori subgroup of $M$, and therefore we may form the analogous pro-$p$-Iwahori--Hecke algebra $\cH_M$.  Its Hecke operators will be denoted $\T_m^M$ for $m \in M$.  Note that $\cH_M$ is not a subalgebra of $\cH$ in general.

Define the monoid
$$M^+ := \{m \in M: mN_0m^{-1} \subset N_0 ~~\textnormal{and}~~ m^{-1}N_0^-m \subset N_0^-\},$$  
where $N_0 := I \cap N$.  (Thus $(M^+)^{-1} \subset M^-$.) The set of functions in $\cH_M$ with support in $M^+$ forms a subalgebra, denoted $\cH_M^+$, and is called the positive subalgebra.  The algebra $\cH_M^+$ admits an embedding $\theta$ into $\cH$, given by sending $\T^M_m$ to $\T_m$ for $m \in M^+$.  Thus, given any $\cH_M$-module $\fn$, we may define an $\cH$-module by
$$\Ind_{\cH_M}^{\cH}(\fn) := \fn \otimes_{\cH_M^+,\theta} \cH.$$
The functor $\Ind_{\cH_M}^{\cH}: \mathfrak{Mod-}\cH_M \longrightarrow \mathfrak{Mod-}\cH$ is exact and fully faithful (\cite[Prop. 4.1]{vigneras:hecke5} and \cite[Lem. 5.2]{abe:inductions}).

The right adjoint of $\Ind_{\cH_M}^{\cH}$ is the functor $\cR^{\cH}_{\cH_M}: \mathfrak{Mod-}\cH \longrightarrow \mathfrak{Mod-}\cH_M$ given by
$$\cR_{\cH_M}^{\cH}(\fm) = \Hom_{\cH_M^+,\theta}(\cH_M,\fm)$$
for $\fm \in \mathfrak{Mod-}\cH$.  The right action of $\cH_M$ on the above space is the evident one, and we view $\fm$ as an $\cH_M^+$-module via the embedding $\theta: \cH_M^+ \longhookrightarrow \cH$ above.

We utilize another description of $\cR_{\cH_M}^{\cH}$, which will be useful in the derived context.  By \cite[Thm. 1.4(ii)]{vigneras:hecke5}, the algebra $\cH_M$ is the localization of $\cH_M^+$ at $\T^M_{z^{-1}}$ (with $z$ as in the definition of $\ord^G_{P^-}$).  Therefore, we have an isomorphism of $C$-vector spaces
\begin{eqnarray*}
 \Hom_{\cH_M^+,\theta}(\cH_M,\fm) & \stackrel{\sim}{\longrightarrow} & \varprojlim_{v\mapsto v\cdot\T_{z^{-1}}}\fm\\
f & \longmapsto & \left(f((\T_{z^{-1}}^M)^{-n})\right)_{n\geq0}
\end{eqnarray*}
Let $\cO: \mathfrak{Mod-}\cH_M \longrightarrow C\mathfrak{-Vec}$ denote the forgetful functor.  Since $\cO$ is exact, the Grothendieck spectral sequence associated to the composition $\cO \circ \cR_{\cH_M}^{\cH} \simeq \varprojlim_{v\mapsto v\cdot\T_{z^{-1}}}$ collapses to give 
$$\cO \circ \textnormal{R}^i\cR_{\cH_M}^{\cH} \simeq \sideset{}{^i}\varprojlim_{v\mapsto v\cdot\T_{z^{-1}}}.$$
In particular, if $\fm$ is finite-dimensional over $C$, then the tower $(\fm)_{n\geq 0}$ satisfies the Mittag-Leffler condition.  Thus, we obtain $\varprojlim^i_{v\mapsto v\cdot\T_{z^{-1}}}\fm = 0$ for all $i > 0$, and consequently $\textnormal{R}^i\cR_{\cH_M}^{\cH}(\fm) = 0$ for all $i > 0$.

\vspace{5pt}

\subsubsection{Comparison}

The right adjoint functors $\ord^G_{P^-}$, $\cR^G_P$ and $\cR^{\cH}_{\cH_M}$ are related as follows.  By \cite[\S 4.3.3, Question 5]{olliviervigneras} the diagram
\begin{equation}
\label{smoothcommdiag}
\begin{tikzcd}[column sep = large, row sep = large]
\mathfrak{Rep}^\infty(G) \ar[r, "\cR^G_P"] \ar[d, "(-)^{I_1}"'] & \mathfrak{Rep}^\infty(M) \ar[d, "(-)^{I_{M,1}}"] \\
\mathfrak{Mod-}\cH \ar[r, "\cR_{\cH_M}^{\cH}"] & \mathfrak{Mod-}\cH_M
\end{tikzcd}
\end{equation}
commutes up to natural equivalence.  Similarly, since $\ord^G_{P^-} \simeq \cR^G_{P}$ on the category $\mathfrak{Rep}^{\textnormal{adm}}(G)$, we also have the following commutative diagram
\begin{equation}
\label{admcommdiag}
\begin{tikzcd}[column sep = large, row sep = large]
\mathfrak{Rep}^{\textnormal{adm}}(G) \ar[r, "\ord^G_{P^-}"] \ar[d, "(-)^{I_1}"'] & \mathfrak{Rep}^{\textnormal{adm}}(M) \ar[d, "(-)^{I_{M,1}}"] \\
\mathfrak{Mod-}\cH \ar[r, "\cR_{\cH_M}^{\cH}"] & \mathfrak{Mod-}\cH_M
\end{tikzcd}
\end{equation}
Our goal will be to compute the derived versions of the above diagrams.  Note that some care must be taken here, as the category $\mathfrak{Rep}^{\textnormal{adm}}(G)$ does not have enough injectives.

\vspace{5pt}

\subsection{The admissible case}

We begin by exploring the case of the ordinary parts functor.  In what follows, our arguments will rely on the following three conjectures.

\begin{conj}
\label{conjA}
Let $c_1, \ldots, c_s \in C^\times$, and let $\pi \in \mathfrak{Rep}^{\textnormal{adm}}_{z_i = c_i}(G)$.  Then there exists $A \in \mathfrak{Rep}^{\textnormal{adm}}_{z_i = c_i}(G)$ and a $G$-equivariant injection $\pi \longhookrightarrow A$, such that $A|_I$ is injective in $\mathfrak{Rep}^\infty(I)$.  
\end{conj}

\begin{conj}[cf. \cite{emerton:ordII}, Conjecture 3.7.2]
\label{conjB}
The functors $\textnormal{H}^i\ord^G_{P^-}$ are effaceable on the category $\mathfrak{Rep}^{\textnormal{ladm}}(G)$ for $i > 0$.  Consequently, we have $\textnormal{H}^i\ord^G_{P^-} \simeq \textnormal{R}^i\ord^G_{P^-}$ for $i > 0$, where $\textnormal{R}^i\ord^G_{P^-}$ is computed in the category $\mathfrak{Rep}^{\textnormal{ladm}}(G)$.  
\end{conj}

\begin{conj}
\label{conjC}
Suppose $\pi \in \mathfrak{Rep}^{\textnormal{ladm}}(G)$ is an injective object.  Then the restriction $\pi|_I \in \mathfrak{Rep}^{\infty}(I)$ is also injective.  
\end{conj}

We note that Conjecture \ref{conjC} for $G$ implies Conjecture \ref{conjB} for $G$ and any standard parabolic $P$.  In Appendix \ref{app}, we provide proofs of these conjectures in several cases: (1) we prove Conjecture \ref{conjA} when the semisimple $F$-rank of $\bG$ is 0 or 1; (2) we prove Conjecture \ref{conjB} when the simply connected cover of the derived subgroup of $\bG$ is $\bS\bL_{2/F}$; (3) we prove Conjecture \ref{conjC} when the semisimple $F$-rank of $\bG$ is 0.  In particular, when the simply connected cover of the derived subgroup of $\bG$ is $\bS\bL_{2/F}$, Conjectures \ref{conjA} and \ref{conjB} hold, and Conjecture \ref{conjC} holds for $\bZ$, the Levi factor of the minimal $F$-parabolic subgroup of $\bG$.

\vspace{5pt}

\subsubsection{}\label{derivedadm}

We first examine the derived functors of the composition $\textnormal{H}^0(I_{M,1},-) \circ \ord_{P^-}^G$.  By composing with the fully faithful inclusion $\mathfrak{Rep}^{\textnormal{adm}}(M) \longhookrightarrow \mathfrak{Rep}^{\textnormal{ladm}}(M)$, we consider $\ord_{P^-}^G$ as a functor $\mathfrak{Rep}^{\textnormal{adm}}_{z_i = c_i}(G) \longrightarrow \mathfrak{Rep}^{\textnormal{ladm}}(M)$, and the commutative diagram \eqref{admcommdiag} becomes
\begin{equation}
\label{admcommdiag2}
\begin{tikzcd}[column sep = large, row sep = large]
\mathfrak{Rep}^{\textnormal{adm}}_{z_i = c_i}(G) \ar[r, "\ord^G_{P^-}"] \ar[d, "(-)^{I_1}"'] & \mathfrak{Rep}^{\textnormal{ladm}}(M) \ar[d, "(-)^{I_{M,1}}"] \\
\mathfrak{Mod-}\cH \ar[r, "\cR_{\cH_M}^{\cH}"] & \mathfrak{Mod-}\cH_M
\end{tikzcd}
\end{equation}

Let $\fA$ denote the full subcategory of $\mathfrak{Rep}^{\textnormal{adm}}_{z_i = c_i}(G)$ consisting of those representations $A$ such that $A|_I$ is injective in $\mathfrak{Rep}^\infty(I)$.  We make the following observations:
\begin{itemize}
\item Assume Conjecture \ref{conjA} is true for $G$.  Then $\fA$ is cogenerating in $\mathfrak{Rep}^{\textnormal{adm}}_{z_i = c_i}(G)$, in the terminology of \cite[Def. 8.3.21(v)]{kashiwaraschapira}.  Therefore, point (i) of Corollary 13.3.8 of \emph{op. cit.} is satisfied.
\item Consider a short exact sequence
$$0 \longrightarrow A' \longrightarrow A \longrightarrow A'' \longrightarrow 0$$
in $\mathfrak{Rep}^{\textnormal{adm}}_{z_i = c_i}(G)$, with $A, A' \in \fA$.  On restriction to $I$, the injectivity of $A|_I$ and $A'|_I$ implies the injectivity of $A''|_I$.  Therefore, $A'' \in \fA$ and point (ii) of Corollary 13.3.8 of \emph{op. cit.} is satisfied.
\item For $A'\in \fA$, the restriction $A'|_I$ is injective, and by definition of $\textnormal{H}^1\ord_{P^-}^G$ and \cite[Prop. 2.1.11]{emerton:ordII} we get $\textnormal{H}^1\ord_{P^-}^G(A') = 0$.  Therefore, given any short exact sequence
$$0 \longrightarrow A' \longrightarrow A \longrightarrow A'' \longrightarrow 0$$
in $\mathfrak{Rep}^{\textnormal{adm}}_{z_i = c_i}(G)$, with $A, A' \in \fA$, we get a short exact sequence
$$0 \longrightarrow \ord_{P^-}^G(A') \longrightarrow \ord_{P^-}^G(A) \longrightarrow \ord_{P^-}^G(A'') \longrightarrow 0.$$
Thus, point (iii) of \cite[Cor. 13.3.8]{kashiwaraschapira} is satisfied (for the functor $\ord_{P^-}^G$).
\end{itemize}
Combining these facts, \cite[Cor. 13.3.8]{kashiwaraschapira} implies that $\fA$ is $\ord_{P^-}^G$-injective (cf. \emph{op. cit.}, Definition 13.3.4).  In particular, by Proposition 13.3.5(i) of \emph{op. cit.}, we have the existence of the total derived functor
$$\textnormal{R}\ord_{P^-}^G: D^+\big(\mathfrak{Rep}^{\textnormal{adm}}_{z_i = c_i}(G)\big) \longrightarrow D^+\big(\mathfrak{Rep}^{\textnormal{ladm}}(M)\big).$$
Precisely, for $\pi \in \mathfrak{Rep}^{\textnormal{adm}}_{z_i = c_i}(G)$, we have $\textnormal{R}\ord_{P^-}^G(\pi) \cong \ord_{P^-}^G(A^\bullet)$, where 
$$0 \longrightarrow \pi \longrightarrow A^0 \longrightarrow A^1 \longrightarrow \ldots$$
is any resolution with $A^i \in \fA$.

\begin{lemma}
\label{comp-ord}
Suppose Conjectures \ref{conjA} and \ref{conjB} are true for $G$, and let $\pi \in \mathfrak{Rep}^{\textnormal{adm}}_{z_i = c_i}(G)$.  Then we have
$$h^i\big(\textnormal{R}\ord_{P^-}^G(\pi)\big) \cong \textnormal{R}^i\ord_{P^-}^G(\pi),$$
where $\textnormal{R}^i\ord_{P^-}^G(\pi)$ is calculated in the category $\mathfrak{Rep}^{\textnormal{ladm}}(G)$.  
\end{lemma}

\begin{proof}
Let $\iota_G:\mathfrak{Rep}^{\textnormal{adm}}_{z_i = c_i}(G) \longhookrightarrow \mathfrak{Rep}^{\textnormal{ladm}}(G)$ denote the fully faithful inclusion, so that we have a commutative diagram:
\begin{center}
\begin{tikzcd}[column sep = large, row sep = large]
 \mathfrak{Rep}^{\textnormal{adm}}_{z_i = c_i}(G) \ar[r, "\textnormal{Ord}_{P^-}^G"] \ar[d, hook, "\iota_G"'] &  \mathfrak{Rep}^{\textnormal{ladm}}(M) \\
  \mathfrak{Rep}^{\textnormal{ladm}}(G)  \ar[ru, "\textnormal{Ord}_{P^-}^{G,\textnormal{ladm}}"']  &  
\end{tikzcd}
\end{center}
(For the duration of this proof, we use the notation $\textnormal{Ord}_{P^-}^{G,\textnormal{ladm}}$ to distinguish between the two ordinary parts functors.)

Let $\fC$ denote the full subcategory of $\mathfrak{Rep}^{\textnormal{ladm}}(G)$ consisting of $\textnormal{Ord}_{P^-}^{G,\textnormal{ladm}}$-acyclic objects.  Note that $\fC$ contains all injective objects of $\mathfrak{Rep}^{\textnormal{ladm}}(G)$, as well as those locally admissible $G$-representations $A$ such that $A|_I$ is injective in $\mathfrak{Rep}^\infty(I)$ (by \cite[Prop. 2.1.11]{emerton:ordII} and Conjecture \ref{conjB}).  By \cite[Cor. 13.3.8]{kashiwaraschapira}, the category $\fC$ is $\textnormal{Ord}_{P^-}^{G,\textnormal{ladm}}$-injective, and similarly the category $\fA$ is $\iota_G$-injective (compare the discussion at the beginning of Subsubsection \ref{derivedadm}).  Since $\iota_G$ is exact and maps $\fA$ into $\fC$, Proposition 13.3.13(ii) of \emph{op. cit.} implies 
$$\textnormal{R}\textnormal{Ord}_{P^-}^G \simeq \textnormal{R}(\textnormal{Ord}_{P^-}^{G,\textnormal{ladm}} \circ \iota_G) \simeq \textnormal{R}\textnormal{Ord}_{P^-}^{G,\textnormal{ladm}} \circ \textnormal{R}\iota_G \simeq \textnormal{R}\textnormal{Ord}_{P^-}^{G,\textnormal{ladm}} \circ \iota_G.$$ 
\end{proof}

\vspace{5pt}

\subsubsection{}

Now let $\fB$ denote the full subcategory of $\mathfrak{Rep}^{\textnormal{ladm}}(M)$ consisting of $\textnormal{H}^0(I_{M,1},-)$-acyclic objects.  Proceeding as in the discussion of Subsubsection \ref{derivedadm} or the proof of Lemma \ref{comp-ord}, we get that $\fB$ is $\textnormal{H}^0(I_{M,1},-)$-injective.  Thus, by \cite[Prop. 13.3.5(i)]{kashiwaraschapira}, we have a total derived functor
$$\textnormal{R}\textnormal{H}^0(I_{M,1},-): D^+\big(\mathfrak{Rep}^{\textnormal{ladm}}(M)\big) \longrightarrow D^+\big(\mathfrak{Mod-}\cH_M\big),$$
which may be computed using injective resolutions as usual.

\begin{lemma}
\label{comp-coh}
Suppose Conjecture \ref{conjC} is true for $M$, and let $\tau \in \mathfrak{Rep}^{\textnormal{ladm}}(M)$.  Then we have
$$h^i\big(\textnormal{R}\textnormal{H}^0(I_{M,1},\tau)\big) \cong \textnormal{H}^i(I_{M,1},\tau),$$
where $\textnormal{H}^i(I_{M,1},\tau)$ is calculated in the category of smooth $M$-representations.  
\end{lemma}

\begin{proof}
Let $\iota_M': \mathfrak{Rep}^{\textnormal{ladm}}(M) \longhookrightarrow \mathfrak{Rep}^{\infty}(M)$ denote the fully faithful inclusion, so that we have a commutative diagram:
\begin{center}
\begin{tikzcd}[column sep = 50pt, row sep = large]
  \mathfrak{Rep}^{\textnormal{ladm}}(M)  \ar[r, "\textnormal{H}^0(I_{M,1}{,}-)"] \ar[d, hook, "\iota_M' "'] &  \mathfrak{Mod-}\cH_M \\
  \mathfrak{Rep}^{\infty}(M)  \ar[ru, "\textnormal{H}^0_{\infty}(I_{M,1}{,}-)"']  &  
\end{tikzcd}
\end{center}
(For the duration of this proof, we use the notation $\textnormal{H}^0_{\infty}(I_{M,1}{,}-)$ to distinguish between the two functors of invariants.)

Let $\fD$ denote the full subcategory of $\mathfrak{Rep}^{\infty}(M)$ consisting of $\textnormal{H}^0_\infty(I_{M,1},-)$-acyclic objects.  The total derived functor 
$$\textnormal{R}\textnormal{H}^0(I_{M,1},-): D^+\big(\mathfrak{Rep}^{\textnormal{ladm}}(M)\big) \longrightarrow D^+\big(\mathfrak{Mod-}\cH_M\big)$$
may be computed using injective resolutions, and Conjecture \ref{conjC} and \cite[Prop. 2.1.2]{emerton:ordII} imply that $\iota_M'(A) \in \fD$ for any injective $A \in  \mathfrak{Rep}^{\textnormal{ladm}}(M)$.  Therefore, \cite[Prop. 13.3.13(ii)]{kashiwaraschapira} and exactness of $\iota_M'$ imply
$$\textnormal{R}\textnormal{H}^0(I_{M,1},-) \simeq \textnormal{R}(\textnormal{H}^0_\infty(I_{M,1},-) \circ \iota_M') \simeq \textnormal{R}\textnormal{H}^0_\infty(I_{M,1},-) \circ \textnormal{R}\iota_M' \simeq \textnormal{R}\textnormal{H}^0_\infty(I_{M,1},-) \circ \iota_M'.$$
\end{proof}

\vspace{5pt}

\subsubsection{}

Next, we claim that $\ord_{P^-}^G(A) \in \fB$ for $A \in \fA$.  Recall that $\ord_{P^-}^G(A) = C[z^{\pm 1}] \otimes_{C[z]} A^{N_0^-}$ with $z$ as in Subsubsection \ref{smooth reps}.  Given a profinite group $H$ and a smooth $H$-representation $V$, we let $\textnormal{C}^i(H,V)$ denote the vector space of $V$-valued inhomogeneous $i$-cochains.  Letting $\cK$ denote the set of open normal subgroups of $I_{M,1}$, we have the following sequence of isomorphisms:
\begin{eqnarray}
C[z^{\pm 1}] \otimes_{C[z]} \textnormal{H}^i(I_{M,1},A^{N_0^-}) & \cong & C[z^{\pm 1}] \otimes_{C[z]} \Big( \varinjlim_{K \in \cK} \textnormal{H}^i(I_{M,1}/K,A^{KN_0^-}) \Big) \label{a}\\
 & \cong & \varinjlim_{K \in \cK}  C[z^{\pm 1}] \otimes_{C[z]}  \textnormal{H}^i(I_{M,1}/K,A^{KN_0^-}) \label{b} \\
 & \cong & \varinjlim_{K \in \cK}  C[z^{\pm 1}] \otimes_{C[z]}  h^i \big(\textnormal{C}^\bullet(I_{M,1}/K,A^{KN_0^-})\big) \notag \\
 & \cong & \varinjlim_{K \in \cK}   h^i \big(C[z^{\pm 1}] \otimes_{C[z]} \textnormal{C}^\bullet(I_{M,1}/K,A^{KN_0^-})\big) \label{c} \\
 & \cong & \varinjlim_{K \in \cK}   h^i \big(\textnormal{C}^\bullet(I_{M,1}/K, C[z^{\pm 1}] \otimes_{C[z]} A^{KN_0^-})\big) \label{d} \\
 & \cong & \varinjlim_{K \in \cK}   \textnormal{H}^i(I_{M,1}/K, C[z^{\pm 1}] \otimes_{C[z]} A^{KN_0^-}) \notag \\ 
 & \cong & \textnormal{H}^i(I_{M,1}, C[z^{\pm 1}] \otimes_{C[z]} A^{N_0^-}) \label{e}
\end{eqnarray}
The isomorphism \eqref{a} follows from \cite[\S I.2.2, Cor. 1]{serre:galoiscoh}; \eqref{b} follows from the fact that direct limits commute with tensor products; \eqref{c} follows from the fact that cohomology of a cochain complex commutes with exact functors; \eqref{d} follows from the isomorphism of $C[z]$-modules $\textnormal{C}^i(I_{M,1}/K,A^{KN_0^-}) \cong (A^{KN_0^-})^{\oplus i[I_{M,1}: K]}$; and \eqref{e} follows from the fact that $\varinjlim_{K \in \cK} C[z^{\pm 1}] \otimes_{C[z]} A^{KN_0^-} \cong C[z^{\pm 1}] \otimes_{C[z]} A^{N_0^-}$ and \cite[\S I.2.2, Prop. 8]{serre:galoiscoh}.

By the paragraph above, in order to show $\textnormal{H}^i(I_{M,1},\ord_{P^-}^G(A)) = 0$ for $i > 0$, it suffices to show $\textnormal{H}^i(I_{M,1},A^{N_0^-}) = 0$ for $i > 0$.  Note that we have a Hochschild--Serre spectral sequence
$$\textnormal{H}^i\big(I_{M,1}, \textnormal{H}^j(N_0^-, A)\big) \Longrightarrow \textnormal{H}^{i + j}(I_1 \cap P^-,A).$$
By the definition of $\fA$ and \cite[Prop. 2.1.11]{emerton:ordII}, we have that $A|_{N_0^-}$ is injective, so that $\textnormal{H}^j(N_0^-, A) = 0$ for $j > 0$.  The above spectral sequence therefore collapses to give
$$\textnormal{H}^i(I_{M,1},  A^{N_0^-}) \cong \textnormal{H}^{i}(I_1 \cap P^-,A).$$
Applying Proposition 2.1.11 of \emph{op. cit.} once more to the group $I_1 \cap P^-$ gives $\textnormal{H}^{i}(I_1 \cap P^-,A) = 0$ for $i > 0$, which gives the claim.

Since $\ord_{P^-}^G(\fA) \subset \fB$, \cite[Prop. 13.3.13(ii)]{kashiwaraschapira} implies that $\fA$ is $\textnormal{H}^0(I_{M,1},-) \circ \ord_{P^-}^G$-injective, and we have a natural equivalence
\begin{equation}\label{admnattrans}
\textnormal{R}\big(\textnormal{H}^0(I_{M,1},-) \circ \ord_{P^-}^G\big) \simeq \textnormal{R}\textnormal{H}^0(I_{M,1},-) \circ \textnormal{R}\ord_{P^-}^G.
\end{equation}
In particular, we may calculate $\textnormal{R}(\textnormal{H}^0(I_{M,1},-) \circ \ord_{P^-}^G)$ using resolutions in $\fA$.

\begin{lemma}
Suppose Conjectures \ref{conjA} and \ref{conjB} are true for $G$, and Conjecture \ref{conjC} is true for $M$.  Applying the natural transformation \eqref{admnattrans} to $\pi \in \mathfrak{Rep}^{\textnormal{adm}}_{z_i = c_i}(G)$ yields a spectral sequence of $\cH_M$-modules
\begin{equation}
\label{admss1}
\textnormal{H}^i\big(I_{M,1},\textnormal{R}^j\ord_{P^-}^G(\pi)\big) \Longrightarrow \textnormal{R}^{i + j}\big(\textnormal{H}^0(I_{M,1},-) \circ \ord_{P^-}^G\big)(\pi),
\end{equation} 
where $\textnormal{R}^j\ord_{P^-}^G$ is calculated in the category $\mathfrak{Rep}^{\textnormal{ladm}}(G)$ and $\textnormal{H}^i\big(I_{M,1},-)$ is calculated in the category $\mathfrak{Rep}^\infty(M)$.  
\end{lemma}

\begin{proof}
This follows from the construction of the Grothendieck spectral sequence and Lemmas \ref{comp-ord} and \ref{comp-coh}.
\end{proof}

\vspace{5pt}

\subsubsection{}

We now examine the derived functors of the composition $\cR^{\cH}_{\cH_M} \circ \textnormal{H}^0(I_1,-)$.  We continue to assume Conjecture \ref{conjA} for $G$.

Let $A \in \fA$.  Since $I_1$ is open in $I$, the restriction $A|_{I_1}$ is injective in $\mathfrak{Rep}^\infty(I_1)$ (\cite[Prop. 2.1.2]{emerton:ordII}).  Therefore, $\fA$ satisfies point (iii) of \cite[Cor. 13.3.8]{kashiwaraschapira} (for the functor $\textnormal{H}^0(I_1,-)$), and points (i) and (ii) hold exactly as in Subsubsection \ref{derivedadm}.  By \emph{op. cit.}, we obtain a total derived functor
$$\textnormal{R}\textnormal{H}^0(I_1,-): D^+\big(\mathfrak{Rep}^{\textnormal{adm}}_{z_i = c_i}(G)\big) \longrightarrow D^+\big(\mathfrak{Mod-}\cH\big),$$
which may be computed using resolutions in $\fA$.

\begin{lemma}
\label{comp-hecke}
Suppose Conjecture \ref{conjA} is true for $G$, and let $\pi \in \mathfrak{Rep}^{\textnormal{adm}}_{z_i = c_i}(G)$.  Then we have
$$h^i\big(\textnormal{R}\textnormal{H}^0(I_1,\pi)\big) \cong \textnormal{H}^i(I_1,\pi),$$
where $\textnormal{H}^i(I_1,\pi)$ is calculated in the category $\mathfrak{Rep}^{\infty}(G)$.  
\end{lemma}

\begin{proof}
Let $\iota_G'':\mathfrak{Rep}^{\textnormal{adm}}_{z_i = c_i}(G) \longhookrightarrow \mathfrak{Rep}^{\infty}(G)$ denote the fully faithful inclusion, so that we have a commutative diagram:
\begin{center}
\begin{tikzcd}[column sep = 50pt, row sep = large]
 \mathfrak{Rep}^{\textnormal{adm}}_{z_i = c_i}(G) \ar[r, "\textnormal{H}^0(I_1{,}-)"] \ar[d, hook, "\iota_G'' "'] &  \mathfrak{Mod-}\cH \\
  \mathfrak{Rep}^{\infty}(G)  \ar[ru, "\textnormal{H}^0_\infty(I_1{,}-)"']  &  
\end{tikzcd}
\end{center}
(For the duration of this proof, we use the notation $\textnormal{H}^0_{\infty}(I_{1}{,}-)$ to distinguish between the two functors of invariants.)

Let $\fE$ denote the full subcategory of $\mathfrak{Rep}^{\infty}(G)$ consisting of $\textnormal{H}^0_\infty(I_{1},-)$-acyclic objects.  By \cite[Cor. 13.3.8]{kashiwaraschapira}, the category $\fE$ is $\textnormal{H}^0_\infty(I_1,-)$-injective, and likewise the category $\fA$ is $\iota_G''$-injective.  By \cite[Prop. 2.1.2]{emerton:ordII}, the functor $\iota_G''$ maps $\fA$ into $\fE$.  Therefore, the exactness of $\iota_G''$ and \cite[Prop. 13.3.13(ii)]{kashiwaraschapira} give
$$\textnormal{R}\textnormal{H}^0(I_{1},-) \simeq \textnormal{R}(\textnormal{H}^0_\infty(I_{1},-) \circ \iota_G'') \simeq \textnormal{R}\textnormal{H}^0_\infty(I_{1},-) \circ \textnormal{R}\iota_G'' \simeq \textnormal{R}\textnormal{H}^0_\infty(I_{1},-) \circ \iota_G''.$$
\end{proof}

\vspace{5pt}

\subsubsection{}

Let $\fF$ denote the full subcategory of $\mathfrak{Mod-}\cH$ consisting of $\cR^{\cH}_{\cH_M}$-acyclic objects.   By \cite[Cor. 13.3.8, Prop. 13.3.5(i)]{kashiwaraschapira}, $\fF$ is $\cR^{\cH}_{\cH_M}$-injective, and we have a total derived functor
$$\textnormal{R}\cR^{\cH}_{\cH_M}: D^+\big(\mathfrak{Mod-}\cH\big) \longrightarrow D^+\big(\mathfrak{Mod-}\cH_M\big),$$
which may be computed using injective resolutions as usual.

\vspace{5pt}

\subsubsection{}

By the final paragraph in Subsubsection \ref{hecke mods}, we see that $\fF$ contains all finite-dimensional $\cH$-modules.   Furthermore, since $\textnormal{H}^0(I_1, A)$ is finite-dimensional for every $A \in \fA$, we have $\textnormal{H}^0(I_1, \fA)\subset \fF$.  Proposition 13.3.13(ii) of \cite{kashiwaraschapira} then implies that $\fA$ is $\cR^{\cH}_{\cH_M} \circ \textnormal{H}^0(I_1,-)$-acyclic, and we have
\begin{equation}
\label{admnattrans2}
\textnormal{R}\big(\cR^{\cH}_{\cH_M} \circ \textnormal{H}^0(I_1,-)\big) \simeq \textnormal{R}\cR^{\cH}_{\cH_M} \circ \textnormal{R}\textnormal{H}^0(I_1,-).
\end{equation}
In particular, we may calculate $\textnormal{R}(\cR^{\cH}_{\cH_M} \circ \textnormal{H}^0(I_1,-))$ using resolutions in $\fA$.

\begin{lemma}
Suppose Conjecture \ref{conjA} is true for $G$.  Applying the natural transformation \eqref{admnattrans2} to $\pi \in \mathfrak{Rep}^{\textnormal{adm}}_{z_i = c_i}(G)$ yields a spectral sequence of $\cH_M$-modules
\begin{equation}
\label{admss2}
\textnormal{R}^i\cR_{\cH_M}^{\cH}\big(\textnormal{H}^j(I_1,\pi)\big) \Longrightarrow \textnormal{R}^{i + j}\big(\cR^{\cH}_{\cH_M} \circ \textnormal{H}^0(I_1,-)\big)(\pi),
\end{equation}
where $\textnormal{H}^j(I_1,-)$ is calculated in the category $\mathfrak{Rep}^\infty(G)$.  
\end{lemma}

\begin{proof}
This follows from the construction of the Grothendieck spectral sequence and Lemma \ref{comp-hecke}.
\end{proof}

\begin{lemma}
\label{lem adm implies fd}
Suppose $\pi\in \mathfrak{Rep}^{\textnormal{adm}}(G)$.  Then $\dim_C(\textnormal{H}^j(I_1,\pi)) < \infty$ for every $j \geq 0$.  
\end{lemma}

\begin{proof}
This follows from \cite[Lem. 3.4.4]{emerton:ordII}.
\end{proof}

Suppose now that $\pi \in \mathfrak{Rep}^{\textnormal{adm}}_{z_i = c_i}(G)$.  Combining Lemma \ref{lem adm implies fd} with the last paragraph of Subsubsection \ref{hecke mods}, we see that $\textnormal{R}^i\cR_{\cH_M}^{\cH}(\textnormal{H}^j(I_1,\pi)) = 0$ for $i > 0$.  Thus the spectral sequence \eqref{admss2} collapses to an isomorphism of $\cH_M$-modules
\begin{equation}
\label{isom2}
 \cR^{\cH}_{\cH_M}\big(\textnormal{H}^j(I_1,\pi)\big) \cong \textnormal{R}^j\big(\cR^{\cH}_{\cH_M} \circ \textnormal{H}^0(I_1,-)\big)(\pi).
\end{equation}
Combining the isomorphism \eqref{isom2} with the spectral sequence \eqref{admss1} along with the commutativity of the diagram \eqref{admcommdiag2} gives the following result.

\begin{thm}
Suppose Conjectures \ref{conjA} and \ref{conjB} are true for $G$, and Conjecture \ref{conjC} is true for $M$.  Let $\pi \in \mathfrak{Rep}^{\textnormal{adm}}_{z_i = c_i}(G)$.  Then we have an $E_2$ spectral sequence of $\cH_M$-modules
$$\textnormal{H}^i\big(I_{M,1},\textnormal{R}^j\ord_{P^-}^G(\pi)\big) \Longrightarrow \cR^{\cH}_{\cH_M}\big(\textnormal{H}^{i + j}(I_1,\pi)\big).$$
\end{thm}

\begin{cor}\label{ordss}
Suppose the simply connected cover of the derived subgroup of $\bG$ is $\bS\bL_{2/F}$, and let $\pi \in \mathfrak{Rep}^{\textnormal{adm}}_{z_i = c_i}(G)$.  Then we have an $E_2$ spectral sequence of $\cH_M$-modules
$$\textnormal{H}^i\big(I_{M,1},\textnormal{R}^j\ord_{P^-}^G(\pi)\big) \Longrightarrow \cR^{\cH}_{\cH_M}\big(\textnormal{H}^{i + j}(I_1,\pi)\big).$$
\end{cor}

\begin{proof}
This follows from Theorems \ref{emerton-conjthm}, \ref{sl2effthm}, and \ref{rk0-thm}.
\end{proof}

\vspace{5pt}

\subsection{The smooth case}

We now examine a variant of the above spectral sequence.  It has the advantage of not being conditional on Conjectures \ref{conjA}, \ref{conjB}, and \ref{conjC}, but seems difficult to compute in practice.

\vspace{5pt}

\subsubsection{}

Recall that the category $\mathfrak{Rep}^\infty(G)$ is abelian and has enough injectives (and likewise for $M$).  Therefore, we have the existence of total derived functors
\begin{alignat*}{3}
\textnormal{R}\cR^G_P: \quad & D^+\big(\mathfrak{Rep}^\infty(G)\big) ~ & \longrightarrow ~ & D^+\big(\mathfrak{Rep}^\infty(M)\big),\\
\textnormal{R}\textnormal{H}^0(I_{M,1},-): \quad & D^+\big(\mathfrak{Rep}^\infty(M)\big) ~ & \longrightarrow ~ & D^+\big(\mathfrak{Mod-}\cH_M\big).
\end{alignat*}
Since the functor $\cR_P^G: \mathfrak{Rep}^\infty(G) \longrightarrow \mathfrak{Rep}^\infty(M)$ is right adjoint to the exact parabolic induction functor $\Ind_P^G: \mathfrak{Rep}^\infty(M) \longrightarrow \mathfrak{Rep}^\infty(G)$, it maps injective objects of $\mathfrak{Rep}^\infty(G)$ to injective objects of $\mathfrak{Rep}^\infty(M)$.  Therefore, \cite[Prop. 13.3.13(ii)]{kashiwaraschapira} implies we have
$$\textnormal{R}\big(\textnormal{H}^0(I_{M,1},-) \circ \cR^G_P\big) \simeq \textnormal{R}\textnormal{H}^0(I_{M,1},-) \circ \textnormal{R}\cR^G_P.$$
Consequently, for $\pi \in \mathfrak{Rep}^\infty(G)$, by taking cohomology of the above we get a spectral sequence of $\cH_M$-modules
\begin{equation}
\label{sseq1}
\textnormal{H}^i\big(I_{M,1}, \textnormal{R}^j\cR^G_P(\pi)\big) \Longrightarrow \textnormal{R}^{i + j}\big(\textnormal{H}^0(I_{M,1},-) \circ \cR^G_P\big)(\pi).
\end{equation}

\vspace{5pt}

\subsubsection{}

We now consider the composition $\cR^{\cH}_{\cH_M} \circ \textnormal{H}^0(I_1,-)$.  As in the previous paragraph, the abelian categories $\mathfrak{Rep}^\infty(G)$ and $\mathfrak{Mod-}\cH$ have enough injectives, and therefore we have the total derived functors
\begin{alignat*}{3}
\textnormal{R}\textnormal{H}^0(I_1,-): \quad & D^+\big(\mathfrak{Rep}^\infty(G)\big) ~ & \longrightarrow ~ & D^+\big(\mathfrak{Mod-}\cH\big),\\
\textnormal{R}\cR^{\cH}_{\cH_M}: \quad & D^+\big(\mathfrak{Mod-}\cH\big) ~ & \longrightarrow ~ & D^+\big(\mathfrak{Mod-}\cH_M\big).
\end{alignat*}

\vspace{5pt}

\subsubsection{}

We wish to understand the composition $\textnormal{R}\cR^{\cH}_{\cH_M} \circ \textnormal{R}\textnormal{H}^0(I_1,-)$.  This will rely on the following lemma.  We thank Noriyuki Abe for his help with the argument.

\begin{lemma}
Suppose $\pi\in \mathfrak{Rep}^\infty(G)$ is an injective object.  Then $\pi^{I_1}$ is an $\cR_{\cH_M}^{\cH}$-acyclic $\cH$-module.  
\end{lemma}

\begin{proof}
Let $\pi^\vee := \Hom_C(\pi,C)$ denote the $C$-linear dual of $\pi$.  Using the contragredient action, we view $\pi^\vee$ as a not necessarily smooth $G$-representation (that is, a $C[G]$-module).  Fix a presentation $\bigoplus_{a \in \sA} C[G] \longtwoheadrightarrow \pi^\vee$, where $\sA$ is some index set.  Dualizing, we obtain injections of $C[G]$-modules
$$\pi \longhookrightarrow \pi^{\vee\vee} \longhookrightarrow \Big(\bigoplus_{a \in \sA} C[G]\Big)^\vee \cong \prod_{a \in \sA} C[G]^\vee.$$
Taking smooth vectors, we obtain an injection of smooth $G$-representations
$$\pi \longhookrightarrow \Big(\prod_{a \in \sA} C[G]^\vee\Big)^\infty.$$
As $\pi$ is injective, this injection splits.  Since $\cR_{\cH_M}^{\cH}$ is additive, in order to prove $\pi^{I_1}$ is $\cR_{\cH_M}^{\cH}$-acyclic, it suffices to show $((\prod_{a \in \sA} C[G]^\vee)^\infty)^{I_1}$ is $\cR_{\cH_M}^{\cH}$-acyclic.

We have isomorphisms of right $\cH$-modules
\begin{eqnarray*}
\bigg(\Big(\prod_{a \in \sA} C[G]^\vee\Big)^\infty\bigg)^{I_1} & = & \Big(\prod_{a \in \sA} C[G]^\vee\Big)^{I_1}\\
 & \cong & \prod_{a \in \sA}(C[G]_{I_1})^\vee \\
 & \cong & \prod_{a \in \sA} C[I_1\backslash G]^\vee \\
 & \cong & \Big(\bigoplus_{a \in \sA} C[I_1\backslash G]\Big)^\vee.
\end{eqnarray*}
A few comments on the above isomorphisms: given a not necessarily smooth $G$-representation $M$, the space of $I_1$-coinvariants $M_{I_1}$ carries a \emph{left} action of $\cH$ by composing the restriction map to $I_1 \cap g^{-1}I_1g$, the conjugation by $g$, and the corestriction map from $I_1 \cap gI_1g^{-1}$.  When $M = C[G]$, this is most easily seen using the isomorphism $C[G]_{I_1} \cong \cind_{I_1}^G(\mathbf{1}_{I_1})$.  We may convert the left $\cH$-action into a right $\cH$-action by twisting by the anti-involution $\T_g \longmapsto \T_{g^{-1}}$, and as a result we view $C[G]_{I_1} \cong C[I_1\backslash G]$ as a right $\cH$-module.  The dual $(C[G]_{I_1})^\vee \cong C[I_1\backslash G]^\vee$ is once again a right $\cH$-module (as in Subsection \ref{subsec:heckealgs}), and the isomorphism $(C[G]^\vee)^{I_1} \cong (C[G]_{I_1})^\vee$ appearing in the second line is equivariant for the right $\cH$-action on both.

Using the above isomorphisms, it suffices to show that if $\fm$ is any $\cH$-module, then $\fm^\vee$ is $R_{\cH_M}^{\cH}$-acyclic.  To this end, let 
$$\ldots \longrightarrow \fp_2 \longrightarrow \fp_1 \longrightarrow \fp_0 \longrightarrow \fm \longrightarrow 0$$
denote a projective resolution of a right $\cH$-module $\fm$.  Dualizing, we obtain an injective resolution
$$0 \longrightarrow \fm^\vee \longrightarrow \fp_0^\vee \longrightarrow \fp_1^\vee \longrightarrow \fp_2^\vee \longrightarrow \ldots$$
of right $\cH$-modules.  Applying $\cR_{\cH_M}^{\cH}$, we obtain a complex
\begin{equation}
\label{rightadjcomplex}
0 \longrightarrow \cR_{\cH_M}^{\cH}(\fm^\vee) \longrightarrow \cR_{\cH_M}^{\cH}(\fp_0^\vee) \longrightarrow \cR_{\cH_M}^{\cH}(\fp_1^\vee) \longrightarrow \cR_{\cH_M}^{\cH}(\fp_2^\vee) \longrightarrow \ldots
\end{equation}
of right $\cH_M$-modules.  Now, the proof of \cite[Prop. 3.6]{abe:extensions} shows that for any right $\cH$-module $\fn$, the module $\cR_{\cH_M}^{\cH}(\fn^\vee)$ is equal to the composition of an exact localization functor, an exact twisting functor, and an exact duality functor.  Therefore, the complex \eqref{rightadjcomplex} is exact, and consequently $\textnormal{R}^i\cR_{\cH_M}^{\cH}(\fm^\vee) = 0$ for $i > 0$.  
\end{proof}

The lemma above shows that $\textnormal{H}^0(I_1,-)$ maps injective objects into $\cR_{\cH_M}^{\cH}$-acyclic objects, and \cite[Prop. 13.3.13(ii)]{kashiwaraschapira} gives a natural equivalence
$$ \textnormal{R}\big(\cR^{\cH}_{\cH_M} \circ \textnormal{H}^0(I_1,-)\big) \simeq \textnormal{R}\cR^{\cH}_{\cH_M} \circ \textnormal{R}\textnormal{H}^0(I_1,-)$$
(by taking $\cJ'$ to be equal to the subcategory of $\cR^{\cH}_{\cH_M}$-acyclic modules).  Thus, for $\pi \in \mathfrak{Rep}^\infty(G)$, taking cohomology of the above gives a spectral sequence of $\cH_M$-modules
\begin{equation}
\label{sseq2}
\textnormal{R}^i\cR_{\cH_M}^{\cH}\big(\textnormal{H}^j(I_1,\pi)\big) \Longrightarrow \textnormal{R}^{i + j}\big(\cR^{\cH}_{\cH_M} \circ \textnormal{H}^0(I_1,-)\big)(\pi).
\end{equation}

\vspace{5pt}

\subsubsection{}

Suppose now that $\pi \in \mathfrak{Rep}^{\textnormal{adm}}(G)$.  Using Lemma \ref{lem adm implies fd} and the discussion of Subsubsection \ref{hecke mods}, the spectral sequence \eqref{sseq2} collapses to give
$$\cR_{\cH_M}^{\cH}(\textnormal{H}^j(I_1,\pi)) \cong \textnormal{R}^j\big(\cR^{\cH}_{\cH_M} \circ \textnormal{H}^0(I_1,-)\big)(\pi).$$
Substituting this into \eqref{sseq1} and using the commutativity of \eqref{smoothcommdiag} gives the following.

\begin{thm}
Suppose $\pi\in \mathfrak{Rep}^{\textnormal{adm}}(G)$.  Then we have an $E_2$ spectral sequence of $\cH_M$-modules
$$\textnormal{H}^i\big(I_{M,1}, \textnormal{R}^j\cR_P^G(\pi)\big) \Longrightarrow \cR_{\cH_M}^{\cH}\big(\textnormal{H}^{i + j}(I_1, \pi)\big).$$
\end{thm}

\vspace{10pt}

\section{Duality}
\label{sec:duality}

We now discuss how pro-$p$-Iwahori cohomology interacts with Kohlhaase's higher duality functors \cite{kohlhaase:duality}.  

\vspace{5pt}

\subsection{Review of Pontryagin duality}

Let $K$ denote a compact open subgroup of $G$, and let $\Lambda(K)$ denote the completed group algebra of $K$:
$$\Lambda(K) := C\llbracket K \rrbracket = \varprojlim_N C[K/N],$$
where $N$ runs over the open normal subgroups of $K$.  We let $\Lambda(K)\mathfrak{-Mod}^{\textnormal{pc}}$ denote the category of pseudocompact left $\Lambda(K)$-modules; it is an abelian category with enough projectives (\cite[Lem. 1.6]{brumer}).  We let $\Ext^i_{\Lambda(K)}(-,-)$ (resp., $\Ext^i_{\Lambda(K)\mathfrak{-Mod}^{\textnormal{pc}}}(-,-)$) denote the Ext functor computed in the category of all left $\Lambda(K)$-modules (resp., computed in the category $\Lambda(K)\mathfrak{-Mod}^{\textnormal{pc}}$).

Fixing a compact open subgroup $K$ as above, we define 
$$\Lambda(G) := C[G] \otimes_{C[K]} \Lambda(K),$$ 
and let $\Lambda(G)\mathfrak{-Mod}^{\textnormal{pc}}$ denote the category of pseudocompact left $\Lambda(G)$-modules (see \cite[\S 1]{kohlhaase:duality} for the relevant definitions).  Note that, up to isomorphism, the algebra structure on $\Lambda(G)$ is independent of the choice of $K$.  Pontryagin duality then induces quasi-inverse anti-equivalences of categories
\begin{eqnarray*}
\mathfrak{Rep}^\infty(G) & \cong & \Lambda(G)\mathfrak{-Mod}^{\textnormal{pc}} \\
\pi & \longmapsto & \pi^\vee := \Hom_C(\pi, C) \\
\Hom_C^{\textnormal{cts}}(M, C) =: M^\vee& \longmapsfrom & M
\end{eqnarray*}

\vspace{5pt}

\subsection{Preparation}

Let $d := \dim_{\bbQ_p}(G)$ denote the dimension of $G$ as a $p$-adic manifold.  For $0 \leq i \leq d$, the smooth duality functors $S^i$ are endofunctors of the category $\mathfrak{Rep}^\infty(G)$ of smooth $G$-representations, defined by 
$$S^i(\pi) := \varinjlim_K \Ext_{\Lambda(K)}^i(C, \pi^\vee),$$
where $\pi \in \mathfrak{Rep}^\infty(G)$, $K$ runs over all compact open subgroups of $G$, and we view $\pi^\vee$ as a pseudocompact left $\Lambda(K)$-module.  The transition maps in the direct limit are given by the restriction
$$\textnormal{res}^{K}_{K'}: \Ext_{\Lambda(K)}^i(C, \pi^\vee) \longrightarrow \Ext_{\Lambda(K')}^i(C, \pi^\vee)$$
for every inclusion $K' \subset K$ of compact open subgroups.  The $S^i$ form a contravariant $\delta$-functor, with $S^d$ being right-exact.  

\vspace{5pt}

\subsubsection{}

We recall the following useful facts regarding the completed group algebras $\Lambda(K)$.

\begin{lemma}\label{props}
Let $K$ denote a compact open pro-$p$ subgroup of $G$.
\begin{enumerate}[(a)]
\item The completed group algebra $\Lambda(K)$ is a noetherian local ring.  
\item Let $M$ denote a pseudocompact $\Lambda(K)$-module.  Then we have canonical isomorphisms of $C$-vector spaces
$$\Ext^i_{\Lambda(K)}(C,M) \cong \Ext^i_{\Lambda(K)\mathfrak{-Mod}^{\textnormal{pc}}}(C,M) \cong \textnormal{H}^i_{\textnormal{cts}}(K,M)$$
for all $i\geq 0$.  \label{props-2}
\item Suppose $K$ is torsion-free.  Then we have
$$\textnormal{Ext}^i_{\Lambda(K)}(C,\Lambda(K)) = \begin{cases}0 & \textnormal{if}~ i \neq d,\\ C & \textnormal{if}~ i = d.\end{cases} $$ \label{props-3}
\end{enumerate}
\end{lemma}

\begin{proof}
(a) When $C = \bbF_p$, this is due to Lazard \cite[Thm. V.2.2.4]{lazard}.  In general, the result may be deduced from \cite[Ch. IX, \S 2.3, Prop. 5]{bourbaki:ac89} and \cite[Thm. 33.4]{schneider:padicliegroups}.

(b) This follows from \cite[Thm. V.3.2.7]{lazard} (see also \cite[Ch. V, Prop. 5.2.14]{nsw:coh}).

(c)  By \cite[Cor. 4.3]{DDMS}, the subgroup $K$ possesses an open normal subgroup $J$ which is uniform.  Using \cite[Eq. (5), pf. of Prop. 3.8]{kohlhaase:duality}, we have
$$\Ext^i_{\Lambda(J)}(C,\Lambda(J)) = \begin{cases}0 & \textnormal{if}~ i \neq d,\\ C & \textnormal{if}~ i = d.\end{cases}$$
Applying the Eckmann--Shapiro lemma
$$\Ext^i_{\Lambda(J)}(C,\Lambda(J)) \cong \Ext^i_{\Lambda(K)}\big(C,\textnormal{coind}_{\Lambda(J)}^{\Lambda(K)}(\Lambda(J))\big) \cong \Ext^i_{\Lambda(K)}(C,\Lambda(K))$$
gives the result.  (Compare \cite[Rmk. 4.2.9]{symondsweigel}.)
\end{proof}

\vspace{5pt}

\subsubsection{}

We now examine the $\delta$-functor $(S^i)_{i \geq 0}$.  The following result is due to Kohlhaase.

\begin{lemma}\label{coefface}
For $0 \leq i < d$, the functors $S^i$ are coeffaceable.  Consequently, we have $\textnormal{L}_iS^d \simeq S^{d - i}.$
\end{lemma}

\begin{proof}
Let $A$ denote an injective object of $\mathfrak{Rep}^\infty(G)$.  We claim $S^i(A) = 0$ for $0 \leq i < d$.  Suppose
$$v\in S^i(A) = \varinjlim_K \Ext_{\Lambda(K)}^i(C, A^\vee),$$
and let $v$ be represented by an element of $\Ext_{\Lambda(K)}^i(C, A^\vee)$ for some fixed $K$.  By definition of the direct limit, we may replace $v$ by $\textnormal{res}^K_J(v)$, where $J$ is a torsion-free compact open pro-$p$ group, and suppose $v$ is represented by an element of $\Ext_{\Lambda(J)}^i(C, A^\vee)$.

We claim that $\Ext_{\Lambda(J)}^i(C, A^\vee)$ is trivial.  Indeed, since $J$ is open in $G$, the functor of compact induction $\cind_J^G$ is exact, and therefore its right adjoint $\textnormal{res}^G_J: \mathfrak{Rep}^\infty(G) \longrightarrow \mathfrak{Rep}^\infty(J)$ preserves injective objects.  Consequently, the Pontryagin dual $A^\vee \in \Lambda(G)\mathfrak{-Mod}^{\textnormal{pc}}$ is projective, and its restriction $A^\vee|_{\Lambda(J)} \in \Lambda(J)\mathfrak{-Mod}^{\textnormal{pc}}$ is projective as well.  By \cite[Lem. 1.6]{brumer} we may write $A^\vee|_{\Lambda(J)}$ as a direct summand of $\prod_{a \in \sA} \Lambda(J)$ for some index set $\sA$.  Therefore it suffices to show the vanishing of 
$$\Ext_{\Lambda(J)}^i\Big(C, \prod_{a \in \sA}\Lambda(J)\Big) \cong \prod_{a\in \sA}\Ext_{\Lambda(J)}^i(C, \Lambda(J)).$$
This follows from Lemma \ref{props}\eqref{props-3}.  
\end{proof}

\begin{corpf}\label{cor-coefface}
Suppose $A$ is an injective $G$-representation and $K$ a torsion-free compact open pro-$p$ subgroup of $G$.  Then we have $\Ext^i_{\Lambda(K)}(C,A^\vee) = 0$ for $0 \leq i < d$.  
\end{corpf}

\vspace{5pt}

\subsection{Duality}

From this point onwards, we suppose that the pro-$p$-Iwahori subgroup $I_1$ is torsion-free.  This implies that the cohomological dimension of $I_1$ is equal to $d = \dim_{\bbQ_p}(G)$ (\cite[Cor. (1)]{serre:cohdim}).  Note that the torsion-freeness condition is satisfied if $p$ is sufficiently large relative to the root system of $\bG$ and the ramification degree of $F$; for an explicit bound (at least when $\bG$ is semisimple), see \cite[Prop. 12.1]{totaro:eulerchar}.

\vspace{5pt}

\subsubsection{}

Let $\pi$ denote a smooth $G$-representation, and choose an injective resolution
$$0 \longrightarrow \pi \longrightarrow A^0 \longrightarrow A^1 \longrightarrow \ldots$$
in $\mathfrak{Rep}^\infty(G)$.  Let us define $B^i := \ker(A^{i} \longrightarrow A^{i + 1})$.  Then we can truncate the above to obtain a resolution
\begin{equation}\label{I1acyclic}
0 \longrightarrow \pi \longrightarrow A^0 \longrightarrow A^1 \longrightarrow \ldots \longrightarrow A^{d - 1} \longrightarrow B^d  \longrightarrow 0.
\end{equation}

\begin{lemma}\label{lem-acyclic}
Let $i \geq d$.  
\begin{enumerate}[(a)]
\item The $G$-representation $B^i$ is $S^d$-acyclic. \label{lem-acyclic-1}
\item The $G$-representation $B^i$ is $I_1$-acyclic. \label{lem-acyclic-2}
\end{enumerate}
\end{lemma}

\begin{proof}
(a)  We have a short exact sequence
$$0 \longrightarrow B^{i} \longrightarrow A^i \longrightarrow B^{i + 1} \longrightarrow 0$$
for all $i \geq 0$.  By examining the long exact sequence induced by applying the contravariant right-exact functor $S^d$, and using injectivity of $A^i$, we see that $\textnormal{L}_{j}S^d(B^{i + 1}) \cong \textnormal{L}_{j + 1}S^d(B^i)$ for every $j \geq 1$.  In particular, if $i \geq d$ and $j \geq 1$, then $d - i - j < 0$ and we have 
$$\textnormal{L}_{j}S^d(B^i) \cong \textnormal{L}_{j + i}S^d(B^0) \cong S^{d - i - j}(B^0) = 0.$$
Thus, $B^i$ is $S^d$-acyclic for $i \geq d$.

(b) This follows from a dimension-shifting argument, exactly as in part \eqref{lem-acyclic-1}.
\end{proof}

\begin{cor}
The resolution \eqref{I1acyclic} of $\pi$ is both $S^d$-acyclic and $I_1$-acyclic.
\end{cor}

\begin{rmk}\label{Biacyclic}
The conclusion of Lemma \ref{lem-acyclic}\eqref{lem-acyclic-1} can be strengthened as follows: for $i \geq d$ and $K$ a torsion-free compact open pro-$p$ subgroup of $G$, we have $\Ext^j_{\Lambda(K)}(C,B^{i,\vee}) = 0$ for $0 \leq j < d$.  Indeed, by Corollary \ref{cor-coefface}, this assertion holds if $B^i$ is replaced by $A^i$ with $i \geq 0$.  The desired claim then follows from a straightforward dimension-shifting argument.  
\end{rmk}

\vspace{5pt}

\subsubsection{}

We define a complex $Y^\bullet$ of smooth $G$-representations by applying $S^d$ to \eqref{I1acyclic}, truncating, and translating.  Explicitly, we have
$$0 \longrightarrow Y^0 \longrightarrow Y^1 \longrightarrow \ldots \longrightarrow Y^d \longrightarrow 0,$$
with $Y^0 = S^d(B^d)$ and $Y^i = S^d(A^{d - i})$ for $1 \leq i \leq d$.  Since \eqref{I1acyclic} is an $S^d$-acyclic resolution of $\pi$, we have 
\begin{equation}\label{cohofY}
h^i(Y^\bullet) \cong \textnormal{L}_{d - i}S^d(\pi) \cong S^i(\pi). 
\end{equation}

Now define a first quadrant double complex $D := D^{\bullet,\bullet}$ by
$$D^{i,j} := \textnormal{C}^i(I_1,Y^j).$$
We multiply the differentials $D^{i,j} \longrightarrow D^{i, j + 1}$ by $(-1)^i$, so that $d_{\textnormal{vert}}\circ d_{\textnormal{hor}} + d_{\textnormal{hor}} \circ d_{\textnormal{vert}} = 0$.

\vspace{5pt}

\subsubsection{Spectral Sequence --- I}

Let $\textnormal{Tot}(D)^\bullet$ denote the direct sum totalization of the double complex $D$.  We have the following decreasing filtration by rows:
$$F^r_{\textnormal{row}}\textnormal{Tot}(D)^s = \bigoplus_{t \geq r} D^{s - t , t}.$$
Since $D$ is a first quadrant double complex, we obtain an associated convergent spectral sequence:
\begin{equation}\label{ss1}
E_2^{i,j} = h^i_{\textnormal{vert}}(h^{j,\bullet}_{\textnormal{hor}}(D)) \Longrightarrow h^{i + j}(\textnormal{Tot}(D)^\bullet)
\end{equation}

We examine the terms in the spectral sequence.  Note first that we have $h^{j,i}_{\textnormal{hor}}(D) = \textnormal{H}^j(I_1,Y^i)$.

\begin{lemma}\label{lem-acyclic2}
\hfill
\begin{enumerate}[(a)]
\item For $i \geq 0$, the $G$-representation $S^d(A^i)$ is $I_1$-acyclic. \label{lem-acyclic2-1}
\item For $i \geq d$, the $G$-representation $S^d(B^i)$ is $I_1$-acyclic. \label{lem-acyclic2-2}
\end{enumerate}
\end{lemma}

\begin{proof}
Let $\cK$ denote the set of compact open normal subgroups of $I_1$.  Since $\cK$ is cofinal in the set of all compact open subgroups of $G$, we have a canonical $I_1$-equivariant isomorphism
$$S^d(\tau) \cong \varinjlim_{K \in \cK} \Ext^d_{\Lambda(K)}(C, \tau^\vee)$$
for $\tau \in \mathfrak{Rep}^\infty(G)$.  By normality, we see that each $\Ext^d_{\Lambda(K)}(C, \tau^\vee)$ is a smooth $I_1/K$-representation.  Therefore, by \cite[\S I.2.2, Prop. 8]{serre:galoiscoh} we have
$$\textnormal{H}^j\big(I_1, S^d(\tau)\big) \cong \varinjlim_{K \in \cK} \textnormal{H}^j\big(I_1/K, \Ext^d_{\Lambda(K)}(C,\tau^\vee)\big).$$

(a)  Suppose now that $\tau = A^i$.  In order to prove the claim, it suffices to show $\textnormal{H}^j(I_1/K, \Ext^d_{\Lambda(K)}(C,A^{i,\vee})) = 0$ for all $j \geq 1$ and all $K \in \cK$.  By Lemma \ref{props}\eqref{props-2}, we have
$$\textnormal{H}^j\big(I_1/K, \Ext^d_{\Lambda(K)}(C,A^{i,\vee})\big) \cong \textnormal{H}^j\big(I_1/K, \textnormal{H}^d_{\textnormal{cts}}(K,A^{i,\vee})\big).$$
By \cite[Thm. 4.2.6]{symondsweigel}, we have a Hochschild--Serre spectral sequence for continuous cohomology
$$\textnormal{H}^j\big(I_1/K, \textnormal{H}^k_{\textnormal{cts}}(K,A^{i,\vee})\big) \Longrightarrow \textnormal{H}^{j + k}_{\textnormal{cts}}(I_1,A^{i,\vee}).$$
By Lemma \ref{props}\eqref{props-2} and Corollary \ref{cor-coefface}, we have $\textnormal{H}^k_{\textnormal{cts}}(K,A^{i,\vee}) = 0$ for $k \neq d$.  Therefore the spectral sequence above collapses to give $\textnormal{H}^j(I_1/K, \textnormal{H}^d_{\textnormal{cts}}(K,A^{i,\vee})) \cong \textnormal{H}^{j + d}_{\textnormal{cts}}(I_1,A^{i,\vee})$.  Applying Lemma \ref{props}\eqref{props-2} and Corollary \ref{cor-coefface} once more gives $\textnormal{H}^j(I_1/K, \textnormal{H}^d_{\textnormal{cts}}(K,A^{i,\vee})) = 0$ for $j \geq 1$.

(b)   Assume that $i\geq d$, and consider the exact sequence 
$$0 \longrightarrow B^{i} \longrightarrow A^i \longrightarrow B^{i + 1} \longrightarrow 0.$$
Applying $S^d$ and using Lemma \ref{lem-acyclic}\eqref{lem-acyclic-1}, we obtain a short exact sequence
$$0 \longrightarrow S^d(B^{i + 1}) \longrightarrow S^d(A^i) \longrightarrow S^d(B^i) \longrightarrow 0.$$
By examining the long exact sequence obtained by applying $\textnormal{H}^0(I_1,-)$, and using part \eqref{lem-acyclic2-1}, we get $\textnormal{H}^j(I_1, S^d(B^i)) \cong \textnormal{H}^{j + 1}(I_1,S^d(B^{i + 1}))$ for $j \geq 1$.  If $1 \leq j \leq d$, then by applying this isomorphism inductively we obtain 
$$\textnormal{H}^j\big(I_1, S^d(B^i)\big) \cong \textnormal{H}^{d + 1}\big(I_1,S^d(B^{i + d + 1 - j})\big) = 0,$$
since the cohomological dimension of $I_1$ is $d$ (by torsion-freeness).  On the other hand, if $j \geq d + 1$, then we obtain $\textnormal{H}^j(I_1, S^d(B^i)) = 0$ straight away.  
\end{proof}

Applying the above lemma, we obtain
$$h^{j,i}_{\textnormal{hor}}(D) = \textnormal{H}^j(I_1,Y^i) = \begin{cases} \textnormal{H}^0(I_1, Y^i) & \textnormal{if}~ j = 0,\\ 0 & \textnormal{if}~ j > 0. \end{cases}$$
Our next task will be to identify the term $\textnormal{H}^0(I_1, Y^i)$.

\begin{lemma}\label{heckeisom}
Suppose $\tau = A^i$ with $i \geq 0$ or $\tau = B^i$ with $i \geq d$.  We then have an isomorphism of $C$-vector spaces
$$S^d(\tau)^{I_1} \cong (\tau^{I_1})^\vee.$$
\end{lemma}

\begin{proof}

Let $\cK$ be as in the proof of Lemma \ref{lem-acyclic2}.  We begin with the following claim.

\begin{claim}
Let $K$ and $K'$ be two subgroups in $\cK$ such that $K' \subset K$.
\begin{enumerate}[(a)]
\item For $i \geq 0$, the restriction map $\textnormal{res}^K_{K'}: \textnormal{Ext}^d_{\Lambda(K)}(C,A^{i,\vee})  \longrightarrow \textnormal{Ext}^d_{\Lambda(K')}(C,A^{i,\vee})$ is injective.  
\item For $i \geq d$, the restriction map $\textnormal{res}^K_{K'}: \textnormal{Ext}^d_{\Lambda(K)}(C,B^{i,\vee})  \longrightarrow \textnormal{Ext}^d_{\Lambda(K')}(C,B^{i,\vee})$ is injective.  
\end{enumerate}
\end{claim}

\begin{proof}[Proof of claim]

(a) Consider first the restriction map
$$\textnormal{res}^K_{K'}: \textnormal{Ext}^d_{\Lambda(K)}(C,\Lambda(K))  \longrightarrow \textnormal{Ext}^d_{\Lambda(K')}(C,\Lambda(K)) \cong \textnormal{Ext}^d_{\Lambda(K')}(C,\Lambda(K'))^{\oplus [K:K']}.$$
Applying Lemma \ref{props}\eqref{props-3}, this becomes the diagonal map
$$C \longrightarrow C^{\oplus [K:K']}.$$
Therefore $\textnormal{res}^K_{K'}: \textnormal{Ext}^d_{\Lambda(K)}(C,\Lambda(K))  \longrightarrow \textnormal{Ext}^d_{\Lambda(K')}(C,\Lambda(K))$ is injective.

Since $A^i$ is an injective $G$-representation, the pseudocompact $\Lambda(K)$-module $A^{i,\vee}|_{\Lambda(K)}$ is projective and therefore it is a direct summand of $\prod_{a\in \sA} \Lambda(K)$ for some index set $\sA$ (cf. proof of Lemma \ref{coefface}).  By the previous paragraph, the restriction map
$$\textnormal{res}^K_{K'}: \prod_{a\in \sA} \Ext^d_{\Lambda(K)}(C,\Lambda(K)) \longrightarrow \prod_{a\in \sA} \Ext^d_{\Lambda(K')}(C,\Lambda(K))$$
is injective; since $\Ext^d_{\Lambda(K)}(C,A^{i,\vee})$ is a direct summand of $\Ext^d_{\Lambda(K)}(C,\prod_{a\in \sA}\Lambda(K)) \cong \prod_{a\in \sA} \Ext^d_{\Lambda(K)}(C,\Lambda(K))$, the result follows.

(b)  Recall that we have a short exact sequence of $G$-representations 
$$0 \longrightarrow B^i \longrightarrow A^i \longrightarrow B^{i + 1} \longrightarrow 0$$
for $i\geq 0$.  Applying Pontryagin duality gives a short exact sequence of pseudocompact $\Lambda(G)$-modules
$$0 \longrightarrow B^{i + 1, \vee} \longrightarrow A^{i, \vee} \longrightarrow B^{i,\vee} \longrightarrow 0.$$
This exact sequence induces long exact sequences which fit into a commutative diagram 
\begin{center}
\begin{tikzcd}[column sep = small]
 \ldots \ar[r] & \Ext^j_{\Lambda(K)}(C, A^{i, \vee}) \ar[r] \ar[d, "\textnormal{res}^K_{K'}"] & \Ext^j_{\Lambda(K)}(C, B^{i, \vee}) \ar[r] \ar[d, "\textnormal{res}^K_{K'}"] & \Ext^{j + 1}_{\Lambda(K)}(C, B^{i + 1, \vee}) \ar[r] \ar[d, "\textnormal{res}^K_{K'}"] & \Ext^{j + 1}_{\Lambda(K)}(C, A^{i, \vee}) \ar[r] \ar[d, "\textnormal{res}^K_{K'}"] & \ldots \\
  \ldots \ar[r] & \Ext^j_{\Lambda(K')}(C, A^{i, \vee}) \ar[r]  & \Ext^j_{\Lambda(K')}(C, B^{i, \vee}) \ar[r]  & \Ext^{j + 1}_{\Lambda(K')}(C, B^{i + 1, \vee}) \ar[r] & \Ext^{j + 1}_{\Lambda(K')}(C, A^{i, \vee}) \ar[r]  & \ldots 
\end{tikzcd}
\end{center}
Suppose $0 \leq j \leq d - 1$.  By Corollary \ref{cor-coefface}, the leftmost vertical map is an isomorphism (both the domain and codomain are $0$); by the previous point, the rightmost vertical map is injective.  A diagram chase then shows that injectivity of the map $\textnormal{res}^K_{K'}: \Ext^j_{\Lambda(K)}(C,B^{i,\vee}) \longrightarrow \Ext^j_{\Lambda(K')}(C,B^{i,\vee})$ implies injectivity of the map $\textnormal{res}^K_{K'}: \Ext^{j + 1}_{\Lambda(K)}(C,B^{i + 1,\vee}) \longrightarrow \Ext^{j + 1}_{\Lambda(K')}(C,B^{i + 1,\vee})$.  Proceeding by induction from the base case $\textnormal{res}^K_{K'}: \Hom_{\Lambda(K)}(C,B^{i,\vee}) \longhookrightarrow \Hom_{\Lambda(K')}(C,B^{i,\vee})$ gives the desired result.
\end{proof}

Suppose now that $\tau = A^i$ with $i \geq 0$ or $\tau = B^i$ with $i \geq d$.  The claim above implies that we have
$$S^d(\tau) \cong \varinjlim_{K\in \cK}\Ext^d_{\Lambda(K)}(C,\tau^\vee)$$
where all transition maps are injective.  The group $I_1$ acts on each $\Ext^d_{\Lambda(K)}(C,\tau^\vee)$, and we then have
\begin{eqnarray*}
S^d(\tau)^{I_1} & \cong & \Big(\varinjlim_{K\in \cK } \Ext^d_{\Lambda(K)}(C,\tau^\vee)\Big)^{I_1} \\
 & \cong & \varinjlim_{K\in \cK}  \Ext^d_{\Lambda(K)}(C,\tau^\vee)^{I_1}  \\ 
 & \cong & \Ext^d_{\Lambda(I_1)}(C,\tau^\vee)\\
 & \cong & \textnormal{H}_{\textnormal{cts}}^d(I_1,\tau^\vee).
\end{eqnarray*}
The last two isomorphisms may be obtained as follows: by Lemma \ref{props}\eqref{props-2}, the space $\Ext^d_{\Lambda(K)}(C,\tau^\vee)^{I_1}$ is isomorphic to $\textnormal{H}^0(I_1/K, \textnormal{H}^d_{\textnormal{cts}}(K,\tau^\vee))$.  Exactly as in the proof of Lemma \ref{lem-acyclic2}\eqref{lem-acyclic2-1}, we have
$$\textnormal{H}^0(I_1/K, \textnormal{H}^d_{\textnormal{cts}}(K,\tau^\vee)) \cong \textnormal{H}^d_{\textnormal{cts}}(I_1,\tau^\vee) \cong  \Ext^d_{\Lambda(I_1)}(C,\tau^\vee)$$
(for $\tau = B^i$, we use Remark \ref{Biacyclic}).  Since the restriction maps in the direct limit are injective, we obtain the desired isomorphisms.

To conclude, we apply \cite[Prop. 4.5.4]{symondsweigel}: we have a natural isomorphism
$$\Phi^d: \textnormal{H}^d_{\textnormal{cts}}(I_1, \tau^\vee) \stackrel{\sim}{\longrightarrow} \textnormal{H}^0(I_1,\tau)^\vee$$
(we will describe $\Phi^d$ explicitly in the proof of Proposition \ref{prop:orient} below).
\end{proof}

Next, we examine the $\cH$-action under the isomorphism $S^d(\tau)^{I_1} \cong (\tau^{I_1})^\vee$.  We thank Peter Schneider for pointing out the need for several results below, and in particular, for suggesting the statement and proof of Lemma \ref{orient-hom}.

We define
$$\fI := \varinjlim_n \varinjlim_K \textnormal{H}^d(K,\bbZ/p^n\bbZ)^\vee,$$ 
where $K$ runs over all compact open subgroups of $G$, and the transition maps in the inner direct limit are Pontryagin duals of corestrictions.  (In the above definition, the notation $\vee$ denotes the Pontryagin dual as in \cite{nsw:coh}: if $A$ is a $p$-power-torsion abelian group, we have $A^\vee = \Hom(A,\bbQ_p/\bbZ_p)$.)  Note that $\fI$ is a torsion $\bbZ_p$-module with a discrete action of the group $G$.

Suppose now that $K$ is a torsion-free open pro-$p$ subgroup of $G$.  By \cite[Cor. (1)]{serre:cohdim} and \cite[Thm. V.2.5.8]{lazard}, $K$ is a Poincar\'e group of dimension $d$.  Since the set of open normal subgroups of $K$ is cofinal in the set of compact open subgroups of $G$, according to \cite[\S III.4]{nsw:coh} the restriction of the $G$-action on $\fI$ to $K$ gives the dualizing module of $K$.  Consequently, using \cite[\S I.4.5, Prop. 30(b)]{serre:galoiscoh}, we see that for every torsion-free open pro-$p$ subgroup $K$ we have a canonically defined trace map
$$\textnormal{tr}_K : \textnormal{H}^d(K,\fI) \stackrel{\sim}{\longrightarrow} \bbQ_p/\bbZ_p.$$

\begin{lemma}
\label{tr-cor}
Let $K' \subset K$ be an inclusion of torsion-free open pro-$p$ subgroups of $G$.  Then the composition
$$\textnormal{H}^d(K',\fI) \xrightarrow{\textnormal{cor}_{K'}^K} \textnormal{H}^d(K,\fI) \xrightarrow{\textnormal{tr}_K} \bbQ_p/\bbZ_p$$
is equal to $\textnormal{tr}_{K'}$.  In particular, the map $\textnormal{cor}_{K'}^K: \textnormal{H}^d(K',\fI) \longrightarrow \textnormal{H}^d(K,\fI)$ is an isomorphism.  
\end{lemma}

\begin{proof}
According to the remark following \cite[\S I.3.5, Prop. 18]{serre:galoiscoh}, we have a commutative diagram
\begin{center}
\begin{tikzcd}[column sep = large, row sep = large]
\Hom_K(\fI[p^n], \fI) \ar[r, "\sim"] \ar[d, "\textnormal{res}^K_{K'}"'] & \textnormal{H}^d(K, \fI[p^n])^\vee \ar[d, "(\textnormal{cor}_{K'}^K)^\vee"]\\
\Hom_{K'}(\fI[p^n], \fI) \ar[r, "\sim"]  & \textnormal{H}^d(K', \fI[p^n])^\vee 
\end{tikzcd}
\end{center}
where the horizontal maps are those coming from the representability property of $\fI$.  By step (7) in the proof of Proposition 30 in \S I.4.5 of \emph{op. cit.}, we have $\Hom_K(\fI[p^n], \fI) \cong \Hom_{K'}(\fI[p^n], \fI) \cong \bbZ_p/p^n\bbZ_p$, and thus both $\textnormal{res}^K_{K'}$ and $(\textnormal{cor}_{K'}^K)^\vee$ in the diagram above are isomorphisms.  Passing to the limit over $n$, we obtain a commutative diagram
\begin{center}
\begin{tikzcd}[column sep = large, row sep = large]
\Hom_K(\fI, \fI) \ar[r, "\sim"] \ar[d, "\textnormal{res}^K_{K'}"', "{\rotatebox{90}{$\sim$}}"] & \textnormal{H}^d(K, \fI)^\vee \ar[d, "(\textnormal{cor}_{K'}^K)^\vee", "{\rotatebox{90}{$\sim$}}"']\\
\Hom_{K'}(\fI, \fI) \ar[r, "\sim"]  & \textnormal{H}^d(K', \fI)^\vee 
\end{tikzcd}
\end{center}
The result now follows, as $\textnormal{tr}_K$ and $\textnormal{tr}_{K'}$ are the images of the respective identity maps under the horizontal arrows.
\end{proof}

In order to alleviate notation, we use the following shorthand:  for $g \in G$, we define $I_g := I_1 \cap gI_1 g^{-1}$.

\begin{defn}
Given $g\in G$, we have a chain of isomorphisms
$$\bbQ_p/\bbZ_p \xrightarrow{\textnormal{tr}_{I_{g^{-1}}}^{-1}} \textnormal{H}^d(I_{g^{-1}},\fI) \stackrel{g_*}{\longrightarrow} \textnormal{H}^d(I_g,\fI) \xrightarrow{\textnormal{tr}_{I_g}}\bbQ_p/\bbZ_p.$$
We denote by $\Psi(g) \in \bbZ_p^\times$ the corresponding element of $\bbZ_p^\times = \textnormal{Aut}(\bbQ_p/\bbZ_p)$.  
\end{defn}

\begin{lemma}
\label{orient-hom}
The map 
\begin{eqnarray*}
G & \longrightarrow & \bbZ_p^\times \\
g & \longmapsto & \Psi(g)
\end{eqnarray*}
is a homomorphism, which is trivial on $I_1$.
\end{lemma}

\begin{proof}
Let $g,h\in G$.  Lemma \ref{tr-cor} implies that we have the following commutative diagrams (we omit decorations on the corestriction maps for readability):
\begin{center}
\begin{tikzcd}[column sep = large, row sep = large]
\textnormal{H}^d(I_{h^{-1}g^{-1}}, \fI)  \ar[d, "h_*"', "{\rotatebox{90}{$\sim$}}"] & \textnormal{H}^d(I_{h^{-1}} \cap I_{h^{-1}g^{-1}}, \fI)  \ar[r, "\textnormal{cor}", "\sim"'] \ar[l, "\textnormal{cor}"', "\sim"] \ar[d, "h_*"', "{\rotatebox{90}{$\sim$}}"] &  \textnormal{H}^d(I_{h^{-1}}, \fI) \ar[d, "h_*"', "{\rotatebox{90}{$\sim$}}"] \\
\textnormal{H}^d(hI_1h^{-1} \cap g^{-1}I_1g, \fI)  & \textnormal{H}^d(I_h \cap I_{g^{-1}}, \fI)  \ar[r, "\textnormal{cor}", "\sim"'] \ar[l, "\textnormal{cor}"', "\sim"] &  \textnormal{H}^d(I_{h}, \fI) \\
\textnormal{H}^d(hI_1h^{-1} \cap g^{-1}I_1g, \fI) \ar[d, "g_*"', "{\rotatebox{90}{$\sim$}}"] & \textnormal{H}^d(I_h \cap I_{g^{-1}}, \fI)  \ar[r, "\textnormal{cor}", "\sim"'] \ar[l, "\textnormal{cor}"', "\sim"] \ar[d, "g_*"', "{\rotatebox{90}{$\sim$}}"] &  \textnormal{H}^d(I_{g^{-1}}, \fI) \ar[d, "g_*"', "{\rotatebox{90}{$\sim$}}"] \\
\textnormal{H}^d(I_{gh}, \fI)  & \textnormal{H}^d(I_{gh} \cap I_g, \fI)  \ar[r, "\textnormal{cor}", "\sim"'] \ar[l, "\textnormal{cor}"', "\sim"] &  \textnormal{H}^d(I_{g}, \fI)
\end{tikzcd}
\end{center}
Thus, applying Lemma \ref{tr-cor} and using commutativity of the above diagrams, we get
\begin{eqnarray*}
\Psi(h) & = & \textnormal{tr}_{I_h} \circ h_* \circ \textnormal{tr}_{I_{h^{-1}}}^{-1} \\
 & = & \textnormal{tr}_{I_h \cap I_{g^{-1}}} \circ h_* \circ \textnormal{tr}_{I_{h^{-1}} \cap I_{h^{-1}g^{-1}}}^{-1} \\
 & = & \textnormal{tr}_{hI_1h^{-1} \cap g^{-1}I_1g} \circ h_* \circ \textnormal{tr}_{I_{h^{-1}g^{-1}}}^{-1} \\
\Psi(g) & = & \textnormal{tr}_{I_g} \circ g_* \circ \textnormal{tr}_{I_{g^{-1}}}^{-1} \\
& = & \textnormal{tr}_{I_{gh} \cap I_g} \circ g_* \circ \textnormal{tr}_{I_h \cap I_{g^{-1}}}^{-1} \\
& = & \textnormal{tr}_{I_{gh}} \circ g_* \circ \textnormal{tr}_{hI_1h^{-1} \cap g^{-1}I_1g}^{-1}.
\end{eqnarray*}
This implies 
$$\Psi(g)\Psi(h) = \textnormal{tr}_{I_{gh}} \circ (gh)_* \circ \circ \textnormal{tr}_{I_{h^{-1}g^{-1}}}^{-1}  = \Psi(gh),$$
which shows $\Psi$ is a homomorphism.  That $\Psi$ is trivial on $I_1$ is a standard fact about group cohomology.
\end{proof}

\begin{defn}
\begin{enumerate}
\item We define 
$$\psi : G \longrightarrow C^\times$$ 
to be the mod-$p$ reduction of the character $\Psi$.
\item Let $\ffi_C := \fI[p]\otimes_{\bbF_p}C$, and note that $\ffi_C$ is a one-dimensional $C$-vector space with an action of $G$.  We let 
$$\lambda : G \longrightarrow C^\times$$
denote character which gives the action of $G$ on $\ffi_C$.  (In particular, note that $\lambda$ is trivial on every pro-$p$ subgroup of $G$.)
\item We define the \textbf{mod-$p$ orientation character of $G$ (relative to $I_1$)}, denoted
$$\xi: G \longrightarrow C^\times,$$
by $\xi := \psi\lambda^{-1}$.  
\end{enumerate}
\end{defn}

\begin{rmk}
If the group $\bG$ is split over $F$, then the character $\xi$ is trivial; for irreducible root systems, this was shown in the proofs of \cite[Thms. 7.1, 7.2]{koziol:h1triv}, and in general in \cite[Prop. 7.16]{ollivierschneider:ext}.  The character $\xi$ is also trivial if $\bG$ is semisimple and simply connected.  
\end{rmk}

\begin{defn}
Suppose $\fm$ is an $\cH$-module.  We let $\fm(\xi)$ denote the $\cH$-module $\fm$ twisted by $\xi$: it has the same underlying vector space as $\fm$, with the operator $\T_g$ acting by $\xi(g^{-1})\T_g$.  
\end{defn}

We may now describe the aforementioned Hecke action.

\begin{propn}
\label{prop:orient}
Suppose $\tau = A^i$ with $i \geq 0$ or $\tau = B^i$ with $i \geq d$.  Then the vector space isomorphism 
$$S^d(\tau)^{I_1} \cong (\tau^{I_1})^\vee$$
upgrades to an isomorphism of $\cH$-modules
$$S^d(\tau)^{I_1} \cong (\tau^{I_1})^\vee(\xi).$$
\end{propn}

\begin{proof}
Recall that the isomorphism $S^d(\tau)^{I_1} \cong (\tau^{I_1})^\vee$ is given as the composition of $S^d(\tau)^{I_1} \stackrel{\sim}{\longrightarrow} \textnormal{H}_{\textnormal{cts}}^d(I_1,\tau^\vee)$ with $\Phi^d: \textnormal{H}^d_{\textnormal{cts}}(I_1, \tau^\vee) \stackrel{\sim}{\longrightarrow} \textnormal{H}^0(I_1,\tau)^\vee$.  The injectivity in the claim of Lemma \ref{heckeisom} and the description of the $G$-action on $S^d(\tau)$ imply that the isomorphism $S^d(\tau)^{I_1} \stackrel{\sim}{\longrightarrow} \textnormal{H}_{\textnormal{cts}}^d(I_1,\tau^\vee)$ is actually $\cH$-equivariant.

It suffices to analyze the map $\Phi^d$ (and its analogs for open subgroups of $I_1$).  The stated result will follow from the commutativity of the following three diagrams, where $g \in G$:

\begin{center}
\begin{tikzcd}[column sep = large, row sep = large]
\textnormal{H}^d_{\textnormal{cts}}(I_1, \tau^\vee) \ar[r, "\Phi^d", "\sim"'] \ar[d, "\textnormal{res}^{I_1}_{I_g}"'] & \textnormal{H}^0(I_1, \tau)^\vee \ar[d, "(\textnormal{cor}^{I_1}_{I_g})^\vee"] \\
\textnormal{H}^d_{\textnormal{cts}}(I_g, \tau^\vee) \ar[r, "\Phi^d", "\sim"'] & \textnormal{H}^0(I_g, \tau)^\vee \\
\textnormal{H}^d_{\textnormal{cts}}(I_g, \tau^\vee) \ar[r, "\Phi^d", "\sim"'] \ar[d, "g_*^{-1}"', "{\rotatebox{90}{$\sim$}}"]& \textnormal{H}^0(I_g, \tau)^\vee \ar[d, "\xi(g^{-1})\cdot (g_*)^\vee", "{\rotatebox{90}{$\sim$}}"']\\
\textnormal{H}^d_{\textnormal{cts}}(I_{g^{-1}}, \tau^\vee) \ar[r, "\Phi^d", "\sim"']  & \textnormal{H}^0(I_{g^{-1}}, \tau)^\vee \\
\textnormal{H}^d_{\textnormal{cts}}(I_{g^{-1}}, \tau^\vee) \ar[r, "\Phi^d", "\sim"'] \ar[d, "\textnormal{cor}^{I_1}_{I_{g^{-1}}}"'] & \textnormal{H}^0(I_{g^{-1}}, \tau)^\vee \ar[d, "(\textnormal{res}^{I_1}_{I_{g^{-1}}})^\vee"] \\
\textnormal{H}^d_{\textnormal{cts}}(I_1, \tau^\vee) \ar[r, "\Phi^d", "\sim"'] & \textnormal{H}^0(I_1, \tau)^\vee \\
\end{tikzcd}
\end{center}

In order to prove the commutativity of the first and third diagrams, it suffices to assume $\tau$ is a finite $I_1$-module, and to replace $\tau^\vee$ with $\Hom(\tau,\fI)$ (this follows from \cite[Prop. 3.6.1]{symondsweigel}; note also that as representations of $I_1$, we have $\tau^\vee \cong \tau^\vee \otimes_C \ffi_C$). The claimed commutativity is then a consequence of the definition of $\Phi^d$ given in Section 4.5 of \emph{op. cit.}, the remark following \cite[\S I.3.5, Prop. 18]{serre:galoiscoh}, and its dual statement.

In order to prove the commutativity of the second diagram, fix a presentation $\tau = \varinjlim_i \tau_i$, where the $\tau_i$ are finite-dimensional $I_g$-subrepresentations of $\tau$.  Dualizing, we have $\tau^\vee \cong \varprojlim_i \tau_i^\vee$, and we let 
$$\textnormal{pr}_i: \textnormal{H}^d_{\textnormal{cts}}(I_g, \tau^\vee\otimes_C \ffi_C) \longrightarrow \textnormal{H}^d(I_g, \tau_i^\vee\otimes_C \ffi_C)$$
denote the map associated to the surjection $\tau^\vee \otimes_C \ffi_C \longtwoheadrightarrow \tau_i^\vee \otimes_C \ffi_C$.  Given $x_d \in \textnormal{H}^d_{\textnormal{cts}}(I_g, \tau^\vee \otimes_C \ffi_C)$ and $y_0 \in \textnormal{H}^0(I_g, \tau)$, there exists an index $i$ such that $y_0 \in \textnormal{H}^0(I_g, \tau_i)$.  The upper horizontal map $\Phi^d$ is then given by
$$\Phi^d(x_d)(y_0) = \textnormal{tr}_{I_g}\big(\textnormal{H}^d(\textnormal{ev})(\textnormal{pr}_i(x_d) \smile y_0)\big).$$

Applying $g^{-1}$ to the presentation above, we have that $\varinjlim_i g^{-1}. \tau_i$ is a presentation of $\tau$ by finite dimensional $I_{g^{-1}}$-subrepresentations.  We let
$$\textnormal{pr}_i': \textnormal{H}^d_{\textnormal{cts}}(I_{g^{-1}}, \tau^\vee \otimes_C \ffi_C) \longrightarrow \textnormal{H}^d(I_{g^{-1}}, (g^{-1}.\tau_i)^\vee \otimes_C \ffi_C)$$
denote the map associated to the surjection $\tau^\vee \otimes_C \ffi_C \longtwoheadrightarrow (g^{-1}.\tau_i)^\vee \otimes_C \ffi_C$.  The lower horizontal map $\Phi^d$ is defined analogously to the above, and we have $\textnormal{pr}_i' \circ g^{-1}_* = g^{-1}_* \circ \textnormal{pr}_i$ as maps $\textnormal{H}^d_{\textnormal{cts}}(I_{g}, \tau^\vee \otimes_C \ffi_C) \longrightarrow \textnormal{H}^d(I_{g^{-1}}, (g^{-1}.\tau_i)^\vee \otimes_C \ffi_C)$.

Now let $x_d \in \textnormal{H}^d_{\textnormal{cts}}(I_g, \tau^\vee \otimes_C \ffi_C), y_0 \in \textnormal{H}^0(I_{g^{-1}}, \tau)$, and choose $i$ so that $y_0 \in \textnormal{H}^0(I_{g^{-1}}, g^{-1}.\tau_i)$.  
We compute:
\begin{eqnarray*}
\Phi^d(g_*^{-1}x_d)(y_0) & = & \textnormal{tr}_{I_{g^{-1}}}\big(\textnormal{H}^d(\textnormal{ev})(\textnormal{pr}'_i(g_*^{-1}x_d) \smile y_0)\big) \\
 & = & \textnormal{tr}_{I_{g^{-1}}}\big(\textnormal{H}^d(\textnormal{ev})(g_*^{-1}(\textnormal{pr}_i(x_d) \smile g_*y_0))\big) \\
 & = & \textnormal{tr}_{I_{g^{-1}}}\big(g_*^{-1}\textnormal{H}^d(\textnormal{ev})(\textnormal{pr}_i(x_d) \smile g_*y_0)\big) \\
 & = & \Big(\textnormal{tr}_{I_{g^{-1}}} \circ g_*^{-1} \circ \textnormal{tr}_{I_g}^{-1}\Big) \circ \textnormal{tr}_{I_g}\big(\textnormal{H}^d(\textnormal{ev})(\textnormal{pr}_i(x_d) \smile g_*y_0)\big) \\
 & = & \psi(g^{-1})\Phi^d(x_d)(g_*y_0)\\
 & = & \big(\psi(g^{-1})\cdot g_*^\vee  \circ \Phi^d\big)(x_d)(y_0).
\end{eqnarray*}
This implies that the following diagram is commutative:
\begin{center}
\begin{tikzcd}[column sep = large, row sep = large]
\textnormal{H}^d_{\textnormal{cts}}(I_g, \tau^\vee \otimes_C \ffi_C) \ar[r, "\Phi^d", "\sim"'] \ar[d, "g_*^{-1}"', "{\rotatebox{90}{$\sim$}}"]& \textnormal{H}^0(I_g, \tau)^\vee \ar[d, "\psi(g^{-1})\cdot (g_*)^\vee", "{\rotatebox{90}{$\sim$}}"']\\
\textnormal{H}^d_{\textnormal{cts}}(I_{g^{-1}}, \tau^\vee \otimes_C \ffi_C) \ar[r, "\Phi^d", "\sim"']  & \textnormal{H}^0(I_{g^{-1}}, \tau)^\vee \\
\end{tikzcd}
\end{center}
Twisting by $\lambda^{-1}$ gives the claim.
\end{proof}

Applying the proposition, we get
$$h^{j,i}_{\textnormal{hor}}(D) \cong \textnormal{H}^j(I_1,Y^i) \cong \begin{cases} \big((B^{d})^{I_1}\big)^\vee(\xi) & \textnormal{if}~ j = 0, i = 0,\\ \big((A^{d - i})^{I_1}\big)^\vee(\xi) & \textnormal{if}~ j = 0, i > 0,\\ 0 & \textnormal{if}~ j > 0. \end{cases}$$
Finally, since \eqref{I1acyclic} is an $I_1$-acyclic resolution of $\pi$ and taking the $C$-linear dual is exact, we obtain
$$h^i_{\textnormal{vert}}(h^{j,\bullet}_{\textnormal{hor}}(D)) \cong \begin{cases} \textnormal{H}^{d - i}(I_1,\pi)^\vee(\xi) & \textnormal{if}~ j = 0, \\ 0 & \textnormal{if}~ j > 0. \end{cases}$$

We therefore see that the spectral sequence \eqref{ss1} degenerates at the $E_2$ page, and we obtain
\begin{equation}\label{cohoftot}
h^n(\textnormal{Tot}(D)^\bullet) \cong \textnormal{H}^{d - n}(I_1, \pi)^\vee(\xi).
\end{equation}

\vspace{5pt}

\subsubsection{Spectral Sequence --- II}

We now consider the decreasing filtration on $\textnormal{Tot}(D)^\bullet$ by columns:
$$F^r_{\textnormal{col}}\textnormal{Tot}(D)^s = \bigoplus_{t \geq r} D^{t, s - t}.$$
Once again, we obtain a convergent spectral sequence:
\begin{equation}\label{ss2}
E_2^{i,j} = h^i_{\textnormal{hor}}(h^{\bullet,j}_{\textnormal{vert}}(D)) \Longrightarrow h^{i + j}(\textnormal{Tot}(D)^\bullet)
\end{equation}

We examine the terms in the spectral sequence.  Since the functor $\textnormal{C}^i(I_1,-)$ is exact, equation \eqref{cohofY} implies $h^{i,j}_{\textnormal{vert}}(D) = \textnormal{C}^i(I_1,S^j(\pi))$.  Therefore, we have $h^i_{\textnormal{hor}}(h^{\bullet,j}_{\textnormal{vert}}(D)) \cong \textnormal{H}^i(I_1,S^j(\pi))$.  Combining this with the isomorphism \eqref{cohoftot}, the spectral sequence \eqref{ss2} becomes
$$\textnormal{H}^i\big(I_1,S^j(\pi)\big) \Longrightarrow \textnormal{H}^{d - i - j}(I_1,\pi)^\vee(\xi).$$

\begin{propn}
The spectral sequence 
\begin{equation}\label{mainss}
\textnormal{H}^i\big( I_1,S^j(\pi)\big) \Longrightarrow \textnormal{H}^{d - i - j}(I_1,\pi)^\vee(\xi).
\end{equation}
is $\cH$-equivariant.
\end{propn}

\begin{proof}
Given a compact open subgroup $K$ of $I_1$, let $D_K := D_K^{\bullet,\bullet}$ denote the double complex defined by
$$D_K^{i,j} := \textnormal{C}^i(K,Y^j),$$
with the differentials normalized as above.  For $g\in G$, the following maps (defined on the level of cocycles)
\begin{alignat*}{3}
\textnormal{res}^{I_1}_{I_g} :  \quad &&  D_{I_1}^{\bullet,\bullet}  &  \longrightarrow  && ~ D_{I_g}^{\bullet,\bullet} \\
g_*^{-1}: \quad && D_{I_g}^{\bullet,\bullet}  &  \longrightarrow  && ~ D_{I_{g^{-1}}}^{\bullet,\bullet} \\
\textnormal{cor}^{I_1}_{I_{g^{-1}}} :  \quad &&  D_{I_{g^{-1}}}^{\bullet,\bullet}  &  \longrightarrow  && ~ D_{I_1}^{\bullet, \bullet}
\end{alignat*}
commute with both the vertical and horizontal differentials.  By taking direct sums, we obtain maps between totalizations
\begin{alignat*}{3}
\textnormal{res}^{I_1}_{I_g} : \quad  &&  \textnormal{Tot}(D_{I_1})^\bullet  &  \longrightarrow  && ~ \textnormal{Tot}(D_{I_g})^\bullet \\
g_*^{-1}: \quad &&  \textnormal{Tot}(D_{I_g})^\bullet  &  \longrightarrow  && ~ \textnormal{Tot}(D_{I_{g^{-1}}})^\bullet \\
\textnormal{cor}^{I_1}_{I_{g^{-1}}} : \quad &&  \textnormal{Tot}(D_{I_{g^{-1}}})^\bullet  &  \longrightarrow  && ~ \textnormal{Tot}(D_{I_1})^\bullet
\end{alignat*}
which commute with the differential.  Consequently, the above maps induce morphisms between $h^n(\textnormal{Tot}(D_K)^\bullet)$ for varying $K$.  Furthermore, the above maps preserve both the filtration by rows $F^\bullet_{\textnormal{row}}$ and the filtration by columns $F^\bullet_{\textnormal{col}}$ on $\textnormal{Tot}(D_K)^\bullet$.  Exactly as above for the group $I_1$, the spectral sequences associated to these filtrations degenerate at the $E_{\max\{2, d + 1\}}$ page, and converge to $h^n(\textnormal{Tot}(D_K)^\bullet)$.  Since the differentials in the spectral sequences associated to $F^\bullet_{\textnormal{row}}$ and $F^\bullet_{\textnormal{col}}$ are constructed from the differential on $\textnormal{Tot}(D_K)^\bullet$, we conclude that the maps $\textnormal{res}^{I_1}_{I_g}, g_*^{-1}$, and $\textnormal{cor}^{I_1}_{I_{g^{-1}}}$ induce morphisms between the spectral sequences associated to each filtration.

Consider first the filtration $F^\bullet_{\textnormal{row}}$ on $\textnormal{Tot}(D_{I_1})^\bullet$.  The $E_1$ page of associated spectral sequence is given by
$$E_1^{i,j} = \textnormal{H}^j(I_1,Y^i),$$
with limit $h^{i + j}(\textnormal{Tot}(D_{I_1})^\bullet)$.  The composition $\textnormal{cor}^{I_1}_{I_{g^{-1}}} \circ g_*^{-1} \circ \textnormal{res}^{I_1}_{I_g}$ gives the action of the Hecke operator $\T_g$ on $\textnormal{H}^j(I_1,Y^i)$, and this extends to an action of $\cH$ on $\textnormal{H}^j(I_1,Y^i)$.  Therefore, the $E_1$ page consists of $\cH$-modules with $\cH$-equivariant differentials, and the same is then true of all subsequent pages.  Since the maps $\textnormal{res}^{I_1}_{I_g}, g_*^{-1}$, and $\textnormal{cor}^{I_1}_{I_{g^{-1}}}$ give morphisms of spectral sequences, we conclude that $E_2^{n,0} = E_\infty^{n,0} \cong h^n(\textnormal{Tot}(D_{I_1})^\bullet)$ is an $\cH$-module, and Proposition \ref{prop:orient} implies we have a $\cH$-equivariant isomorphism $h^n(\textnormal{Tot}(D_{I_1})^\bullet) \cong \textnormal{H}^{d - n}(I_1,\pi)^\vee(\xi)$.

Consider now the filtration $F^\bullet_{\textnormal{col}}$ on $\textnormal{Tot}(D_{I_1})^\bullet$.  The $E_2$ page of associated spectral sequence is given by
$$E_2^{i,j} = \textnormal{H}^i(I_1, S^j(\pi)),$$
with limit $h^{i + j}(\textnormal{Tot}(D_{I_1})^\bullet)$.  As above, the action of $\textnormal{cor}^{I_1}_{I_{g^{-1}}} \circ g_*^{-1} \circ \textnormal{res}^{I_1}_{I_g}$ on $\textnormal{H}^i(I_1, S^j(\pi))$ extends to an action of the entire Hecke algebra $\cH$, and all subsequent pages of the spectral sequence consist of $\cH$-modules with $\cH$-equivariant differentials.  Since the maps $\textnormal{res}^{I_1}_{I_g}, g_*^{-1}$, and $\textnormal{cor}^{I_1}_{I_{g^{-1}}}$ give morphisms of spectral sequences, we obtain a spectral sequence of $\cH$-modules
$$\textnormal{H}^i(I_1, S^j(\pi)) \Longrightarrow h^{i + j}(\textnormal{Tot}(D_{I_1})^\bullet);$$
composing with the $\cH$-equivariant isomorphism $h^{i + j}(\textnormal{Tot}(D_{I_1})^\bullet) \cong \textnormal{H}^{d - i - j}(I_1,\pi)^\vee(\xi)$ concludes the proof.  
\end{proof}

\begin{rmk}
The existence of the above spectral sequence may also be proved without using continuous cohomology.  This approach will appear in forthcoming work of Schneider--Sorensen.  
\end{rmk}

\vspace{10pt}

\section{Examples}
\label{sec:examples}

We now compute some examples using the spectral sequences above.  As in previous sections, we let $\bG$ denote a connected reductive group over $F$, and $I_1$ a pro-$p$-Iwahori subgroup of $G = \bG(F)$.  We assume throughout that $I_1$ is torsion-free, and set $d := \dim_{\bbQ_p}(G) = \textnormal{cd}(I_1)$.

\vspace{5pt}

\subsection{Finite-dimensional representations}\label{findim}

Let $\pi$ denote a finite-dimensional representation of $G$.  Then the Pontryagin (or $C$-linear) dual $\pi^\vee$ is once again a smooth $G$-representation, and we have
$$S^j(\pi) = \begin{cases}\pi^\vee & \textnormal{if}~ j = 0, \\ 0 & \textnormal{if}~ j > 0,\end{cases}$$
(cf. \cite[Cor. 3.16]{kohlhaase:duality}).  The spectral sequence \eqref{mainss} thus gives an isomorphism of $\cH$-modules
$$\textnormal{H}^i(I_1, \pi^\vee) \cong \textnormal{H}^{d - i}(I_1, \pi)^\vee(\xi).$$
In particular, letting $\pi \cong \mathbf{1}_G$ denote the trivial $G$-representation, we obtain isomorphisms of $\cH$-modules
$$\textnormal{H}^i(I_1, \mathbf{1}_G) \cong \textnormal{H}^{d - i}(I_1, \mathbf{1}_G)^\vee(\xi)$$
and
$$\textnormal{H}^d(I_1, \mathbf{1}_G) \cong \textnormal{H}^{0}(I_1, \mathbf{1}_G)^\vee(\xi) \cong \chi_{\textnormal{triv}}^\vee(\xi) \cong \chi_{\textnormal{triv}}(\xi),$$
where $\chi_{\triv}$ denotes the trivial character of $\cH$ (see \cite[\S 2.5.4]{olliviervigneras}).

On the other hand, if $\pi$ is an irreducible admissible \emph{infinite}-dimensional representation of $G$, then \cite[Prop. 3.9]{kohlhaase:duality} shows that $S^0(\pi) = 0$ (see also \cite[Thm. 6.4]{ahv:pro-p}).  The spectral sequence \eqref{mainss} then shows in this case that $\textnormal{H}^d(I_1, \pi) = 0$.

\vspace{5pt}

\subsection{Parabolic induction}
\label{parabindexs}

In this subsection we suppose $\bP$ is a rational parabolic subgroup of $\bG$ with rational Levi component $\bM$, and let $\chi: M \longrightarrow C^\times$ denote a smooth character of $M$, which we inflate to $P$.  We have $\chi^{I_{M,1}} = \chi$, so that $\chi$ inherits the structure of a right $\cH_M$-module.  We will use the letter $\chi$ to refer to both the character of $M$ and the associated $\cH_M$-module; the meaning should be clear from context.

We consider the parabolically induced representation $\pi := \Ind_P^G(\chi)$.  By \cite[Prop. 5.4]{kohlhaase:duality}, we have
$$S^j(\pi)  =  \begin{cases} \Ind_P^G(\chi^{-1}\chi_P) & \textnormal{if}~ j = \dim_{\bbQ_p}(G/P), \\ 0 & \textnormal{if}~ j \neq \dim_{\bbQ_p}(G/P),\end{cases}$$
where $\chi_P: P \longrightarrow C^\times$ denotes the dualizing character (see the remark below).  The spectral sequence \eqref{mainss} therefore collapses to give an isomorphism of $\cH$-modules
$$\textnormal{H}^i\big(I_1, \Ind_P^G(\chi^{-1}\chi_P)\big) \cong \textnormal{H}^{\dim_{\bbQ_p}(P) - i}\big(I_1, \Ind_P^G(\chi)\big)^\vee(\xi).$$
In particular, applying \cite[Prop. 4.4]{olliviervigneras} gives an isomorphism of $\cH$-modules
$$\textnormal{H}^{\dim_{\bbQ_p}(P)}\big(I_1, \Ind_P^G(\chi)\big) \cong \textnormal{H}^0\big(I_1, \Ind_P^G(\chi^{-1}\chi_P)\big)^\vee(\xi) \cong \Ind_{\cH_M}^{\cH}(\chi^{-1}\chi_P)^\vee(\xi).$$
Specializing further to the case $\bG = \bG\bL_{n/\bbQ_p}$ and $\bP = \bB$ the upper-triangular Borel subgroup, the results of \cite{koziol:h1ps} can be used to determine $\textnormal{H}^{\dim_{\bbQ_p}(B) - 1}(I_1, \Ind_B^G(\chi))$.

\begin{rmk}
\label{rmkdualizing}
The formula for $\chi_P$ given in \cite[Cor. 5.3]{kohlhaase:duality} seems to be inconsistent with other results in the literature.  For example, suppose $p \geq 5$, $\bG = \bG\bL_{2/\bbQ_p}$, and $\bB = \bT\ltimes \bU$ is the upper-triangular Borel subgroup of $\bG$.  Let $\overline{\alpha}:T \longrightarrow C^\times$ denote the character $\sm{xp^a}{0}{0}{yp^b} \longmapsto \overline{xy^{-1}}$, where $x,y\in \bbZ_p^\times$, $a,b\in \bbZ$.  Corollary 5.3 of \emph{op. cit.} would then imply that we have a nonsplit extension
$$0 \longrightarrow \mathbf{1}_G \longrightarrow S^1(\textnormal{St}_G) \longrightarrow \Ind_B^G(\overline{\alpha}^{-1}) \longrightarrow 0,$$
where $\textnormal{St}_G$ denotes the Steinberg representation of $G$ (cf. \emph{op. cit}, Proposition 5.7).  However, \cite[Prop. 4.3.13]{emerton:ordII} implies that $\Ext^1_{G}(\Ind_B^G(\overline{\alpha}^{-1}), \mathbf{1}_G) = 0$, noting that the results of \cite{emerton:ordII} are stated in terms of the lower triangular Borel subgroup $\bB^-$.  (The existence of such a nonsplit extension would also imply that we have an injection $\Ind_{\cH_T}^{\cH}(\overline{\alpha}^{-1}) \longhookrightarrow \textnormal{H}^1(I_1, \mathbf{1}_G)$, which contradicts \cite[Thm. 6.4(a)]{koziol:h1triv}.)  We shall therefore assume that, in the $\bbQ_p$-split case, the dualizing character is given by 
\begin{equation}\label{dualcharacter}
\chi_P(mn) = \prod_{\alpha \in \Phi^+ \smallsetminus \Phi_M^+} \alpha(m) |\alpha(m)|_p,
\end{equation}
where $mn \in P = M \ltimes N$.  (We are being sloppy with notation on the right hand side: the expression ``$\prod_{\alpha \in \Phi^+ \smallsetminus \Phi_M^+} \alpha(m) |\alpha(m)|_p$'' denotes the unique extension to $M$ of the character 
$$T \ni t \longmapsto \prod_{\alpha \in \Phi^+ \smallsetminus \Phi_M^+} \alpha(t) |\alpha(t)|_p.\bigg)$$
\end{rmk}

\vspace{5pt}

\subsection{$\textrm{GL}_2(\bbQ_p)$}

We now assume $\bG = \bG\bL_{2/\bbQ_p}$ with $p \geq 5$, $\bB = \bT\ltimes \bU$ is the upper-triangular Borel subgroup, and $I_1$ is the ``upper-triangular mod $p$'' pro-$p$ Iwahori subgroup.  In this subsection we determine almost all $\textnormal{H}^i(I_1,\pi)$, where $\pi$ denotes an irreducible smooth $G$-representation.  Note that the assumption $p \geq 5$ implies $I_1$ is torsion-free, and therefore has cohomological dimension $\dim_{\bbQ_p}(G) = 4$.  Since $\textnormal{H}^i(I_1, \pi \otimes \chi\circ\det) \cong \textnormal{H}^i(I_1, \pi) \otimes \chi\circ\det$, we may twist our irreducible representation as is convenient.

\vspace{5pt}

\subsubsection{Trivial representation}\label{gl2triv}

Suppose first that $\pi = \mathbf{1}_G$ is the trivial representation.  By Subsection \ref{findim}, we have
\begin{eqnarray*}
\textnormal{H}^0(I_1, \mathbf{1}_G) & \cong & \chi_{\textnormal{triv}}, \\
\textnormal{H}^4(I_1, \mathbf{1}_G) & \cong & \chi_{\textnormal{triv}}.
\end{eqnarray*}
Theorem 6.4(a) of \cite{koziol:h1triv} gives the structure of $\textnormal{H}^1$, and we use Subsection \ref{findim} to determine $\textnormal{H}^3$:
\begin{eqnarray*}
\textnormal{H}^1(I_1, \mathbf{1}_G) & \cong & \chi_{\textnormal{triv}} \oplus \Ind_{\cH_T}^{\cH}(\overline{\alpha}),\\
\textnormal{H}^3(I_1, \mathbf{1}_G) & \cong & \chi_{\textnormal{triv}}^\vee \oplus \Ind_{\cH_T}^{\cH}(\overline{\alpha})^\vee\\
 & \cong & \chi_{\textnormal{triv}} \oplus \Ind_{\cH_T}^{\cH}(\overline{\alpha}),
\end{eqnarray*}
where the last isomorphism follows from \cite[Thm. 4.9]{abe:involutions}, and $\overline{\alpha}$ is as in Remark \ref{rmkdualizing}.

It therefore suffices to compute $\textnormal{H}^2$.  Note first that by \cite[Eqn. (166)]{paskunas:montrealfunctor}, we have $\ord^G_{B^-}(\mathbf{1}_G) = 0$ and $\textnormal{R}^1\ord^G_{B^-}(\mathbf{1}_G) = \overline{\alpha}$.  By Corollary \ref{ordss}, we obtain isomorphisms of $\cH_T$-modules
$$\cR^{\cH}_{\cH_T}\big(\textnormal{H}^2(I_1, \mathbf{1}_G)\big) \cong \textnormal{H}^1(I_{T,1},\overline{\alpha}) \cong \textnormal{H}^1(I_{T,1}, \mathbf{1}_T) \otimes_C \overline{\alpha} \cong \overline{\alpha}^{\oplus 2},$$
where the last isomorphism follows from $I_{T,1} \cong (1 + p\bbZ_p)^{\oplus 2} \cong \bbZ_p^{\oplus 2}$.  This implies we have an injection 
\begin{equation}\label{h2triv}
\Ind_{\cH_T}^{\cH}(\overline{\alpha})^{\oplus 2} \longhookrightarrow \textnormal{H}^2(I_1, \mathbf{1}_G).
\end{equation}
On the other hand, the assumption $p \geq 5$ implies that we may write $I_1 \cong I_1' \times \cZ_1$, where $I_1' := I_1 \cap \textnormal{SL}_2(\bbQ_p)$ and $\cZ_1$ denotes the pro-$p$ part of the center $\cZ$.  The K\"unneth formula then gives an isomorphism of $C$-vector spaces
$$\textnormal{H}^2(I_1, C) \cong \bigoplus_{i + j = 2} \textnormal{H}^i(I_1', C)\otimes_C \textnormal{H}^j(\cZ_1, C).$$
Since $\cZ_1 \cong 1 + p\bbZ_p \cong \bbZ_p$ has cohomological dimension 1, we have $\dim_C(\textnormal{H}^0(\cZ_1, C)) = \dim_C(\textnormal{H}^1(\cZ_1, C)) = 1$.  By the calculation of $\textnormal{H}^i(I_1', \mathbf{1}_{\textnormal{SL}_2(\bbQ_p)})$ in Subsubsection \ref{sl2triv} below, we have $\dim_C(\textnormal{H}^1(I_1', C)) = \dim_C(\textnormal{H}^2(I_1', C)) = 2$.  Consequently, we get $\dim_C(\textnormal{H}^2(I_1,C)) = 4$, and the injection \eqref{h2triv} must be an isomorphism.

\vspace{5pt}

\subsubsection{Principal series}\label{gl2ps}

Now let $\chi:T \longrightarrow C^\times$ denote a smooth character, and let $\Ind_B^G(\chi)$ denote the parabolically induced representation.  By \cite[Prop. 4.4]{olliviervigneras}, Subsection \ref{parabindexs}, Remark \ref{rmkdualizing}, and the fact that $\dim_{\bbQ_p}(B) = 3$, we have
\begin{eqnarray*}
\textnormal{H}^0\big(I_1, \Ind_B^G(\chi)\big) & \cong & \Ind_{\cH_T}^{\cH}(\chi) \\
\textnormal{H}^3\big(I_1, \Ind_B^G(\chi)\big) & \cong & \Ind_{\cH_T}^{\cH}(\chi^{-1}\overline{\alpha})^\vee
\end{eqnarray*}

By \cite[Props. 4.5, 4.9]{koziol:h1ps} and \cite[Cor. 3.3, Prop. 4.3]{abe:involutions}, we have a short exact sequence
$$ 0 \longrightarrow \Ind_{\cH_T}^{\cH}(\chi)^{\oplus 2} \longrightarrow \textnormal{H}^1\big(I_1, \Ind_B^G(\chi)\big) \longrightarrow \Ind_{\cH_T}^{\cH}(\chi^{-1}\overline{\alpha})^\vee \longrightarrow 0.$$
Lemma 5.13 of \cite{koziol:h1ps} shows that this exact sequence is nonsplit if and only if $\chi = \chi^s\overline{\alpha}$ (where $\chi^s$ denotes the character $t \longmapsto \chi(\sm{0}{1}{1}{0}t\sm{0}{1}{1}{0})$~).  In this case $\Ind_{\cH_T}^{\cH}(\chi^{-1}\overline{\alpha})$ is simple and $\Ind_{\cH_T}^{\cH}(\chi^{-1}\overline{\alpha})^\vee \cong \Ind_{\cH_T}^{\cH}(\chi^s\overline{\alpha}) \cong \Ind_{\cH_T}^{\cH}(\chi)$.  

Applying the isomorphisms in Subsection \ref{parabindexs} to the exact sequence above, we see that
$$0 \longrightarrow \Ind_{\cH_T}^{\cH}(\chi) \longrightarrow \textnormal{H}^2\big(I_1, \Ind_B^G(\chi)\big) \longrightarrow \big(\Ind_{\cH_T}^{\cH}(\chi^{-1}\overline{\alpha})^{\oplus 2}\big)^\vee \longrightarrow 0.$$
As above, this short exact sequence is nonsplit if and only if $\chi = \chi^s\overline{\alpha}$, in which case $\Ind_{\cH_T}^{\cH}(\chi^{-1}\overline{\alpha})$ is simple and $\Ind_{\cH_T}^{\cH}(\chi^{-1}\overline{\alpha})^\vee \cong \Ind_{\cH_T}^{\cH}(\chi^s\overline{\alpha}) \cong \Ind_{\cH_T}^{\cH}(\chi)$.

(For generic $\chi$, one can also use Corollary \ref{ordss} along with \cite[Thm. 4.2.12(1)]{emerton:ordII} to determine all $\textnormal{H}^i(I_1, \Ind_B^G(\chi))$, without appealing to the calculations in \cite{koziol:h1ps}.)

\vspace{5pt}

\subsubsection{Steinberg}

Let $\textnormal{St}_G$ denote the Steinberg representation of $G$, defined by the exact sequence
$$0 \longrightarrow \mathbf{1}_G \longrightarrow \Ind_B^G(\mathbf{1}_T) \longrightarrow \textnormal{St}_G \longrightarrow 0.$$
We also let $\cE$ denote the unique nonsplit extension of $\Ind_B^G(\overline{\alpha})$ by $\mathbf{1}_G$ (cf. \cite[Prop. 4.3.13]{emerton:ordII}):
$$0 \longrightarrow \mathbf{1}_G \longrightarrow \cE \longrightarrow \Ind_B^G(\overline{\alpha}) \longrightarrow 0.$$
By \cite[Prop. 5.7]{kohlhaase:duality}, we have
$$S^j(\textnormal{St}_G)  =  \begin{cases} \cE & \textnormal{if}~ j = 1, \\ 0 & \textnormal{if}~ j \neq 1,\end{cases}$$
and therefore the spectral sequence \eqref{mainss} collapses to give
\begin{equation}\label{steinbergdual}
\textnormal{H}^i(I_1, \cE) \cong \textnormal{H}^{3 - i}(I_1, \textnormal{St}_G)^\vee.
\end{equation}

Note first that we have
\begin{eqnarray*}
\textnormal{H}^0(I_1, \textnormal{St}_G) & \cong & \chi_{\textnormal{sign}}^\star, \\
\textnormal{H}^0(I_1, \cE) & \cong & \chi_{\textnormal{triv}},
\end{eqnarray*}
where $\chi_{\sign}^\star := \chi_{\sign}\otimes_C \textnormal{nr}_{-1}\circ\det$, $\chi_{\sign}$ denotes the sign character of $\cH$ (see \cite[Rmks. 2.23, 2.24(1)]{olliviervigneras}), and $\textnormal{nr}_{-1}:\bbQ_p^\times \longrightarrow C^\times$ denotes the unramified character sending $p$ to $-1$.  The first isomorphism comes from \cite[Thm. 4.17]{ahv:pro-p} and \cite[Prop. 3.12]{abe:inductions} (or \cite[Thm. 4.2]{vigneras:gl2}), while the second follows from the fact that $\cE$ is a nonsplit extension.  Applying \eqref{steinbergdual} gives
$$\textnormal{H}^3(I_1, \textnormal{St}_G) \cong \textnormal{H}^0(I_1, \cE)^\vee \cong \chi_{\triv}^\vee \cong \chi_{\triv}.$$

We now discuss the remaining cohomology groups.  Applying the functor of invariants to the exact sequence defining $\textnormal{St}_G$ gives an exact sequence
$$0 \longrightarrow \chi_{\textnormal{triv}} \longrightarrow \Ind_{\cH_T}^\cH(\mathbf{1}_T) \longrightarrow \chi_{\textnormal{sign}}^\star \longrightarrow 0.$$
Substituting the results of Subsubsections \ref{gl2triv} and \ref{gl2ps}, the long exact sequence for the higher cohomology groups is
\begin{center}
\begin{tikzcd}
0\ar[r]  & \Ind_{\cH_T}^{\cH}(\overline{\alpha})\oplus \chi_{\triv} \ar[r] & \Ind_{\cH_T}^{\cH}(\mathbf{1}_T)^{\oplus 2} \oplus \Ind_{\cH_T}^{\cH}(\overline{\alpha}) \ar[r] \arrow[d, phantom, ""{coordinate, name=Z}]   & \textnormal{H}^1(I_1,\textnormal{St}_G) \arrow[dll, rounded corners=8pt, to path= { -- ([xshift=3ex]\tikztostart.east) |- (Z) [near end]\tikztonodes -| ([xshift=-3ex]\tikztotarget.west) -- (\tikztotarget)}] & \\
 & \Ind_{\cH_T}^{\cH}(\overline{\alpha})^{\oplus 2} \ar[r] & \Ind_{\cH_T}^{\cH}(\mathbf{1}_T) \oplus \Ind_{\cH_T}^{\cH}(\overline{\alpha})^{\oplus 2} \ar[r] \arrow[d, phantom, ""{coordinate, name=Z}]   & \textnormal{H}^2(I_1, \textnormal{St}_G) \arrow[dll, rounded corners=8pt, to path= { -- ([xshift=3ex]\tikztostart.east) |- (Z) [near end]\tikztonodes -| ([xshift=-3ex]\tikztotarget.west) -- (\tikztotarget)}] & \\
 &  \Ind_{\cH_T}^{\cH}(\overline{\alpha})\oplus \chi_{\triv}  \ar[r] & \Ind_{\cH_T}^{\cH}(\overline{\alpha}) \ar[r, "0"] \arrow[d, phantom, ""{coordinate, name=Z}]   & \chi_{\triv} \arrow[dll, rounded corners=8pt, to path= { -- ([xshift=3ex]\tikztostart.east) |- (Z) [near end]\tikztonodes -| ([xshift=-3ex]\tikztotarget.west) -- (\tikztotarget)}] &  \\
 & \chi_{\triv} \ar[r] & 0   &   & 
\end{tikzcd}
\end{center}
We therefore see that $\textnormal{H}^1(I_1,\textnormal{St}_G)$ is an extension
$$0 \longrightarrow \Ind_{\cH_T}^{\cH}(\mathbf{1}_T) \oplus \chi_{\textnormal{sign}}^\star \longrightarrow \textnormal{H}^1(I_1,\textnormal{St}_G) \longrightarrow \Ind_{\cH_T}^{\cH}(\overline{\alpha})^{\oplus k} \longrightarrow 0,$$
where $0\leq k \leq 2$.  Since the functor $\cR^{\cH}_{\cH_T}$ is exact on finite-dimensional $\cH$-modules, \cite[Lem. 5.2, Thm. 5.20]{abe:inductions} imply that
$$0 \longrightarrow \mathbf{1}_T^{\oplus 2} \longrightarrow \cR_{\cH_T}^{\cH}\big(\textnormal{H}^1(I_1,\textnormal{St}_G)\big) \longrightarrow  \overline{\alpha}^{\oplus k} \longrightarrow 0.$$
On the other hand, by \cite[Thm. 4.2.12(2)]{emerton:ordII} we have $\ord^{G}_{B^-}(\textnormal{St}_G) \cong \mathbf{1}_T$ and $\textnormal{R}^1\ord^{G}_{B^-}(\textnormal{St}_G) = 0$.  Hence, Corollary \ref{ordss} gives an isomorphism of $\cH_T$-modules
$$\cR_{\cH_T}^{\cH}\big(\textnormal{H}^1(I_1,\textnormal{St}_G)\big) \cong \textnormal{H}^1(I_{T,1}, \mathbf{1}_T) \cong \mathbf{1}_T^{\oplus 2}.$$
Consequently $k$ must be equal to 0, and from the exact sequence above we deduce that 
$$\textnormal{H}^1(I_1,\textnormal{St}_G) \cong  \Ind_{\cH_T}^{\cH}(\mathbf{1}_T) \oplus \chi_{\textnormal{sign}}^\star,$$
and that $\textnormal{H}^2(I_1,\textnormal{St}_G)$ is a (possibly split) extension of $\chi_{\triv}$ by $\Ind_{\cH_T}^{\cH}(\mathbf{1}_T)$.  (Note also that this extension splits when restricted to the pro-$p$-Iwahori--Hecke algebra of $\textnormal{SL}_2(\bbQ_p)$.)

\vspace{5pt}

\subsubsection{Supersingular representations}\label{gl2ssing}

We now discuss supersingular representations.

For $0 \leq r \leq p - 1$, we let $\pi(r,0,1)$ denote the $G$-representation given by
$$\cind_{\cZ\textnormal{GL}_2(\bbZ_p)}^{\textnormal{GL}_2(\bbQ_p)}\big(\textnormal{Sym}^r(C^{\oplus 2})\big)\Big/\T,$$
where $\textnormal{GL}_2(\bbZ_p)$ acts on $\textnormal{Sym}^r(C^{\oplus 2})$ via reduction mod $p$, the matrix $\sm{p}{0}{0}{p} \in \cZ$ acts trivially on $\textnormal{Sym}^r(C^{\oplus 2})$, and $\T$ denotes a certain spherical Hecke operator.  Up to twist, these representations constitute all absolutely irreducible supersingular $G$-representations over $C$.  Moreover, the two-dimensional space $\pi(r,0,1)^{I_1}$ is simple as an $\cH$-module, and we define
$$\fm(r,0,1) := \pi(r,0,1)^{I_1} = \textnormal{H}^0(I_1, \pi(r,0,1)).$$

The $\cH$-module $\textnormal{H}^1(I_1, \pi(r,0,1))$ is computed in \cite[Prop. 10.5, Eqn. (49)]{paskunas:exts}.  More precisely, the dimension calculations of \emph{op. cit.} are slightly different than ours, as they are stated in the case where the central character is fixed throughout.  When the central character is not fixed, the relevant differences are the following (all references are to \cite{paskunas:exts}): in Lemmas 5.7, 5.8, and 7.1, we have $\dim(\Ext^1_{I \cap B}(\chi, \chi)) = 2$; in Corollary 7.4, we have $\Ext^2_I(\chi,\chi) = 0$ and $\dim(\Ext^1_I(\chi,N)) = 2$ (using the results of Subsubsection \ref{gl2triv}); in Theorem 7.9, we have $\dim(\Ext^1_I(\chi,\pi_\sigma)) = 3$.  With these modifications, we may proceed as in the proofs of \cite[Prop. 10.5, Thm. 10.7]{paskunas:exts} to obtain
$$\textnormal{H}^1(I_1, \pi(r,0,1)) \cong \fm(r,0,1)^{\oplus 3}.$$

We now compute the higher cohomology.  By \cite[Thm. 5.13]{kohlhaase:duality}, we have 
$$S^j(\pi(r,0,1)) = \begin{cases} \pi(r, 0, \omega^{-r}) := \pi(r,0,1) \otimes_C \omega^{-r}\circ \det & \textnormal{if}~ j = 1, \\ 0 & \textnormal{if}~ j \neq 1,\end{cases}$$
where $\omega: \bbQ_p^\times \longrightarrow C^\times$ is the character $xp^a \longmapsto \overline{x}$ ($x\in \bbZ_p^\times, a\in \bbZ$).  On the side of Hecke modules, by unraveling the definitions in \cite[Thm. 4.8]{abe:involutions} we obtain
$$\fm(r,0,1)^\vee \cong \fm(r,0,\omega^{-r}) := \pi(r,0,\omega^{-r})^{I_1}.$$
Therefore, the spectral sequence \eqref{mainss} and the previous paragraphs give
\begin{align*}
\textnormal{H}^2(I_1, \pi(r,0,1)) & \cong  \textnormal{H}^1(I_1, \pi(r,0,\omega^{-r}))^\vee \cong \fm(r,0,1)^{\oplus 3}, \\
\textnormal{H}^3(I_1, \pi(r,0,1)) & \cong  \textnormal{H}^0(I_1, \pi(r,0,\omega^{-r}))^\vee \cong \fm(r,0,1).
\end{align*}

\vspace{5pt}

\subsection{$\textrm{SL}_2(\bbQ_p)$}

We now suppose $\bG = \bS\bL_{2/\bbQ_p}$ with $p \geq 5$.  As before, we let $\bB = \bT\ltimes \bU$ be the upper-triangular Borel subgroup, and $I_1$ the ``upper-triangular mod $p$'' pro-$p$ Iwahori subgroup.  Once again, the assumption $p \geq 5$ implies $I_1$ is torsion-free, and therefore has cohomological dimension $\dim_{\bbQ_p}(G) = 3$.  The calculations below are similar to those for $\textnormal{GL}_2(\bbQ_p)$, using the Poincar\'e duality spectral sequence \eqref{mainss} and Corollary \ref{ordss}; we simply record the nonzero cohomology spaces.

\vspace{5pt}

\subsubsection{Trivial representation}\label{sl2triv}

We first let $\pi = \mathbf{1}_G$ denote the trivial representation.  We have
\begin{eqnarray*}
\textnormal{H}^0(I_1, \mathbf{1}_G) & \cong & \chi_{\triv}, \\
\textnormal{H}^1(I_1, \mathbf{1}_G) & \cong & \Ind_{\cH_T}^{\cH}(\overline{\alpha}), \\
\textnormal{H}^2(I_1, \mathbf{1}_G) & \cong & \Ind_{\cH_T}^{\cH}(\overline{\alpha}), \\
\textnormal{H}^3(I_1, \mathbf{1}_G) & \cong & \chi_{\triv}.
\end{eqnarray*}
Here $\overline{\alpha}: T \longrightarrow C^\times$ denotes the character $\sm{xp^a}{0}{0}{x^{-1}p^{-a}} \longmapsto \overline{x}^2$ ($x\in \bbZ_p^\times, a\in \bbZ$).  This can quickly be obtained from the two spectral sequences, using \cite[Lem. 3.1.1]{hauseux:parabind}.  (See also \cite[Rmk. 6.6]{koziol:h1triv}.)

\vspace{5pt}

\subsubsection{Principal series}

We now let $\chi:T \longrightarrow C^\times$ denote a smooth character of the torus, and let $\Ind_B^G(\chi)$ denote the parabolically induced representation.  We have
\begin{eqnarray*}
\textnormal{H}^0\big(I_1, \Ind_B^G(\chi)\big) & \cong & \Ind_{\cH_T}^{\cH}(\chi), \\
\textnormal{H}^2\big(I_1, \Ind_B^G(\chi)\big)  & \cong & \Ind_{\cH_T}^{\cH}(\chi^{-1}\overline{\alpha})^\vee.
\end{eqnarray*}
Furthermore, the results of \cite{koziol:h1ps} can be adapted to show that we have an exact sequence
$$0 \longrightarrow  \Ind_{\cH_T}^{\cH}(\chi) \longrightarrow \textnormal{H}^1\big(I_1, \Ind_B^G(\chi)\big) \longrightarrow  \Ind_{\cH_T}^{\cH}(\chi^{-1}\overline{\alpha})^\vee \longrightarrow 0,$$
which is nonsplit if and only if $\chi$ is equal to the character $\sm{xp^a}{0}{0}{x^{-1}p^{-a}} \longmapsto \overline{x}$ ($x\in \bbZ_p^\times, a\in \bbZ$).

\vspace{5pt}

\subsubsection{Steinberg}

Suppose that $\pi = \textnormal{St}_G$ is the Steinberg representation of $G$.  Using the previous two sections, we obtain
\begin{eqnarray*}
\textnormal{H}^0(I_1, \textnormal{St}_G) & \cong & \chi_{\sign}, \\
\textnormal{H}^1(I_1, \textnormal{St}_G) & \cong & \Ind_{\cH_T}^{\cH}(\mathbf{1}_T), \\
\textnormal{H}^2(I_1, \textnormal{St}_G) & \cong & \chi_{\triv}.
\end{eqnarray*}

\vspace{5pt}

\subsubsection{Supersingular representations}

Finally, we discuss cohomology of supersingular representations.  Most of the techniques of \cite{paskunas:exts} work mutatis mutandis in the $\textnormal{SL}_2(\bbQ_p)$ setting, so we only outline the main ideas (see also \cite{nadimpalli}).

Recall from Subsection \ref{gl2ssing} the supersingular representations $\pi(r,0,1)$ of $\textnormal{GL}_2(\bbQ_p)$.  Define $v := \overline{[1,X^r]}, v' := \overline{[\sm{0}{1}{p}{0}, X^r]} \in \pi(r,0,1)$, and let $\pi_r$ denote the $G$-subrepresentation of $\pi(r,0,1)|_{G}$ generated by $v$.  (Here $X^r$ denotes a vector spanning the one-dimensional vector space $\textnormal{Sym}^r(C^{\oplus 2})^{I_1}$, and $[1,X^r] \in \cind_{\cZ\textnormal{GL}_2(\bbZ_p)}^{\textnormal{GL}_2(\bbQ_p)}(\textnormal{Sym}^r(C^{\oplus 2}))$ denotes the function with support $\cZ\textnormal{GL}_2(\bbZ_p)$ and value $X^r$ at $1$; see \cite[\S\S 2.3, 2.6]{breuil:gl2qp}.)  The representations $\pi_0, \pi_1, \ldots, \pi_{p - 1}$ are pairwise non-isomorphic, and constitute all absolutely irreducible supersingular $G$-representations over $C$ (see \cite[Thm. 4.12(1)]{abdellatif:sl2qp}).  Further, by \emph{op. cit.}, Proposition 4.7, the space $\pi_r^{I_1}$ is one-dimensional, and consequently simple as an $\cH$-module.  We denote this $\cH$-module by $\fm_r$, so that
$$\textnormal{H}^0(I_1,\pi_r) = \fm_r$$
by definition.  The forthcoming PhD thesis of Jake Postema shows that
$$S^j(\pi_r) = \begin{cases}\pi_{p - 1 - r} & \textnormal{if}~ j = 1,\\ 0 & \textnormal{if}~ j \neq 1.\end{cases}$$
Therefore, applying \eqref{mainss}, we obtain
$$\textnormal{H}^2(I_1,\pi_r) \cong \textnormal{H}^0(I_1, \pi_{p - 1 - r})^\vee \cong \fm_{p - 1 - r}^\vee = \begin{cases}\fm_{p - 1 - r} & \textnormal{if}~ r \in \{0, p - 1\},\\ \fm_r & \textnormal{if}~ 0 < r < p - 1,\end{cases}$$
where the last equality follows from \cite[Thm. 4.8]{abe:involutions} (or an easy calculation by hand).

We now determine $\textnormal{H}^1$.  Define the following subspaces of $\pi_r$:
\begin{eqnarray*}
M & := & \left\langle \begin{pmatrix}1 & \bbZ_p \\ 0 & 1\end{pmatrix}\begin{pmatrix} p^n & 0 \\ 0 & p^{-n}\end{pmatrix}.v\right\rangle_{n\geq 0}, \\
\Pi.M' & := & \left\langle \begin{pmatrix}0 & 1 \\ p & 0\end{pmatrix}\begin{pmatrix}1 & \bbZ_p \\ 0 & 1\end{pmatrix}\begin{pmatrix} p^n & 0 \\ 0 & p^{-n}\end{pmatrix}.v'\right\rangle_{n\geq 0}.
\end{eqnarray*}
Both $M$ and $\Pi.M'$ are in fact stable by $I$, and fit into an $I$-equivariant resolution
$$0 \longrightarrow \xi_r \longrightarrow M \oplus \Pi.M' \longrightarrow \pi_r \longrightarrow 0,$$
where $\xi_r$ is the character $\sm{x}{0}{0}{x^{-1}} \longmapsto x^r$ ($x \in \bbF_p^\times$), inflated to $I$.  We may now proceed as in \cite[Thm. 7.9]{paskunas:exts} to conclude that $\Ext^1_I(\xi_r,\pi_r)$ is two-dimensional, and the analogous Ext spaces for other characters of $I$ are all 0.  (Note that fixing the central character in the $\textnormal{GL}_2(\bbQ_p)$ case in \cite{paskunas:exts} will produce the same dimensions as the $\textnormal{SL}_2(\bbQ_p)$ case.)  We conclude that $\textnormal{H}^1(I_1, \pi_r)$ is two-dimensional, with the action of $I$ given by the character $\xi_r$.  When $0 < r < p - 1$, this is enough to conclude
$$\textnormal{H}^1(I_1, \pi_r) \cong \fm_r^{\oplus 2}$$
(special care must be taken with the case $r = (p - 1)/2$).  On the other hand, when $r \in \{0, p - 1\}$, we have
$$\textnormal{H}^1(I_1, \pi_r) \cong \fm_0 \oplus \fm_{p - 1}.$$
Note that these isomorphisms are consistent with the relation $\textnormal{H}^1(I_1,\pi_r) \cong \textnormal{H}^1(I_1, \pi_{p - 1 - r})^\vee$.

\vspace{5pt}

\subsection{Steinberg representation of $\textnormal{GL}_3(\bbQ_p)$}

Finally, let $\bG = \bG\bL_{3/\bbQ_p}$, let $\bB$ denote the upper-triangular Borel subgroup, and $I_1$ the ``upper-triangular mod $p$'' pro-$p$ Iwahori subgroup.  Suppose $p \geq 5$, so that $I_1$ is torsion-free of cohomological dimension 9.  Denote by $\bP_1$ and $\bP_2$ the two standard parabolic subgroups for which $\bB \subsetneq \bP_i \subsetneq \bG$.

Let $\textnormal{St}_G$ denote the Steinberg representation of $G = \textnormal{GL}_3(\bbQ_p)$.  By \cite[pf. of Prop. 5.6]{kohlhaase:duality}, we have $S^j(\textnormal{St}_G) = 0$ for $j \notin \{2,3\}$, $S^2(\textnormal{St}_G) = \mathbf{1}_G$, and $S^3(\textnormal{St}_G)$ fits into a short exact sequence
$$0 \longrightarrow \Ind_{P_1}^G(\chi_{P_1}) \oplus \Ind_{P_2}^G(\chi_{P_2}) \longrightarrow S^3(\textnormal{St}_G) \longrightarrow \Ind_{B}^G(\chi_{B}) \longrightarrow 0.$$
(We shall soon show that this short exact sequence is nonsplit.)

We will determine some of the $\cH$-modules $\textnormal{H}^{i}(I_1,\textnormal{St}_G)$.  Note first that $\textnormal{H}^{0}(I_1,\textnormal{St}_G) \cong \chi_{\sign}$ by \cite[Thm. 4.17]{ahv:pro-p}.  To access higher cohomology groups, we will use the spectral sequence \eqref{mainss}
$$E_2^{i,j} = \textnormal{H}^i\big(I_1,S^j(\textnormal{St}_G)\big) \Longrightarrow \textnormal{H}^{9 - i - j}(I_1,\textnormal{St}_G)^\vee.$$
By the above calculation of $S^j(\textnormal{St}_G)$, this spectral sequence degenerates at the $E_3$ page.  In particular, we obtain
\begin{eqnarray*}
\textnormal{H}^{9}(I_1,\textnormal{St}_G) & = & 0 \\
\textnormal{H}^{8}(I_1,\textnormal{St}_G) & = & 0 \\
\textnormal{H}^{7}(I_1,\textnormal{St}_G) & \cong & \textnormal{H}^{0}\big(I_1,S^2(\textnormal{St}_G)\big)^\vee \\
 & \cong & \textnormal{H}^{0}(I_1,\mathbf{1}_G)^\vee \\
 & \cong & \chi_{\triv}
\end{eqnarray*}
Further, one can show that $\textnormal{H}^{6}(I_1,\textnormal{St}_G)$ has a quotient isomorphic to $\chi_{\triv}\oplus \fm_{\alpha}^\vee$, where $\fm_{\alpha}$ is the 3-dimensional simple supersingular $\cH$-module constructed in \cite[Thm. 6.4(b)]{koziol:h1triv}.

We also obtain information about the pro-$p$-Iwahori cohomology of $S^3(\textnormal{St}_G)$.  Since $E_3^{i,j} = E_\infty^{i,j} = 0$ for $i + j > 9$, we obtain
\begin{eqnarray*}
\textnormal{H}^{9}\big(I_1,S^3(\textnormal{St}_G)\big) & = & 0 \\
\textnormal{H}^{8}\big(I_1,S^3(\textnormal{St}_G)\big) & = & 0 \\
\textnormal{H}^{7}\big(I_1,S^3(\textnormal{St}_G)\big) & \cong & \textnormal{H}^{9}\big(I_1,S^2(\textnormal{St}_G)\big) \\
 & \cong & \textnormal{H}^{9}(I_1,\mathbf{1}_G) \\
 & \cong & \chi_{\triv},
\end{eqnarray*}
where the last isomorphism follows from Subsection \ref{findim}.  As above, we may also deduce that $\textnormal{H}^6(I_1, S^3(\textnormal{St}_G))$ has a quotient isomorphic to $\chi_\triv \oplus \fm_\alpha^\vee$.

We now show how these results may be used to show that the extension defining $S^3(\textnormal{St}_G)$ is nonsplit.  Assume the contrary, so that
$$S^3(\textnormal{St}_G) \cong \bigoplus_{\bP \subsetneq \bG}\Ind_{P}^G(\chi_{P}).$$
By Subsection \ref{parabindexs}, we would then obtain
\begin{eqnarray*}
\textnormal{H}^{7}\big(I_1,S^3(\textnormal{St}_G)\big) & \cong & \bigoplus_{\bP \subsetneq \bG} \textnormal{H}^{7}\big(I_1, \Ind_P^G(\chi_P)\big) \\
 & \cong & \textnormal{H}^{7}\big(I_1, \Ind_{B}^G(\chi_{B})\big) \oplus \textnormal{H}^{7}\big(I_1, \Ind_{P_1}^G(\chi_{P_1})\big) \oplus \textnormal{H}^{7}\big(I_1, \Ind_{P_2}^G(\chi_{P_2})\big) \\
 & \cong & \textnormal{H}^{0}\big(I_1,\Ind_{P_1}^G(\mathbf{1}_{M_1})\big)^\vee \oplus \textnormal{H}^{0}\big(I_1,\Ind_{P_2}^G(\mathbf{1}_{M_2})\big)^\vee.
\end{eqnarray*}
This would then imply that $\textnormal{H}^{7}(I_1,S^3(\textnormal{St}_G))$ has dimension greater than 1, a contradiction.  Thus, the short exact sequence does not split, and we obtain 
$$\textnormal{H}^{0}\big(I_1,S^3(\textnormal{St}_G)\big) \cong \Ind_{\cH_{M_1}}^\cH(\chi_{P_1}) \oplus \Ind_{\cH_{M_2}}^\cH(\chi_{P_2}).$$

\vspace{10pt}

\appendix

\section{}\label{app}

In this appendix we slightly expand on results of Emerton and Pa\v{s}k\={u}nas.

We maintain the same notation as in the body of the article.  Namely, we let $\bG$ denote a connected reductive group over $F$, let $\boldsymbol{\cZ}$ denote the connected center of $\bG$, and let $\widetilde{\cZ}_0$ denote the maximal compact subgroup of $\cZ = \boldsymbol{\cZ}(F)$.

\vspace{5pt}

\subsection{On Conjecture \ref{conjA}}

Let $G_\aff$ denote the (open) subgroup of $G$ generated by all parahoric subgroups (equivalently, $G_\aff$ is the kernel of the Kottwitz homomorphism).  We also let $I \subset G_\aff$ denote a choice of Iwahori subgroup, that is, the stabilizer in $G_\aff$ of a chamber of the semisimple Bruhat--Tits building of $G$.

Recall that the group $\cZ/\widetilde{\cZ}_0$ is free of finite rank $s$, and we have chosen central elements $z_1, \ldots, z_s \in G$ freely generating $\cZ/\widetilde{\cZ}_0$.  We define 
$$H := \langle z_1, \ldots, z_s \rangle G_\aff,$$
which is an open, normal, finite-index subgroup of $G$.  Given $c_1, c_2, \ldots, c_s \in C^\times$, recall that $\mathfrak{Rep}^{\textnormal{adm}}_{z_i = c_i}(G)$ denotes the abelian category of admissible $G$-representations on which the $z_i$ act by $c_i$.

\begin{conj}\label{emerton-conj}
Let $\pi\in \mathfrak{Rep}^{\textnormal{adm}}_{z_i = c_i}(G)$.  Then there exists $A \in \mathfrak{Rep}^{\textnormal{adm}}_{z_i = c_i}(G)$ and a $G$-equivariant injection $\pi \longhookrightarrow A$, such that $A|_I$ is injective in $\mathfrak{Rep}^\infty(I)$.  
\end{conj}

\begin{thm}\label{emerton-conjthm}
Suppose the semisimple $F$-rank of $\bG$ is $0$ or $1$.  Then Conjecture \ref{emerton-conj} is true.
\end{thm}

We prove the above theorem in a series of steps.  If $G'$ is an open subgroup of $G$ containing $I$, we will say ``Conjecture \ref{emerton-conj} is true for $G'$'' if the statement remains valid when $G$ is replaced by $G'$.  (When $G'$ does not contain the elements $z_i$, we omit the condition ``$z_i$ acts by $c_i$.'')

\begin{step}\label{step1}
Suppose the semisimple $F$-rank of $\bG$ is $0$ or $1$.  Then Conjecture \ref{emerton-conj} is true for $G_\aff$.  
\end{step}

\begin{proof}
When the semisimple $F$-rank of $\bG$ is 0, the semisimple Bruhat--Tits building of $G$ is a point.  Consequently, $G_\aff = I$ is a profinite group, and we therefore we have the existence of an injective envelope $\pi \longhookrightarrow \inj_{G_\aff}(\pi)$ in the category $\mathfrak{Rep}^\infty(G_\aff)$.  It suffices to show that $\inj_{G_\aff}(\pi)$ is admissible.  Note that $\inj_{G_\aff}(\pi) \cong \inj_{G_\aff}(\soc_{G_\aff}(\pi))$, and the $G_\aff$-representation $\soc_{G_\aff}(\pi)$ is semisimple and finite-dimensional (by the admissibility of $\pi$).  We are now reduced to showing that $\inj_{G_\aff}(\sigma)$ is admissible, where $\sigma$ denotes an irreducible smooth $G_\aff$-representation.  This follows from \cite[Lems. 6.2.4 and 6.3.2]{paskunas:diags} (the proofs do not require that the coefficient field be algebraically closed).

Suppose now that the semisimple $F$-rank of $\bG$ is 1.  In this case, the semisimple Bruhat--Tits building $\sB$ of $G$ is a tree, and we let $\fe$ denote the edge whose stabilizer is $I$, and let $\fv$ and $\fv'$ denote the two vertices in the closure of $\fe$.  The group $G_\aff$ acts on $\sB$, with one orbit in the set of edges, and two orbits in the set of vertices.  If we let $K$ and $K'$ denote the parahoric subgroups associated to $\fv$ and $\fv'$, then \cite[\S 4, Thm. 6]{serre:trees} implies
$$G_\aff = K *_I K'.$$

Since $G_\aff$ is an amalgamated product of two parahoric subgroups, we may use the formalism of diagrams utilized in \cite{koziolxu}.  In particular, the $G_\aff$-representation $\pi$ gives a diagram
\begin{center}
$D_\pi \quad := \quad \left(
\begin{tikzcd}
 & \pi|_K \\
 \pi|_I \ar[ur, "\textnormal{id}", "\mathclap{\sim}"' sloped] \ar[dr, "\textnormal{id}"', "\mathclap{\sim}" sloped]& \\
 & \pi|_{K'}
\end{tikzcd}
\right),$
\end{center}
with the property that $\textnormal{H}_0(\sB, D_\pi) \cong \pi$ as $G_\aff$-representations.  Fixing injective envelopes of each representation, we obtain
\begin{center}
\begin{tikzcd}
 & \pi|_K \ar[rr, hook] & & \textnormal{inj}_K(\pi|_K) \\
 \pi|_I  \ar[ur, "\textnormal{id}", "\mathclap{\sim}"' sloped] \ar[dr, "\textnormal{id}"', "\mathclap{\sim}" sloped]\ar[rr, hook] & & \textnormal{inj}_I(\pi|_I) &  \\
 & \pi|_{K'} \ar[rr, hook] & & \textnormal{inj}_{K'}(\pi|_{K'})
\end{tikzcd}
\end{center}
Note that $\inj_I(\pi|_I)$ is an admissible $I$-representation (and similarly for the groups $K$ and $K'$), as in the first paragraph.  By \cite[Lem. 6.2.3]{paskunas:diags}, the $I$-equivariant injection $\pi|_I \longhookrightarrow (\pi|_K)|_I \longhookrightarrow \textnormal{inj}_K(\pi|_K)|_I$ extends to an injection $\alpha: \textnormal{inj}_I(\pi|_I) \longhookrightarrow \textnormal{inj}_K(\pi|_K)|_I$ (and similarly for $K'$).  Therefore, we get a morphism of diagrams
\begin{center}
\begin{tikzcd}
 & \pi|_K \ar[rr, hook] & & \textnormal{inj}_K(\pi|_K) \\
 \pi|_I  \ar[ur, "\textnormal{id}", "\mathclap{\sim}"' sloped] \ar[dr, "\textnormal{id}"', "\mathclap{\sim}" sloped] \ar[rr, hook] & & \textnormal{inj}_I(\pi|_I) \ar[ur, hook, "\alpha"] \ar[dr, hook, "\alpha'"'] &  \\
 & \pi|_{K'} \ar[rr, hook] & & \textnormal{inj}_{K'}(\pi|_{K'})
\end{tikzcd}
\end{center}
(that is, a diagram in which each square is commutative).

For $J\in \{K, K', I\}$, we have $\textnormal{inj}_J(\pi|_J) \cong \textnormal{inj}_J(\textnormal{soc}_J(\pi|_J))$.  Since $\pi$ is admissible, there exists an integer $a_J$ such that $\textnormal{soc}_J(\pi|_J) \longhookrightarrow C[J/J_1]^{\oplus a_J}$, where $J_1$ denotes the pro-$p$ radical of $J$, and therefore $\inj_J(\pi|_J) \cong \inj_J(\soc_J(\pi|_J)) \longhookrightarrow \inj_J(C[J/J_1])^{\oplus a_J}$.  Note that we have $\inj_K(C[K/K_1])|_I \cong \inj_I(C[I/I_1])^{\oplus [K:I]}$ (and similarly for $K'$).  Indeed, by Pontryagin duality, it suffices to show $C\llbracket K\rrbracket \cong C\llbracket I\rrbracket^{\oplus [K:I]}$, which follows from \cite[Cor. 19.4 iv.]{schneider:padicliegroups}.

We may therefore choose the integers $a_K, a_{K'}$, and $a_I$ such that we have injections
\begin{center}
\begin{tikzcd}
& \textnormal{inj}_K(\pi|_K) \ar[rr, hook, "j_K"] & & \inj_K(C[K/K_1])^{\oplus a_K}\\
 \textnormal{inj}_I(\pi|_I) \ar[rr, hook, "j_I"]\ar[ur, hook, "\alpha"] \ar[dr, hook, "\alpha'"'] & & \inj_I(C[I/I_1])^{\oplus a_I} &  \\
 & \textnormal{inj}_{K'}(\pi|_{K'}) \ar[rr, hook, "j_{K'}"] & & \inj_{K'}(C[K'/K'_1])^{\oplus a_{K'}}
\end{tikzcd}
\end{center}
and such that we have isomorphisms of $I$-representations 
$$\inj_K(C[K/K_1])^{\oplus a_K}|_I \cong \inj_I(C[I/I_1])^{\oplus a_I} \cong \inj_{K'}(C[K'/K'_1])^{\oplus a_{K'}}|_I.$$
Since $\textnormal{inj}_I(\pi|_I)$ is an injective $I$-representation, we have splittings
\begin{eqnarray*}
\inj_I(C[I/I_1])^{\oplus a_I} & = & j_I\big(\textnormal{inj}_I(\pi|_I)\big) \oplus A_I\\
\inj_K(C[K/K_1])^{\oplus a_K}|_I & = & j_K \circ \alpha\big(\inj_I(\pi|_I)\big) \oplus A_K\\
\inj_{K'}(C[K'/K'_1])^{\oplus a_{K'}}|_I & = & j_{K'} \circ \alpha'\big(\inj_I(\pi|_I)\big) \oplus A_{K'}.
\end{eqnarray*}
By construction, there exist $I$-equivariant isomorphisms $\beta: A_I \stackrel{\sim}{\longrightarrow} A_K$ and $\beta': A_I \stackrel{\sim}{\longrightarrow} A_{K'}$.  Therefore, we may construct $I$-equivariant isomorphisms
\begin{align*}
\gamma := (j_K \circ \alpha \circ j_I^{-1})\oplus \beta : \quad & \inj_I(C[I/I_1])^{\oplus a_I}  \stackrel{\sim}{\longrightarrow}  \inj_K(C[K/K_1])^{\oplus a_K}|_I \\
\gamma' := (j_{K'} \circ \alpha' \circ j_I^{-1})\oplus \beta' : \quad &  \inj_I(C[I/I_1])^{\oplus a_I}  \stackrel{\sim}{\longrightarrow}  \inj_{K'}(C[K'/K'_1])^{\oplus a_{K'}}|_I.
\end{align*}
This implies that we have a morphism of diagrams
\begin{equation}\label{morph-diag}
\begin{tikzcd}
& \textnormal{inj}_K(\pi|_K) \ar[rr, hook, "j_K"] & & \inj_K(C[K/K_1])^{\oplus a_K}\\
 \textnormal{inj}_I(\pi|_I) \ar[rr, hook, "j_I"]\ar[ur, hook, "\alpha"] \ar[dr, hook, "\alpha'"'] & & \inj_I(C[I/I_1])^{\oplus a_I} \ar[ur, "\gamma", "\mathclap{\sim}"' sloped] \ar[dr, "\gamma'"', "\mathclap{\sim}" sloped]&  \\
 & \textnormal{inj}_{K'}(\pi|_{K'}) \ar[rr, hook, "j_{K'}"] & & \inj_{K'}(C[K'/K'_1])^{\oplus a_{K'}}
\end{tikzcd}
\end{equation}

Let us denote the diagram on the right of \eqref{morph-diag} by $D$.  We have thus constructed a morphism of diagrams $D_\pi \longrightarrow D$ in which all arrows are injections, and where all arrows of $D_\pi$ and $D$ are isomorphisms.  We therefore obtain an injection of $G_\aff$-representations
$$\pi \cong \textnormal{H}_0(\sB, D_\pi) \longhookrightarrow \textnormal{H}_0(\sB, D).$$
By \cite[Prop. 5.3.5]{paskunas:diags}, we have
$$\textnormal{H}_0(\sB, D)|_I \cong \inj_I(C[I/I_1])^{\oplus a_I},$$
which is injective as a smooth $I$-representation.  Furthermore, since $C[I/I_1]$ is injective as a representation of $I/I_1$, we have
$$\textnormal{H}_0(\sB, D)^{I_1} \cong \big(\inj_I(C[I/I_1])^{\oplus a_I}\big)^{I_1} \stackrel{\textnormal{\cite[Lem. 6.2.4]{paskunas:diags}}}{\cong} \inj_{I/I_1}(C[I/I_1])^{\oplus a_I} \cong C[I/I_1]^{\oplus a_I},$$
and therefore $\textnormal{H}_0(\sB, D)$ is admissible.  
\end{proof}

\begin{step}
If Conjecture \ref{emerton-conj} is true for $G_\aff$, then it is true for $H$.  
\end{step}

\begin{proof}
Let $\pi \in \mathfrak{Rep}^{\textnormal{adm}}_{z_i = c_i}(H)$.  Assuming Conjecture \ref{emerton-conj} is true for $G_\aff$, we can find an admissible $G_\aff$-representation $A\in \mathfrak{Rep}^{\textnormal{adm}}(G_\aff)$ and a $G_\aff$-equivariant injection
$$\pi|_{G_\aff} \longhookrightarrow A,$$
such that $A|_I$ is injective in $\mathfrak{Rep}^\infty(I)$.  If we define an $H$-representation $\widetilde{A}$ by the conditions that $\widetilde{A}|_{G_\aff} = A$ and that the $z_i$ act by $c_i$, then the above injection extends to an $H$-equivariant injection
$$\pi \longhookrightarrow \widetilde{A}.$$
One easily checks that $\widetilde{A}$ satisfies the conditions of Conjecture \ref{emerton-conj} for $H$.
\end{proof}

\begin{step}
If Conjecture \ref{emerton-conj} is true for $H$, then it is true for $G$.  
\end{step}

\begin{proof}
Suppose $\pi \in \mathfrak{Rep}^{\textnormal{adm}}_{z_i = c_i}(G)$.  Assuming Conjecture \ref{emerton-conj} is true for $H$, we can find an admissible $H$-representation $A \in \mathfrak{Rep}^{\textnormal{adm}}_{z_i = c_i}(H)$ and an $H$-equivariant injection
$$\pi|_H \longhookrightarrow A,$$
such that $A|_I$ is injective in $\mathfrak{Rep}^\infty(I)$.  Taking inductions, we obtain $G$-equivariant injections
$$\pi \longhookrightarrow \Ind_H^G(\pi|_H) \longhookrightarrow \Ind_H^G(A).$$
Since $H$ is of finite index in $G$, the representation $\Ind_H^G(A) \cong \cind_H^G(A)$ is admissible (\cite[Lem. 2.2]{emertonpaskunas}), and by normality of $H$ we have 
$$\Ind_H^G(A)|_I \cong  \prod_{g\in G/H} \Ind^I_{I \cap gHg^{-1}}(A^g|_{I \cap gHg^{-1}}) \cong \prod_{g\in G/H} A^g|_I.$$  
We conclude that $\Ind_H^G(A)|_I$ is an injective $I$-representation.  
\end{proof}

\vspace{5pt}

\subsection{On Conjecture \ref{conjB}}

Our next task will be to investigate the derived functors of ordinary parts $\textnormal{R}^i\ord^G_{P^-}$.  We let $\bG^{\textnormal{der}}$ denote the derived subgroup of $\bG$, and $\bG^{\textnormal{sc}}$ the simply connected cover of $\bG^{\textnormal{der}}$.  Recall that $\boldsymbol{\cZ}$ is the connected center of $\bG$ and $\widetilde{\cZ}_0$ is the maximal compact subgroup of $\cZ$.

\begin{conj}[cf. \cite{emerton:ordII}, Conjecture 3.7.2] 
\label{sl2eff}
Let $\bP = \bM \ltimes \bN$ denote a standard parabolic subgroup of $\bG$, and let $\pi \in \mathfrak{Rep}^{\textnormal{ladm}}(G)$ be an injective object.  Then $\pi|_{N_0^-}$ is injective in $\mathfrak{Rep}^{\infty}(N_0^-)$.  Consequently the functors $\textnormal{H}^i\ord^G_{P^-}$ are effaceable on the category $\mathfrak{Rep}^{\textnormal{ladm}}(G)$ for $i > 0$, and therefore $\textnormal{H}^i\ord^G_{P^-} \simeq \textnormal{R}^i\ord^G_{P^-}$ for $i > 0$.  
\end{conj}

\begin{thm}\label{sl2effthm}
Suppose $\bG^{\textnormal{sc}} \cong \bS\bL_{2/F}$.  Then Conjecture \ref{sl2eff} is true.
\end{thm}

We again proceed in several steps.  For the proof, it will be convenient to allow more general coefficient rings.  We therefore let $R$ denote a local Artinian $\bbZ_p$-algebra with finite residue field, and suppose the residue field of $R$ contains $C$.  We refer to \cite[\S 2]{emerton:ordI} for the relevant definitions of smooth, admissible, etc., $G$-representations on $R$-modules.  We denote the relevant categories by appending the subscript ``$R$'' to notation already introduced; so, for example, $\mathfrak{Rep}^\infty_R(G)$ denotes the category of smooth $G$-representations over $R$.

\setcounter{step}{0}
\begin{step}\label{eff-1}
Conjecture \ref{sl2eff} is true when $\bG = \bS\bL_{2/F}$.
\end{step}

\begin{proof}
This follows in exactly the same way as the proofs of Theorem 3.4 and Corollary 3.8 of \cite{emertonpaskunas} (with $G^0$ replaced by $\textnormal{SL}_2(F)$); in particular, if $\pi \in \mathfrak{Rep}^{\textnormal{ladm}}_R(G)$ is an injective object, then $\pi|_{\textnormal{SL}_2(\cO_F)} \in \mathfrak{Rep}^{\infty}_R(\textnormal{SL}_2(\cO_F))$ is injective.  The result then follows from \cite[Prop. 2.1.11]{emerton:ordII}.  
\end{proof}

\begin{step}\label{eff-2}
Conjecture \ref{sl2eff} is true when $\bG$ is semisimple and $\bG^{\textnormal{sc}} \cong \bS\bL_{2/F}$.
\end{step}

\begin{proof}
Let $\bG^{\textnormal{sc}} \stackrel{\textnormal{pr}}{\longrightarrow} \bG$ denote the simply connected cover, and denote its (finite, central) kernel by $\boldsymbol{\Delta}$.  We then have a short exact sequence
$$1 \longrightarrow \boldsymbol{\Delta} \longrightarrow \bG^{\textnormal{sc}} \stackrel{\textnormal{pr}}{\longrightarrow} \bG \longrightarrow 1;$$
taking Galois-fixed points, we get
$$1 \longrightarrow \Delta \longrightarrow G^{\textnormal{sc}} \stackrel{\textnormal{pr}}{\longrightarrow} G \longrightarrow \textnormal{H}^1(F,\boldsymbol{\Delta}).$$
Since $F$ is of characteristic 0, the group $\textnormal{H}^1(F,\boldsymbol{\Delta})$ is finite (\cite[Thm. 6.14]{platonovrapinchuk}), and consequently $H := \textnormal{pr}(G^{\textnormal{sc}})$ is a finite index open subgroup of $G$.  In particular, a locally admissible, injective $G$-representation remains locally admissible and injective after restriction to $H$ (cf. \cite[Lem. 2.2, Prop. 2.3]{emertonpaskunas}).  Since $\textnormal{pr}$ induces an isomorphism between unipotent radicals of parabolic subgroups of $\bG^{\textnormal{sc}}$ and $\bG$, it therefore suffices to prove the claim with $G$ replaced by $H$.

Let $\pi$ denote an injective locally admissible $H$-representation over $R$.  Since $H \cong G^{\textnormal{sc}}/\Delta$, we may inflate $\pi$ to a locally admissible $G^{\textnormal{sc}}$-representation over $R$ (denoted by the same symbol).  Let 
$$\pi \longhookrightarrow A$$ 
denote an injective envelope of $\pi$ in the category $\mathfrak{Rep}^{\textnormal{ladm}}_R(G^{\textnormal{sc}})$.  Since $\Delta$ acts trivially on $\pi$, this injection factors as 
$$\pi \longhookrightarrow A^{\Delta} \longhookrightarrow A.$$
Note that we have an equivalence of categories between $\mathfrak{Rep}^{\textnormal{ladm}, \Delta = 1}_R(G^{\textnormal{sc}})$ (locally admissible representations of $G^{\textnormal{sc}}$ on which $\Delta$ acts trivially) and $\mathfrak{Rep}^{\textnormal{ladm}}_R(H)$.  Viewed in the latter category, the map $\pi \longhookrightarrow A^{\Delta}$ is an essential injection from an injective object.  Therefore, we must have $\pi \cong A^{\Delta}$ as $H$-representations (and thus also as $G^{\textnormal{sc}}$-representations).

By the proof of step \ref{eff-1}, we have that $A|_{\textnormal{SL}_2(\cO_F)}$ is injective in $\mathfrak{Rep}^{\infty}_R(\textnormal{SL}_2(\cO_F))$.  Since $\Delta N_0^- \subset \textnormal{SL}_2(\cO_F)$, \cite[Prop. 2.1.11]{emerton:ordII} implies that $A|_{\Delta N_0^-}$ is injective in $\mathfrak{Rep}^{\infty}_R(\Delta N_0^-)$.  Now, the functor 
\begin{eqnarray*}
\mathfrak{Rep}^{\infty}_R(\Delta N_0^-) & \longrightarrow & \mathfrak{Rep}^{\infty, \Delta = 1}_R(\Delta N_0^-)\\
 \tau & \longmapsto & \tau^\Delta
\end{eqnarray*} 
is right adjoint to the (exact) forgetful functor.  In particular, the above functor preserves injectives, so $A^\Delta|_{\Delta N_0^-}$ is an injective object of $\mathfrak{Rep}^{\infty, \Delta = 1}_R(\Delta N_0^-)$ (cf. \cite[pf. of Lem. 2.5]{emertonpaskunas}).  Using the equivalence of categories $\mathfrak{Rep}^{\infty, \Delta = 1}_R(\Delta N_0^-) \cong \mathfrak{Rep}^{\infty}_R(N_0^-)$ gives the result.
\end{proof}

\begin{step}\label{eff-3}
Conjecture \ref{sl2eff} is true when $\bG^{\textnormal{sc}} \cong \bS\bL_{2/F}$, and with \textnormal{``}$\mathfrak{Rep}^{\textnormal{ladm}}(G)$\textnormal{''} replaced by \textnormal{``}$\mathfrak{Rep}^{\textnormal{ladm}}_R(\widetilde{\mathcal{Z}}_0G^{\textnormal{der}})$\textnormal{''}.
\end{step}

\begin{proof}
We let $\cZ_1 \supset \cZ_2 \supset \ldots$ denote a decreasing sequence of open subgroups of $\widetilde{\cZ}_0$ such that $\bigcap_{i \geq 1} \cZ_i = \{1\}$ and $\cZ_i \cap G^{\textnormal{der}} = \{1\}$ for every $i \geq 1$ (this is possible since $\cZ \cap G^{\textnormal{der}}$ is finite).  Define $H_i := \cZ_i G^{\textnormal{der}}$.  Then each $H_i$ is a finite index open subgroup of $\widetilde{\cZ}_0G^{\textnormal{der}}$.

Let $\pi \in \mathfrak{Rep}^{\textnormal{ladm}}_R(\widetilde{\mathcal{Z}}_0G^{\textnormal{der}})$ be an injective object.  Since $H_i$ is of finite index in $\widetilde{\cZ}_0G^{\textnormal{der}}$, the representation $\pi|_{H_i}$ is locally admissible and injective.  As in the proof of part \ref{eff-2}, the functor 
\begin{eqnarray*}
\mathfrak{Rep}^{\textnormal{ladm}}_R(H_i) & \longrightarrow & \mathfrak{Rep}^{\textnormal{ladm}, \cZ_i = 1}_R(H_i)\\
\tau & \longmapsto & \tau^{\cZ_i}
\end{eqnarray*}
preserves injectives.  Therefore, $\pi^{\cZ_i}|_{H_i}$ is an injective object of $\mathfrak{Rep}^{\textnormal{ladm}, \cZ_i = 1}_R(H_i)$.  Since this category is equivalent to $\mathfrak{Rep}^{\textnormal{ladm}}_R(G^{\textnormal{der}})$, part \ref{eff-2} implies that $\pi^{\cZ_i}|_{N_0^-}$ is injective in $\mathfrak{Rep}^{\infty}_R(N_0^-)$.

By smoothness, we may write $\pi$ as an inductive limit
$$\pi = \varinjlim_{i}~ \pi^{\cZ_i}.$$
Restricting to $N_0^-$, we obtain
$$\pi|_{N_0^-} = \varinjlim_{i}~ (\pi^{\cZ_i}|_{N_0^-}).$$
Since $\pi^{\cZ_i}|_{N_0^-}$ is injective, \cite[Prop. 2.1.3]{emerton:ordII} implies that $\pi|_{N_0^-}$ is injective in $\mathfrak{Rep}^{\infty}_R(N_0^-)$.  
\end{proof}

\begin{step}\label{eff-4}
Conjecture \ref{sl2eff} is true when $\bG^{\textnormal{sc}} \cong \bS\bL_{2/F}$, and with \textnormal{``}$\mathfrak{Rep}^{\textnormal{ladm}}(G)$\textnormal{''} replaced by \textnormal{``}$\mathfrak{Rep}^{\textnormal{ladm}}_R(\mathcal{Z}G^{\textnormal{der}})$\textnormal{''}.
\end{step}

\begin{proof}
By induction on the rank of $\cZ/\widetilde{\cZ}_0$, it suffices to assume $\cZ/\widetilde{\cZ}_0 = \langle z_1 \rangle \cong \bbZ$.  We outline the proof, which closely follows the proof of \cite[Cor. 3.9]{emertonpaskunas}.

Let $\pi \in \mathfrak{Rep}^{\textnormal{ladm}}_R(\cZ G^{\textnormal{der}})$ be an injective object. We view $\pi$ as a $\cZ G^{\textnormal{der}}$-representation over $R[t^{\pm 1}]$, with $t$ acting via $z_1$.  We may then write
$$\pi = \bigoplus_{\fn \in \textnormal{m-Spec}(R[t^{\pm 1}])} \pi[\fn^{\infty}],$$
where $\pi[\fn^{\infty}] = \varinjlim_i \pi[\fn^i]$.  Let us write $\fn = (\fm, f)$, where $\fm$ is the maximal ideal of $R$ and $f \in R[t]$ is a monic polynomial.  Since $R$ is Artinian, we obtain 
$$\pi[\fn^{\infty}] = \pi[f^\infty] = \varinjlim_i~ \pi[f^i].$$

We claim that $\pi[f^i]|_{\widetilde{\cZ}_0 G^{\textnormal{der}}} \in \mathfrak{Rep}^{\textnormal{ladm}}_R(\widetilde{\mathcal{Z}}_0G^{\textnormal{der}})$ is injective.  To see this, note first that $\pi[f^i]$ is injective in the category $\mathfrak{Rep}^{\textnormal{ladm}, f^i = 0}_R(\cZ G^{\textnormal{der}})$ (locally admissible $\cZ G^{\textnormal{der}}$-representations which are annihilated by $f^i$), by an argument similar to the proofs of parts \ref{eff-2} and \ref{eff-3}.  Set $Q := R[t^{\pm 1}]/(f^i)$.  Then there is an equivalence of categories between $\mathfrak{Rep}^{\textnormal{ladm}, f^i = 0}_R(\cZ G^{\textnormal{der}})$ and $\mathfrak{Rep}^{\textnormal{ladm}}_{Q}(\widetilde{\cZ}_0 G^{\textnormal{der}})$.  Further, since $Q$ is free of finite rank over $R$, the functor 
\begin{eqnarray*}
\mathfrak{Rep}^{\textnormal{ladm}}_{R}(\widetilde{\cZ}_0 G^{\textnormal{der}}) & \longrightarrow & \mathfrak{Rep}^{\textnormal{ladm}}_{Q}(\widetilde{\cZ}_0 G^{\textnormal{der}}) \\
\tau & \longmapsto & Q \otimes_{R} \tau
\end{eqnarray*}
is exact.  Consequently, the forgetful right adjoint functor $\mathfrak{Rep}^{\textnormal{ladm}}_{Q}(\widetilde{\cZ}_0 G^{\textnormal{der}})  \longrightarrow  \mathfrak{Rep}^{\textnormal{ladm}}_{R}(\widetilde{\cZ}_0 G^{\textnormal{der}})$ preserves injectives.  This gives the claim.

Finally, since $\pi[f^i]|_{\widetilde{\cZ}_0 G^{\textnormal{der}}}$ is injective in $\mathfrak{Rep}^{\textnormal{ladm}}_R(\widetilde{\mathcal{Z}}_0G^{\textnormal{der}})$, part \ref{eff-3} implies that $\pi[f_i]|_{N_0^-}$ is injective in $\mathfrak{Rep}^{\infty}_R(N_0^-)$.  Since $\pi$ is an inductive limit of the $\pi[f^i]$, for various monic polynomials $f$ of $R[t]$, \cite[Prop. 2.1.3]{emerton:ordII} gives the result for $\pi|_{N_0^-}$.  
\end{proof}

\begin{step}\label{eff-5}
Conjecture \ref{sl2eff} is true when $\bG^{\textnormal{sc}} \cong \bS\bL_{2/F}$.
\end{step}

\begin{proof}
We have the following short exact sequence of groups
$$1 \longrightarrow \boldsymbol{\cZ} \cap \bG^{\textnormal{der}} \longrightarrow \boldsymbol{\cZ} \times \bG^{\textnormal{der}} \stackrel{\times}{\longrightarrow} \bG \longrightarrow 1,$$
where $\boldsymbol{\cZ} \cap \bG^{\textnormal{der}}$ is finite.  Taking Galois invariants, we obtain
$$1 \longrightarrow \cZ \cap G^{\textnormal{der}} \longrightarrow \cZ \times G^{\textnormal{der}} \stackrel{\times}{\longrightarrow} G \longrightarrow \textnormal{H}^1(F,\boldsymbol{\cZ} \cap \bG^{\textnormal{der}}).$$
Appealing to \cite[Thm. 6.14]{platonovrapinchuk} once again, we get that $\textnormal{H}^1(F,\boldsymbol{\cZ} \cap \bG^{\textnormal{der}})$ is finite, so that $\cZ G^{\textnormal{der}}$ is a finite index open subgroup of $G$.

Now let $\pi \in \mathfrak{Rep}^{\textnormal{ladm}}(G)$ be injective.  As before, the representation $\pi|_{\cZ G^{\textnormal{der}}} \in \mathfrak{Rep}^{\textnormal{ladm}}(\cZ G^{\textnormal{der}})$ is injective; since $N_0^- \subset \cZ G^{\textnormal{der}}$, part \ref{eff-4} implies the result.
\end{proof}

\begin{rmk}
The proofs of steps \ref{eff-2} through \ref{eff-5} above hold (with minor modifications) for arbitrary reductive groups.  In particular, Conjecture \ref{sl2eff} is implied by the following conjecture:

\begin{conj}
Suppose $\bG$ is semisimple and simply connected, and let $K$ denote a compact open subgroup of $G$ containing the center of $G$.  Let $\pi \in \mathfrak{Rep}^{\textnormal{ladm}}(G)$ be an injective object.  Then $\pi|_{K}$ is injective in $\mathfrak{Rep}^{\infty}(K)$.  Consequently we have $\textnormal{H}^i\ord^G_{P^-} \simeq \textnormal{R}^i\ord^G_{P^-}$ for $i > 0$.  
\end{conj}

\end{rmk}

\vspace{5pt}

\subsection{On Conjecture \ref{conjC}}

When $\bG$ has semisimple $F$-rank 0, we can refine the construction of Theorem \ref{emerton-conjthm}.

\begin{conj}\label{rk0}
Suppose $\pi \in \mathfrak{Rep}^{\textnormal{ladm}}(G)$ is injective, and suppose $I$ is an Iwahori subgroup of $G$.  Then the restriction $\pi|_I \in \mathfrak{Rep}^\infty(I)$ is injective.  
\end{conj}

\begin{thm}
\label{rk0-thm}
Suppose $\bG$ has semisimple $F$-rank $0$.  Then Conjecture \ref{rk0} is true.  
\end{thm}

The proof again proceeds in several steps.

\setcounter{step}{0}
\begin{step}\label{rk0-1}
Conjecture \ref{rk0} is true when $\bG$ has semisimple $F$-rank $0$ and with \textnormal{``}$\mathfrak{Rep}^{\textnormal{ladm}}(G)$\textnormal{''} is replaced by \textnormal{``}$\mathfrak{Rep}^{\textnormal{ladm}}(G_\aff)$\textnormal{''}.
\end{step}

\begin{proof}
Suppose $\pi \in \mathfrak{Rep}^{\textnormal{ladm}}(G_\aff)$ is injective.  Let $\sA$ denote the set of admissible subrepresentations of $\pi$; this set is filtered, and we have
$$\pi \cong \varinjlim_{\tau \in \sA} \tau.$$
By the proof of step \ref{step1} of Theorem \ref{emerton-conjthm}, there exists $A_\tau \in \mathfrak{Rep}^{\textnormal{adm}}(G_\aff)$ such that $\iota_\tau:\tau \longhookrightarrow A_\tau$ and $A_\tau \cong \inj_{G_\aff}(\tau)$.  Since $\pi$ is injective, we have the following commutative diagram in $\mathfrak{Rep}^{\textnormal{ladm}}(G_\aff)$:
\begin{center}
\begin{tikzcd}
0 \ar[r] & \tau \ar[r, "\iota_\tau"] \ar[d, hookrightarrow] & A_\tau \ar[dl, dashrightarrow, "\exists~\!\varphi"] \\
 & \pi & 
\end{tikzcd}
\end{center}
If $\ker(\varphi)$ was nonzero, then we would have $\ker(\varphi) \cap \iota_\tau(\tau) \neq 0$ by essentialness of $\iota_\tau:\tau \longhookrightarrow A_\tau$, contradicting the injectivity of $\tau \longhookrightarrow \pi$.  Therefore, if we denote by $\sI$ the set of admissible subrepresentations of $\pi$ which are injective in $\mathfrak{Rep}^{\infty}(G_\aff)$, we get 
$$\pi \cong \varinjlim_{A \in \sI} A.$$
Viewing this inductive limit in $\mathfrak{Rep}^{\infty}(G_\aff)$ and using \cite[Prop. 2.1.3]{emerton:ordII}, we get that $\pi$ is injective in $\mathfrak{Rep}^{\infty}(G_\aff)$.  
\end{proof}

\begin{step}\label{rk0-2}
Conjecture \ref{rk0} is true when $\bG$ has semisimple $F$-rank $0$ and with \textnormal{``}$\mathfrak{Rep}^{\textnormal{ladm}}(G)$\textnormal{''} is replaced by \textnormal{``}$\mathfrak{Rep}^{\textnormal{ladm}}(\widetilde{\cZ}_0 G_\aff)$\textnormal{''}.
\end{step}

\begin{proof}
Suppose $\pi \in \mathfrak{Rep}^{\textnormal{ladm}}(\widetilde{\cZ}_0 G_\aff)$ is injective.  Since $G_\aff$ is of finite index in $\widetilde{\cZ}_0 G_\aff$, \cite[Lem. 2.2, Prop. 2.3]{emertonpaskunas} implies that $\pi|_{G_\aff}$ is locally admissible and injective.  The claim then follows from step \ref{rk0-1}.
\end{proof}

\begin{step}\label{rk0-3}
Conjecture \ref{rk0} is true when $\bG$ has semisimple $F$-rank $0$ and with \textnormal{``}$\mathfrak{Rep}^{\textnormal{ladm}}(G)$\textnormal{''} is replaced by \textnormal{``}$\mathfrak{Rep}^{\textnormal{ladm}}(\cZ G_\aff)$\textnormal{''}.
\end{step}

\begin{proof}
By induction on the rank of $\cZ/\widetilde{\cZ}_0$, it suffices to assume $\cZ/\widetilde{\cZ}_0 = \langle z_1 \rangle \cong \bbZ$.  The argument then follows exactly as in the proof of \cite[Cor. 3.9]{emertonpaskunas}, using step \ref{rk0-2} (see also the proof of step \ref{eff-4} in the previous subsection).
\end{proof}

\begin{step}\label{rk0-4}
Conjecture \ref{rk0} is true when $\bG$ has semisimple $F$-rank $0$.
\end{step}

\begin{proof}
Suppose $\pi \in \mathfrak{Rep}^{\textnormal{ladm}}(G)$ is injective.  Since $\cZ G_\aff$ is of finite index in $G$, the restriction $\pi|_{\cZ G_\aff}$ is locally admissible and injective by \cite[Lem. 2.2, Prop. 2.3]{emertonpaskunas}.  The claim then follows from step \ref{rk0-3}.
\end{proof}

\bibliographystyle{amsalpha}
\bibliography{refs}

\newcommand{\etalchar}[1]{$^{#1}$}
\providecommand{\bysame}{\leavevmode\hbox to3em{\hrulefill}\thinspace}
\providecommand{\MR}{\relax\ifhmode\unskip\space\fi MR }
\providecommand{\MRhref}[2]{%
  \href{http://www.ams.org/mathscinet-getitem?mr=#1}{#2}
}
\providecommand{\href}[2]{#2}
\begin{thebibliography}{DdSMS99}

\bibitem[Abd14]{abdellatif:sl2qp}
Ramla Abdellatif, \emph{Classification des repr\'esentations modulo {$p$} de
  {${\rm SL}(2,F)$}}, Bull. Soc. Math. France \textbf{142} (2014), no.~3,
  537--589. \MR{3295722}

\bibitem[Abe17]{abe:extensions}
Noriyuki Abe, \emph{{Extension between simple modules of pro-$p$-Iwahori Hecke
  algebras}}, ArXiv e-prints (2017).

\bibitem[Abe19a]{abe:involutions}
\bysame, \emph{Involutions on pro-{$p$}-{I}wahori {H}ecke algebras}, Represent.
  Theory \textbf{23} (2019), 57--87. \MR{3902325}

\bibitem[Abe19b]{abe:inductions}
\bysame, \emph{Parabolic inductions for pro-{$p$}-{I}wahori {H}ecke algebras},
  Adv. Math. \textbf{355} (2019), 106776, 63. \MR{3996728}

\bibitem[AHHV17]{ahhv}
N.~Abe, G.~Henniart, F.~Herzig, and M.-F. Vign\'eras, \emph{A classification of
  irreducible admissible {${\rm mod}\, p$} representations of {$p$}-adic
  reductive groups}, J. Amer. Math. Soc. \textbf{30} (2017), no.~2, 495--559.
  \MR{3600042}

\bibitem[AHV18]{ahv:pro-p}
N.~Abe, G.~Henniart, and M.-F. Vign\'{e}ras, \emph{On pro-{$p$}-{I}wahori
  invariants of {$R$}-representations of reductive {$p$}-adic groups},
  Represent. Theory \textbf{22} (2018), 119--159. \MR{3864023}

\bibitem[AHV19]{AHV}
\bysame, \emph{Modulo {$p$} representations of reductive {$p$}-adic groups:
  functorial properties}, Trans. Amer. Math. Soc. \textbf{371} (2019), no.~12,
  8297--8337. \MR{3955548}

\bibitem[Ast08]{Ast319}
\emph{Repr\'{e}sentations $p$-adiques de groupes $p$-adiques {I}:
  {R}epr\'{e}sentations galoisiennes et $(\phi,{\Gamma})$-modules},
  Ast\'{e}risque \textbf{319} (2008).

\bibitem[Ast10a]{Ast330}
\emph{Repr\'{e}sentations $p$-adiques de groupes $p$-adiques {II}:
  {R}epr\'{e}sentations de $\mathbf{GL}2(\mathbf{Q}_p)$ et
  $(\phi,{\Gamma})$-modules}, Ast\'{e}risque \textbf{330} (2010).

\bibitem[Ast10b]{Ast331}
\emph{Repr\'{e}sentations $p$-adiques de groupes $p$-adiques {III}:
  {M}\'{e}thodes globales et g\'{e}om\'{e}triques}, Ast\'{e}risque \textbf{331}
  (2010).

\bibitem[Bou06]{bourbaki:ac89}
N.~Bourbaki, \emph{\'{E}l\'{e}ments de math\'{e}matique. {A}lg\`ebre
  commutative. {C}hapitres 8 et 9}, Springer, Berlin, 2006, Reprint of the 1983
  original. \MR{2284892}

\bibitem[BP12]{breuilpaskunas}
Christophe Breuil and Vytautas Pa{\v{s}}k{\={u}}nas, \emph{Towards a modulo
  {$p$} {L}anglands correspondence for {${\rm GL}_2$}}, Mem. Amer. Math. Soc.
  \textbf{216} (2012), no.~1016, vi+114. \MR{2931521}

\bibitem[Bre03]{breuil:gl2qp}
Christophe Breuil, \emph{Sur quelques repr\'esentations modulaires et
  {$p$}-adiques de {${\rm GL}_2(\bold Q_p)$}. {I}}, Compositio Math.
  \textbf{138} (2003), no.~2, 165--188. \MR{2018825}

\bibitem[Bru66]{brumer}
Armand Brumer, \emph{Pseudocompact algebras, profinite groups and class
  formations}, J. Algebra \textbf{4} (1966), 442--470. \MR{0202790}

\bibitem[CEG{\etalchar{+}}16]{ceggps}
Ana Caraiani, Matthew Emerton, Toby Gee, David Geraghty, Vytautas
  Pa\v{s}k\={u}nas, and Sug~Woo Shin, \emph{Patching and the {$p$}-adic local
  {L}anglands correspondence}, Camb. J. Math. \textbf{4} (2016), no.~2,
  197--287. \MR{3529394}

\bibitem[Col10]{colmez:gl2}
Pierre Colmez, \emph{Repr\'{e}sentations de {${\rm GL}_2(\bold Q_p)$} et
  {$(\phi,\Gamma)$}-modules}, Ast\'{e}risque (2010), no.~330, 281--509.
  \MR{2642409}

\bibitem[DdSMS99]{DDMS}
J.~D. Dixon, M.~P.~F. du~Sautoy, A.~Mann, and D.~Segal, \emph{Analytic
  pro-{$p$} groups}, second ed., Cambridge Studies in Advanced Mathematics,
  vol.~61, Cambridge University Press, Cambridge, 1999. \MR{1720368}

\bibitem[Eme10a]{emerton:ordI}
Matthew Emerton, \emph{Ordinary parts of admissible representations of
  {$p$}-adic reductive groups {I}. {D}efinition and first properties},
  Ast\'{e}risque (2010), no.~331, 355--402. \MR{2667882}

\bibitem[Eme10b]{emerton:ordII}
\bysame, \emph{Ordinary parts of admissible representations of {$p$}-adic
  reductive groups {II}. {D}erived functors}, Ast\'erisque (2010), no.~331,
  403--459. \MR{2667883}

\bibitem[EP10]{emertonpaskunas}
Matthew Emerton and Vytautas Pa\v{s}k\={u}nas, \emph{On the effaceability of
  certain {$\delta$}-functors}, Ast\'erisque (2010), no.~331, 461--469.
  \MR{2667892}

\bibitem[Har16]{harris:specs}
Michael Harris, \emph{Speculations on the {${\rm mod}\, p$} representation
  theory of {$p$}-adic groups}, Ann. Fac. Sci. Toulouse Math. (6) \textbf{25}
  (2016), no.~2-3, 403--418. \MR{3530163}

\bibitem[Hau18]{hauseux:parabind}
Julien Hauseux, \emph{Parabolic induction and extensions}, Algebra Number
  Theory \textbf{12} (2018), no.~4, 779--831. \MR{3830204}

\bibitem[Koh17]{kohlhaase:duality}
Jan Kohlhaase, \emph{Smooth duality in natural characteristic}, Adv. Math.
  \textbf{317} (2017), 1--49. \MR{3682662}

\bibitem[Koz18]{koziol:h1triv}
Karol Kozio\l, \emph{Hecke module structure on first and top
  pro-{$p$}-{I}wahori cohomology}, Acta Arith. \textbf{186} (2018), no.~4,
  349--376. \MR{3879398}

\bibitem[Koz19]{koziol:h1ps}
\bysame, \emph{The first pro-{$p$}-{I}wahori cohomology of mod-{$p$} principal
  series for {$p$}-adic {${\rm GL}_n$}}, Trans. Amer. Math. Soc. \textbf{372}
  (2019), no.~2, 1237--1288. \MR{3968802}

\bibitem[KS06]{kashiwaraschapira}
Masaki Kashiwara and Pierre Schapira, \emph{Categories and sheaves},
  Grundlehren der Mathematischen Wissenschaften [Fundamental Principles of
  Mathematical Sciences], vol. 332, Springer-Verlag, Berlin, 2006. \MR{2182076}

\bibitem[KX15]{koziolxu}
Karol Kozio{\l} and Peng Xu, \emph{Hecke modules and supersingular
  representations of {${\rm U}(2,1)$}}, Represent. Theory \textbf{19} (2015),
  56--93. \MR{3321473}

\bibitem[Laz65]{lazard}
Michel Lazard, \emph{Groupes analytiques {$p$}-adiques}, Inst. Hautes
  \'{E}tudes Sci. Publ. Math. (1965), no.~26, 389--603. \MR{0209286}

\bibitem[Nad19]{nadimpalli}
Santosh Nadimpalli, \emph{On extensions of supersingular representations of
  {${\rm SL}_2(\Bbb{Q}_p)$}}, J. Number Theory \textbf{199} (2019), 150--167.
  \MR{3926192}

\bibitem[NSW08]{nsw:coh}
J\"urgen Neukirch, Alexander Schmidt, and Kay Wingberg, \emph{Cohomology of
  number fields}, second ed., Grundlehren der Mathematischen Wissenschaften
  [Fundamental Principles of Mathematical Sciences], vol. 323, Springer-Verlag,
  Berlin, 2008. \MR{2392026}

\bibitem[OS19]{ollivierschneider:ext}
Rachel Ollivier and Peter Schneider, \emph{The modular pro-{$p$}
  {I}wahori-{H}ecke {E}xt-algebra}, Representations of reductive groups, Proc.
  Sympos. Pure Math., vol. 101, Amer. Math. Soc., Providence, RI, 2019,
  pp.~255--308. \MR{3930021}

\bibitem[OV18]{olliviervigneras}
Rachel Ollivier and Marie-France Vign\'{e}ras, \emph{Parabolic induction in
  characteristic {$p$}}, Selecta Math. (N.S.) \textbf{24} (2018), no.~5,
  3973--4039. \MR{3874689}

\bibitem[Pa{\v{s}}04]{paskunas:diags}
Vytautas Pa{\v{s}}k{\=u}nas, \emph{Coefficient systems and supersingular
  representations of {${\rm GL}_2(F)$}}, M\'{e}m. Soc. Math. Fr. (N.S.) (2004),
  no.~99, vi+84. \MR{2128381}

\bibitem[Pa{\v{s}}10]{paskunas:exts}
\bysame, \emph{Extensions for supersingular representations of {${\rm
  GL}_2(\Bbb Q_p)$}}, Ast\'erisque (2010), no.~331, 317--353. \MR{2667891}

\bibitem[Pa{\v{s}}13]{paskunas:montrealfunctor}
\bysame, \emph{The image of {C}olmez's {M}ontreal functor}, Publ. Math. Inst.
  Hautes \'Etudes Sci. \textbf{118} (2013), 1--191. \MR{3150248}

\bibitem[PR94]{platonovrapinchuk}
Vladimir Platonov and Andrei Rapinchuk, \emph{Algebraic groups and number
  theory}, Pure and Applied Mathematics, vol. 139, Academic Press, Inc.,
  Boston, MA, 1994, Translated from the 1991 Russian original by Rachel Rowen.
  \MR{1278263}

\bibitem[Sch11]{schneider:padicliegroups}
Peter Schneider, \emph{{$p$}-adic {L}ie groups}, Grundlehren der Mathematischen
  Wissenschaften [Fundamental Principles of Mathematical Sciences], vol. 344,
  Springer, Heidelberg, 2011. \MR{2810332}

\bibitem[Sch15]{schneider:dga}
\bysame, \emph{Smooth representations and {H}ecke modules in characteristic p},
  Pacific J. Math. \textbf{279} (2015), no.~1-2, 447--464. \MR{3437786}

\bibitem[Sch18]{scholze:LT}
Peter Scholze, \emph{On the {$p$}-adic cohomology of the {L}ubin-{T}ate tower},
  Ann. Sci. \'{E}c. Norm. Sup\'{e}r. (4) \textbf{51} (2018), no.~4, 811--863,
  With an appendix by Michael Rapoport. \MR{3861564}

\bibitem[Ser65]{serre:cohdim}
Jean-Pierre Serre, \emph{Sur la dimension cohomologique des groupes profinis},
  Topology \textbf{3} (1965), 413--420. \MR{0180619}

\bibitem[Ser02]{serre:galoiscoh}
\bysame, \emph{Galois cohomology}, english ed., Springer Monographs in
  Mathematics, Springer-Verlag, Berlin, 2002, Translated from the French by
  Patrick Ion and revised by the author. \MR{1867431}

\bibitem[Ser03]{serre:trees}
\bysame, \emph{Trees}, Springer Monographs in Mathematics, Springer-Verlag,
  Berlin, 2003, Translated from the French original by John Stillwell,
  Corrected 2nd printing of the 1980 English translation. \MR{1954121}

\bibitem[SW00]{symondsweigel}
Peter Symonds and Thomas Weigel, \emph{Cohomology of {$p$}-adic analytic
  groups}, New horizons in pro-{$p$} groups, Progr. Math., vol. 184,
  Birkh\"auser Boston, Boston, MA, 2000, pp.~349--410. \MR{1765127}

\bibitem[Tot99]{totaro:eulerchar}
Burt Totaro, \emph{Euler characteristics for {$p$}-adic {L}ie groups}, Inst.
  Hautes \'{E}tudes Sci. Publ. Math. (1999), no.~90, 169--225 (2001).
  \MR{1813226}

\bibitem[Vig04]{vigneras:gl2}
Marie-France Vign{\'e}ras, \emph{Representations modulo {$p$} of the {$p$}-adic
  group {${\rm GL}(2,F)$}}, Compos. Math. \textbf{140} (2004), no.~2, 333--358.
  \MR{2027193 (2004m:22028)}

\bibitem[Vig15]{vigneras:hecke5}
\bysame, \emph{The pro-p {I}wahori {H}ecke algebra of a reductive p-adic group,
  {V} (parabolic induction)}, Pacific J. Math. \textbf{279} (2015), no.~1-2,
  499--529. \MR{3437789}

\bibitem[Vig16a]{vigneras:hecke1}
\bysame, \emph{The pro-{$p$}-{I}wahori {H}ecke algebra of a reductive
  {$p$}-adic group {I}}, Compos. Math. \textbf{152} (2016), no.~4, 693--753.
  \MR{3484112}

\bibitem[Vig16b]{vigneras:rightadj}
\bysame, \emph{The right adjoint of the parabolic induction}, Arbeitstagung
  {B}onn 2013, Progr. Math., vol. 319, Birkh\"{a}user/Springer, Cham, 2016,
  pp.~405--425. \MR{3618059}

\end{thebibliography}

\end{document}